\documentclass[12pt]{article}
\usepackage{mathrsfs}
\usepackage{amsfonts}
\usepackage{amssymb,bbm}  
\usepackage{amsmath,amscd,graphicx,amsthm}
\usepackage{subeqnarray}
\usepackage{cases}
\usepackage{paralist}  
\usepackage[colorlinks,linkcolor=blue,anchorcolor=blue,citecolor=blue]{hyperref}
       \newtheorem{lemma}{\bf Lemma}[section]
       \newtheorem{theorem}[lemma]{\bf Theorem}
       \newtheorem{proposition}[lemma]{\bf Proposition}
       \newtheorem{corollary}[lemma]{\bf Corollary}

       \numberwithin{equation}{section}

\usepackage{mathrsfs}
\usepackage{amsfonts}
\usepackage{amssymb,bbm}  
\usepackage{amsmath,amscd}
\usepackage{subeqnarray}
\usepackage{mathrsfs,amsfonts,amssymb,bbm,amsmath,amscd,subeqnarray,cases,bm,wasysym}
\newcommand{\md}{n} 
\newcommand{\rmd}{w} 
\newcommand{\dl}{j} 
\newcommand{\dws}{\phi} 
\newcommand{\smd}{\tilde{n}} 
\newcommand{\rsmd}{\tilde{w}} 
\newcommand{\sdl}{\tilde{j}} 
\newcommand{\sdws}{\tilde{\phi}} 
\newcommand{\prmd}{\psi} 
\newcommand{\pdl}{\eta} 
\newcommand{\pdws}{\sigma} 
\newcommand{\pc}{\varepsilon} 
\newcommand{\pcmd}[1]{n^{#1}} 
\newcommand{\pcdl}[1]{j^{#1}} 
\newcommand{\pcdws}[1]{\phi^{#1}} 
\newcommand{\pcsmd}[1]{\tilde{n}^{#1}} 
\newcommand{\pcszmd}[1]{\tilde{z}^{#1}} 
\newcommand{\pcsdl}[1]{\tilde{j}^{#1}} 
\newcommand{\pcsdws}[1]{\tilde{\phi}^{#1}} 
\newcommand{\pcpmd}{\mathcal{N}^{\pc}} 
\newcommand{\pcpdl}{\mathcal{J}^{\pc}} 
\newcommand{\pcpdws}{\varPhi^{\pc}} 
\newcommand{\pcpszmd}{\tilde{\mathcal{Z}}^{\pc}} 
\newcommand{\pcpsdl}{\tilde{\mathcal{J}}^{\pc}} 
\newcommand{\pcpsdws}{\tilde{\varPhi}^{\pc}} 
\newcommand{\pcsS}[1]{\tilde{S}^{#1}} 
\newcommand{\pcS}[1]{S^{#1}} 
\newcommand{\pcsK}[1]{\tilde{K}^{#1}} 
\newcommand{\pcR}[1]{\tilde{\mathcal{R}}^{#1}} 
\newcommand{\pcr}[1]{\tilde{\mathcal{Q}}^{#1}} 
\newcommand{\pcN}{N_\pc(T)} 
\newcommand{\pcn}{n_\pc(t)} 
\newcommand{\pd}[2]{\partial_{#1}^{#2}} 
\makeatletter 
\newcommand{\Rmnum}[1]{\expandafter\@slowromancap\romannumeral #1@} 
\makeatother 
\newcommand{\wnd}{\theta} 
\newcommand{\swnd}{\tilde{\theta}} 
\newcommand{\pwnd}{\chi} 
\newcommand{\pcwnd}[1]{\theta^{#1}} 
\newcommand{\pcswnd}[1]{\tilde{\theta}^{#1}} 
\newcommand{\pcpswnd}{\tilde{\varTheta}^{\pc}} 
\newcommand{\pcpwnd}{\varTheta^{\pc}} 

\begin{document}

\title{{Stability and semi-classical limit in a semiconductor full quantum hydrodynamic model with non-flat doping profile}
\footnotetext{\small $*$ : The corresponding author\\
\small E-mail: huhf836@nenu.edu.cn, zhangkj201@nenu.edu.cn}}

\author{{Haifeng Hu$^{1}$ and Kaijun Zhang$^{2,*}$}\\[2mm]
\normalsize\it 1. Center for Partial Differential Equations, East China Normal University, \\
\normalsize\it Minhang, Shanghai 200241, P.R. China \\
\normalsize\it 2. School of Mathematics and Statistics, Northeast Normal University, \\
\normalsize\it Changchun 130024, P.R. China
}

\date{}

\maketitle

\noindent\textbf{Abstract.} {\small We present the new results on stability and semi-classical limit in a semiconductor full quantum hydrodynamic (FQHD) model with non-flat doping profile. The FQHD model can be used to analyze the thermal and quantum influences on the transport of carriers (electrons or holes) in semiconductor device. Inspired by the physical motivation, we consider the initial-boundary value problem of this model over the one-dimensional bounded domain and adopt the ohmic contact boundary condition and the vanishing bohmenian-type boundary condition. Firstly, the existence and asymptotic stability of a stationary solution are proved by Leray-Schauder fixed-point theorem, Schauder fixed-point theorem and the refined energy method. Secondly, we show the semi-classical limit results for both stationary solutions and global solutions by the elaborate energy estimates and the compactness argument. The strong convergence rates of the related asymptotic sequences of solutions are also obtained.}


\section{Introduction}\label{Sect.1}
In the mathematical modeling of the nano-size semiconductor devices (e.g. HEMTs, MOSFETs, RTDs and superlattice devices ), the quantum effects (like particle tunneling through potential barriers and particle buildup in quantum wells) take place and can not be simulated by classical hydrodynamic models. Therefore, the quantum hydrodynamical (QHD) equations are important and dominative in the description of the motion of electrons or holes transport under the self-consistent electric field.

The QHD conservation laws have the same form as the
classical hydrodynamic equations (for simplicity, we treat the flow of electrons in the self-consistent electric field for unipolar devices):
\begin{subequations}\label{qhdcl}
\begin{numcases}{}
\pd{t}{}n+\pd{x_k}{}j_k=0, \label{qhdcl1}\\
\pd{t}{}j_l+\pd{x_k}{}(u_kj_l-P_{kl})=n\pd{x_l}{}\phi-\frac{j_l}{\tau_m},\quad l=1,2,3, \label{qhdcl2}\\
\pd{t}{}e+\pd{x_k}{}(u_ke-u_lP_{kl}+q_k)=j_k\pd{x_k}{}\phi+C_e, \label{qhdcl3}\\
\lambda^2\Delta\phi=n-D(\bm{x}), \label{qhdcl4}
\end{numcases}
\end{subequations}
where $n>0$ is the electron density, $\bm{u}=(u_1,u_2,u_3)$ is the velocity, $\bm{j}=(j_1,j_2,j_3)$ is the momentum density, $\bm{P}=(P_{kl})$ is the stress tensor, $\phi$ is the self-consistent electrostatic potential, $e$ is the energy density, $\bm{q}=(q_1,q_2,q_3)$ is the heat flux. Indices $k,l$ equal $1,2,3$, and repeated indices are summed over using the Einstein convention. Equation \eqref{qhdcl1} expresses conservation of electron number, \eqref{qhdcl2} expresses conservation of momentum, and \eqref{qhdcl3} expresses conservation of energy. The last terms in \eqref{qhdcl2} and \eqref{qhdcl3} represent electron scattering (the collision terms may include the effects of electron-phonon and electron-impurity collisions, intervalley and interband scattering), which is modeled by the standard relaxation time approximation with momentum and energy relaxation times $\tau_m>0$ and $\tau_e>0$. The energy relaxation term $C_e$ is given by
\begin{equation*}
C_e=-\frac{1}{\tau_e}\Bigg(\frac{1}{2}n|\bm{u}|^2+\frac{3}{2}n(\theta-\theta_{L})\Bigg),
\end{equation*}
where $\theta>0$ is the electron temperature and $\theta_{L}>0$ is the temperature of the semiconductor lattice in energy units. The transport equations \eqref{qhdcl1}$\sim$\eqref{qhdcl3} are coupled to Poisson's equation \eqref{qhdcl4} for the self-consistent electrostatic potential, where $\lambda>0$ is the Debye length, $D=N_d-N_a$ is the doping profile, $N_d>0$ is the density of donors, and $N_a>0$ is the density of acceptors.

The QHD equations \eqref{qhdcl1}$\sim$\eqref{qhdcl3} are derived as a set of nonlinear conservation laws by a moment expansion of the Wigner-Boltzmann equation \cite{W32} and an expansion of the thermal equilibrium Wigner distribution function to $O(\pc^2)$, where $\pc>0$ is the scaled Planck constant. However, to close the moment expansion at the first three moments, we must define, for example, $\bm{j}$, $\bm{P}$, $e$ and $\bm{q}$ in terms of $n$, $\bm{u}$ and $\theta$. According to the closure assumption \cite{G94}, up to order
$O(\pc^2)$, we define the momentum density $\bm{j}$, the stress tensor $\bm{P}=(P_{kl})$, the energy density $e$ and the heat flux $\bm{q}$ as follows:
\begin{gather*}
\bm{j}=n\bm{u},\qquad P_{kl}=-n\theta\delta_{kl}+\frac{\pc^2}{2}n\pd{x_k}{}\pd{x_l}{}\ln n,\\
e=\frac{3}{2}n\theta+\frac{1}{2}n|\bm{u}|^2-\frac{\pc^2}{4}n\Delta\ln n,\qquad \bm{q}=-\kappa\nabla\theta-\frac{3\pc^2}{4}n\Delta\bm{u},
\end{gather*}
with the Kronecker symbol $\delta_{kl}$ and the heat conductivity $\kappa>0$. The quantum correction to the stress tensor was first stated in the semiconductor context by Ancona and Iafrate \cite{AI89} and Ancona and Tiersten \cite{AT87}. Since
\begin{equation*}
\frac{\pc^2}{2}\mathrm{div}(n(\nabla\otimes\nabla)\ln n)=\pc^2n\nabla\Bigg(\frac{\Delta\sqrt{n}}{\sqrt{n}}\Bigg),
\end{equation*}
it can be interpreted as a force including the Bohm potential $\pc^2\Delta\sqrt{n}/\sqrt{n}$ \cite{FZ93}. The quantum correction to the energy density was first derived by Wigner \cite{W32}. The heat conduction term consists of a classical Fourier law $-\kappa\nabla\theta$ plus a new quantum contribution $-3\pc^2n\Delta\bm{u}/4$ which can be interpreted as a dispersive heat flux \cite{G95,JMM06}. For details on the more general quantum models for semiconductor devices, one can refer to the references \cite{J01,MRS90,ZH16book}.

Interestingly, most quantum terms cancel out in the energy equation \eqref{qhdcl3}. In fact, by substituting the above expressions for $C_e$, $\bm{j}$, $\bm{P}$, $e$ and $\bm{q}$ into \eqref{qhdcl}, a computation yields the multi-dimensional full quantum hydrodynamic (FQHD) model for semiconductors as follows.
\begin{subequations}\label{fqhd}
\begin{numcases}{}
n_t+\mathrm{div}(n\bm{u})=0, \label{fqhd1}\\
(n\bm{u})_t+\mathrm{div}(n\bm{u}\otimes\bm{u})+\nabla(n\theta)-\pc^2n\nabla\Bigg(\frac{\Delta\sqrt{n}}{\sqrt{n}}\Bigg)=n\nabla\phi-\frac{n\bm{u}}{\tau_m}, \label{fqhd2}\\
n\theta_t+n\bm{u}\cdot\nabla\theta+\frac{2}{3}n\theta\mathrm{div}\bm{u}-\frac{2}{3}\mathrm{div}(\kappa\nabla\theta)\notag\\
\qquad\qquad\qquad\qquad\qquad\qquad\quad\,-\frac{\pc^2}{3}\mathrm{div}(n\Delta\bm{u})=\frac{2\tau_e-\tau_m}{3\tau_m\tau_e}n|\bm{u}|^2-\frac{n(\theta-\theta_{L})}{\tau_e}, \label{fqhd3}\\
\lambda^2\Delta\phi=n-D(\bm{x}). \label{fqhd4}
\end{numcases}
\end{subequations}
Comparing with the classical full hydrodynamic (FHD) model, the new feature of the FQHD model is the Bohm potential term
\begin{equation*}
-\pc^2n\nabla\Bigg(\frac{\Delta\sqrt{n}}{\sqrt{n}}\Bigg)
\end{equation*}
in the momentum equation \eqref{fqhd2} and the dispersive velocity term
\begin{equation*}
-\frac{\pc^2}{3}\mathrm{div}(n\Delta\bm{u})
\end{equation*}
in the energy equation \eqref{fqhd3}. Both of them are called quantum correction terms (or dispersive terms) and belong to the third-order derivative terms of the system \eqref{fqhd}.

Recently, the study concerning the semiconductor quantum models and the related quantum systems has become popular. J\"ungel and Li \cite{JL04,JL04-1} investigated the one-dimensional unipolar isentropic QHD model with the Dirichlet-Neumann boundary condition and the flat doping profile. The authors proved the existence, uniqueness and exponential stability of the subsonic stationary solution for the quite general pressure-density function. Nishibata and Suzuki \cite{NS08} reconsidered this QHD model with isothermal simplification and the vanishing bohmenian-type boundary condition. The authors generalized J\"ungel and Li's results to the non-flat doping profile case and also discussed the semi-classical limit for both stationary and global solutions. Hu, Mei and Zhang \cite{HMZ16} generalized Nishibata and Suzuki's results to the bipolar case with non-constant but flat doping profile. Huang, Li and Matsumura \cite{HLM06} proved the existence, exponential stability and semi-classical limit of stationary solution of Cauchy problem for the one-dimensional  isentropic unipolar QHD model. Li and Yong \cite{LY17} studied the nonlinear diffusion phenomena on the Cauchy problem of the one-dimensional isentropic bipolar QHD model. The authors proved the algebraic stability of the diffusion waves.

In multi-dimensional case, J\"ungel \cite{J98} first considered the unipolar stationary isothermal and isentropic QHD model for potential flows on a bounded domain. The existence of solutions was proved under the assumption that the electric energy was small compared to the thermal energy, where Dirichlet boundary conditions were addressed. This result was then generalized to bipolar case by Liang and Zhang \cite{LZ07}. Unterreiter \cite{U97} proved the existence of the thermal equilibrium solution of the bipolar isentropic QHD model confined to a bounded domain by variational analysis, and the semi-classical limit is carried out recovering the minimizer of the limiting functional. This result recently was developed by Di Michele, Mei, Rubino and Sampalmieri \cite{MMRS17} to a new model of the bipolar isentropic hybrid quantum hydrodynamics. Regarding the unipolar QHD model for irrotational fluid in spatial periodic domain, the global existence of the dynamic solutions and the exponential convergence to their equilibria were artfully proved by Li and Marcati in \cite{LM04}. Remarkably, the weak solutions with large initial data for the quantum hydrodynamic system in multiple dimensions were further obtained by Antonelli and Marcati in \cite{AM09,AM12}. Li, Zhang and Zhang \cite{LZZ08} investigated the large-time behavior of solutions to the initial value problem of the isentropic QHD model in the whole space $\mathbb{R}^3$ and obtained the algebraic time-decay rate, and further showed in \cite{ZLZ08} the semi-classical and relaxation limits of global solutions. Recently, Pu and Guo \cite{PG16} studied the Cauchy problem of a quantum hydrodynamic equations with viscosity and heat conduction in the whole space $\mathbb{R}^3$. The global existence around a constant steady-state and semi-classical limit of the global solutions were shown by the energy method. This result was developed by Pu and Xu \cite{PX17}, the authors obtained the optimal convergence rates to the constant equilibrium solution by the pure energy method and negative Sobolev space estimates.

However, all of these research results more or less have some limitations from both physical and mathematical points of view. Actually, in practical applications, the semiconductor quantum models should be treated under the following physically motivated settings which make the mathematical analysis more difficult:
\begin{itemize}
\item[(1)] The model system should be considered on a bounded domain $\Omega$ and be supplemented by physical boundary conditions.
\item[(2)] In realistic semiconductor devices, the doping profile will be a non-flat function of the spatial variable. For instance, it has two steep slops in $n^+-n-n^+$ diodes \cite{G94}. Therefore, we should assume the continuity and positivity only to cover the actual devices. Namely,
\begin{equation}\label{nonflat}
D\in C(\overline{\Omega}),\qquad \inf_{x\in\overline{\Omega}}D(x)>0.
\end{equation}
\item[(3)] To study the FQHD model which includes the quantum corrected energy equation is essentially significant for understanding the quantum transport of the hot carriers in semiconductor devices. The new feature is that one has to investigate both thermal and quantum effects in the more complex model system than the various simplified models.
\end{itemize}

In this paper, under the above physical principle settings (1)$\sim$(3), we will study the FQHD model \eqref{fqhd} in one space dimension with $\tau_m=\tau_e=\kappa=\lambda=1$. Namely, we consider
\begin{subequations}\label{1dfqhd}
\begin{numcases}{}
\md_t+\dl_x=0, \label{1dfqhd1}\\
\dl_t+\Bigg(\frac{\dl^2}{\md}+\md\wnd\Bigg)_x-\pc^2\md\Bigg[\frac{\big(\sqrt{\md}\big)_{xx}}{\sqrt{\md}}\Bigg]_x=\md\dws_x-\dl, \label{1dfqhd2}\\
\md\wnd_t+\dl\wnd_x+\frac{2}{3}\md\wnd\Bigg(\frac{\dl}{\md}\Bigg)_x-\frac{2}{3}\wnd_{xx}-\frac{\pc^2}{3}\Bigg[n\Bigg(\frac{j}{n}\Bigg)_{xx}\Bigg]_x=\frac{1}{3}\frac{\dl^2}{\md}-\md(\wnd-\theta_{L}), \label{1dfqhd3}\\
\dws_{xx}=\md-D(x), \qquad\forall t>0,\ \forall x\in\Omega:=(0,1),\label{1dfqhd4}
\end{numcases}
\end{subequations}
with the initial condition
\begin{equation}\label{ic}
(\md,\dl,\wnd)(0,x)=(\md_0,\dl_0,\wnd_0)(x),
\end{equation}
and the boundary conditions
\begin{subequations}\label{bc}
\begin{gather}
\md(t,0)=\md_{l},\qquad \md(t,1)=\md_{r},\label{bc-a}\\
\big(\sqrt{\md}\big)_{xx}(t,0)=\big(\sqrt{\md}\big)_{xx}(t,1)=0,\label{bc-b}\\
\wnd(t,0)=\wnd_{l},\qquad \wnd(t,1)=\wnd_{r},\label{bc-c}\\
\dws(t,0)=0,\qquad \dws(t,1)=\phi_r,\label{bc-d}
\end{gather}
\end{subequations}
where the boundary data $\md_{l},\md_{r},\wnd_{l},\wnd_{r}$ and $\phi_r$ are positive constants. The vanishing bohmenian-type boundary condition \eqref{bc-b} means that the quantum Bohm potential vanishes on the boundary, which is derived in \cite{G94,P99} and is also physically reasonable. The other boundary conditions in \eqref{bc} are called ohmic contact boundary condition. In order to establish the existence of a classical solution, we further assume the initial data $(\md_0,\dl_0,\wnd_0)$ is compatible with the boundary data \eqref{bc-a}$\sim$\eqref{bc-c} and $\md_t(t,0)=\md_t(t,1)=0$, namely,
\begin{gather}
\md_0(0)=\md_l,\quad \md_0(1)=\md_r,\quad \wnd_0(0)=\theta_l,\quad \wnd_0(1)=\theta_r,\notag\\
\dl_{0x}(0)=\dl_{0x}(1)=\big(\sqrt{\md_0}\big)_{xx}(0)=\big(\sqrt{\md_0}\big)_{xx}(1)=0.\label{compatibility}
\end{gather}

An explicit formula of the electrostatic potential
\begin{align}\label{efep}
\phi(t,x)=&\Phi[\md](t,x) \notag\\
:=&\int_0^x\int_0^y\big(\md(t,z)-D(z)\big)dzdy+\Bigg(\phi_r-\int_0^1\int_0^y\big(\md(t,z)-D(z)\big)dzdy\Bigg)x,
\end{align}
follows from \eqref{1dfqhd4} and \eqref{bc-d}.

In consideration of the solvability of the system \eqref{1dfqhd}, the following properties
\begin{subequations}\label{psc}
\begin{gather}
\inf_{x\in\Omega}\md>0,\qquad \inf_{x\in\Omega}\wnd>0,\label{pc}\\
\inf_{x\in\Omega} S[\md,\dl,\wnd]>0,\qquad \text{where}\  S[\md,\dl,\wnd]:=\wnd-\frac{\dl^2}{\md^2}\label{sc}
\end{gather}
\end{subequations}
attract our main interest. The condition \eqref{pc} means the positivity of the electron density and temperature. The other one \eqref{sc} is called the subsonic condition . Apparently, if we want to construct the solution in the physical region where the conditions  \eqref{psc} hold, then the initial data \eqref{ic} must satisfy the same conditions
\begin{equation}\label{ipsc}
\inf_{x\in\Omega}\md_{0}>0,\qquad \inf_{x\in\Omega}\wnd_0>0, \qquad\inf_{x\in\Omega}S[\md_{0},\dl_{0},\wnd_0]>0.
\end{equation}

The strength of the boundary data, which is defined by
\begin{equation}\label{delta}
\delta:=|n_l-n_r|+|\theta_l-\theta_{L}|+|\theta_r-\theta_{L}|+|\phi_r|,
\end{equation} 
plays a crucial role in the proofs of our main results in what follows.

The first aim in this paper is to investigate the existence, uniqueness and asymptotic stability of the stationary solution satisfying the following boundary value problem,
\begin{subequations}\label{1dsfqhd}
\begin{numcases}{}
\sdl_x=0, \label{1dsfqhd1}\\
S[\smd,\sdl,\swnd]\smd_x+\smd\swnd_x-\pc^2\smd\Bigg[\frac{\big(\sqrt{\smd}\big)_{xx}}{\sqrt{\smd}}\Bigg]_x=\smd\sdws_x-\sdl, \label{1dsfqhd2}\\
\sdl\swnd_x-\frac{2}{3}\sdl\swnd\big(\ln\smd\big)_x-\frac{2}{3}\swnd_{xx}-\frac{\pc^2}{3}\Bigg[\smd\Bigg(\frac{\sdl}{\smd}\Bigg)_{xx}\Bigg]_x=\frac{1}{3}\frac{\sdl^2}{\smd}-\smd(\swnd-\theta_{L}), \label{1dsfqhd3}\\
\sdws_{xx}=\smd-D(x), \qquad\forall x\in\Omega,\label{1dsfqhd4}
\end{numcases}
\end{subequations}
and
\begin{subequations}\label{sbc}
\begin{gather}
\smd(0)=\md_{l},\qquad \smd(1)=\md_{r},\label{sbc-a}\\
\big(\sqrt{\smd}\big)_{xx}(0)=\big(\sqrt{\smd}\big)_{xx}(1)=0,\label{sbc-b}\\
\swnd(0)=\wnd_{l},\qquad \swnd(1)=\wnd_{r},\label{sbc-c}\\
\sdws(0)=0,\qquad \sdws(1)=\phi_r.\label{sbc-d}
\end{gather}
\end{subequations}

The second aim in the present paper is to study the singular limit as the scaled Planck constant $\pc>0$ tends to zero in both the stationary problem \eqref{1dsfqhd}$\sim$\eqref{sbc} and the transient problem \eqref{1dfqhd}$\sim$\eqref{bc}. Formally, we let $\pc=0$ in the model system \eqref{1dfqhd} and its stationary counterpart \eqref{1dsfqhd}, respectively, we then obtain the following limit systems. The limit transient system can be written as
\begin{subequations}\label{1dfhd}
\begin{numcases}{}
\pcmd{0}_t+\pcdl{0}_x=0, \label{1dfhd1}\\
\pcdl{0}_t+\Bigg(\frac{(\pcdl{0})^2}{\pcmd{0}}+\pcmd{0}\pcwnd{0}\Bigg)_x=\pcmd{0}\pcdws{0}_x-\pcdl{0}, \label{1dfhd2}\\
\pcmd{0}\pcwnd{0}_t+\pcdl{0}\pcwnd{0}_x+\frac{2}{3}\pcmd{0}\pcwnd{0}\Bigg(\frac{\pcdl{0}}{\pcmd{0}}\Bigg)_x-\frac{2}{3}\pcwnd{0}_{xx}=\frac{1}{3}\frac{(\pcdl{0})^2}{\pcmd{0}}-\pcmd{0}(\pcwnd{0}-\theta_{L}), \label{1dfhd3}\\
\pcdws{0}_{xx}=\pcmd{0}-D(x), \qquad\forall t>0,\ \forall x\in\Omega,\label{1dfhd4}
\end{numcases}
\end{subequations}
and is supplemented by the same initial and ohmic contact boundary conditions with \eqref{ic} and \eqref{bc},
\begin{equation}\label{0ic}
(\pcmd{0},\pcdl{0},\pcwnd{0})(0,x)=(\md_0,\dl_0,\wnd_0)(x),
\end{equation}
and
\begin{subequations}\label{0bc}
\begin{gather}
\pcmd{0}(t,0)=\md_{l},\qquad \pcmd{0}(t,1)=\md_{r},\label{0bc-a}\\
\pcwnd{0}(t,0)=\wnd_{l},\qquad \pcwnd{0}(t,1)=\wnd_{r},\label{0bc-b}\\
\pcdws{0}(t,0)=0,\qquad \pcdws{0}(t,1)=\phi_r.\label{0bc-c}
\end{gather}
\end{subequations}
We call the limit system \eqref{1dfhd} as the full hydrodynamic (FHD) model for semiconductor devices. The limit stationary system is the stationary version of the FHD model \eqref{1dfhd}, it can be written by
\begin{subequations}\label{1dsfhd}
\begin{numcases}{}
\pcsdl{0}_x=0, \label{1dsfhd1}\\
S[\pcsmd{0},\pcsdl{0},\pcswnd{0}]\pcsmd{0}_x+\pcsmd{0}\pcswnd{0}_x=\pcsmd{0}\pcsdws{0}_x-\pcsdl{0}, \label{1dsfhd2}\\
\pcsdl{0}\pcswnd{0}_x-\frac{2}{3}\pcsdl{0}\pcswnd{0}\big(\ln\pcsmd{0}\big)_x-\frac{2}{3}\pcswnd{0}_{xx}=\frac{1}{3}\frac{(\pcsdl{0})^2}{\pcsmd{0}}-\pcsmd{0}(\pcswnd{0}-\theta_{L}), \label{1dsfhd3}\\
\pcsdws{0}_{xx}=\pcsmd{0}-D(x), \qquad\forall x\in\Omega,\label{1dsfhd4}
\end{numcases}
\end{subequations}
and is supplemented by the same ohmic contact boundary condition with \eqref{sbc},
\begin{subequations}\label{0sbc}
\begin{gather}
\pcsmd{0}(0)=\md_{l},\qquad \pcsmd{0}(1)=\md_{r},\label{0sbc-a}\\
\pcswnd{0}(0)=\wnd_{l},\qquad \pcswnd{0}(1)=\wnd_{r},\label{0sbc-b}\\
\pcsdws{0}(0)=0,\qquad \pcsdws{0}(1)=\phi_r.\label{0sbc-c}
\end{gather}
\end{subequations}

Throughout the rest of this paper, we will use the following notations. For a nonnegative integer $l\geq0$, $H^l(\Omega)$ denotes the $l$-th order Sobolev space in the $L^2$ sense, equipped with the norm $\|\cdot\|_l$. In particular, $H^0=L^2$ and $\|\cdot\|:=\|\cdot\|_0$. For a nonnegative integer $k\geq0$, $C^k(\overline{\Omega})$ denotes the $k$-times continuously differentiable function space, equipped with the norm $|f|_k:=\sum_{i=0}^k\sup_{x\in\overline{\Omega}}|\pd{x}{i}f(x)|$. The positive constants $C$, $C_1$, $\cdots$ only depend on $\md_{l}$, $\theta_L$ and $|D|_0$. If the constants  $C$, $C_1$, $\cdots$ additionally depend on some other quantities $\alpha$, $\beta$, $\cdots$, we write $C(\alpha,\beta,\cdots)$, $C_1(\alpha,\beta,\cdots)$, $\cdots$. The notations $\mathfrak{X}_m^l$, $\mathfrak{Y}_m^l$ and $\mathfrak{Z}$ denote the function spaces defined by
\begin{gather*}
\mathfrak{X}_m^l([0,T]):=\bigcap_{k=0}^m C^k([0,T];H^{l+m-k}(\Omega)),\\
\mathfrak{Y}_m^l([0,T]):=\bigcap_{k=0}^{[m/2]} C^k([0,T];H^{l+m-2k}(\Omega)),\quad\text{for}\ m,l=0,1,2,\cdots,\\
\mathfrak{Z}([0,T]):=C^2([0,T];H^2(\Omega)).
\end{gather*}

The limit problems \eqref{1dfhd}$\sim$\eqref{0bc} and \eqref{1dsfhd}$\sim$\eqref{0sbc} have been studied by Nishibata and Suzuki \cite{NS09}. The authors obtain the existence, uniqueness and asymptotic stability of the stationary solution. The corresponding results are stated in the following lemmas.
\begin{lemma}[Existence and uniqueness of the limit stationary solution, \cite{NS09}]\label{lem1}
Let the doping profile and the boundary data satisfy conditions \eqref{nonflat} and \eqref{0sbc}. For arbitrary positive constants $n_l$ and $\theta_{L}$, there exist three positive constants $\delta_1$, $c$ and $C$ such that if $\delta\leq\delta_1$, then the BVP \eqref{1dsfhd}$\sim$\eqref{0sbc} has a unique solution $(\pcsmd{0},\pcsdl{0},\pcswnd{0},\pcsdws{0})$ satisfying the condition \eqref{psc} in the space $C^2(\overline{\Omega})\times C^2(\overline{\Omega})\times H^3(\Omega)\times C^2(\overline{\Omega})$. Moreover, the stationary solution satisfies the estimates
\begin{equation}\label{0se}
0<c\leq\pcsmd{0},\pcswnd{0},S[\pcsmd{0},\pcsdl{0},\pcswnd{0}]\leq C,\quad |\pcsdl{0}|+\|\pcswnd{0}-\theta_{L}\|_3\leq C\delta,\quad |(\pcsmd{0},\pcsdws{0})|_2\leq C.
\end{equation}
\end{lemma}

\begin{lemma}[Stability of the limit stationary solution, \cite{NS09}]\label{lem2}
Let the doping profile and the boundary data satisfy conditions \eqref{nonflat} and \eqref{0bc}. Assume that the initial data $(\md_0,\dl_0,\wnd_0)\in\big[H^2(\Omega)\big]^3$ and satisfies the conditions \eqref{compatibility} and \eqref{ipsc}. For arbitrary positive constants $n_l$ and $\theta_{L}$, there exist three positive constants $\delta_2$, $\gamma_1$ and $C$ such that if $\delta+\|(\md_0-\pcsmd{0},\dl_0-\pcsdl{0},\wnd_0-\pcswnd{0})\|_2\leq\delta_2$, then the IBVP \eqref{1dfhd}$\sim$\eqref{0bc} has a unique global solution $(\pcmd{0},\pcdl{0},\pcwnd{0},\pcdws{0})$ satisfying the condition \eqref{psc} in the space $\mathfrak{X}_2([0,\infty))\times\big[\mathfrak{X}_1^1([0,\infty))\cap H_{loc}^2(0,\infty;L^2(\Omega))\big]\times\big[\mathfrak{Y}_2([0,\infty))\cap H_{loc}^1(0,\infty;H^1(\Omega))\big]\times\mathfrak{Z}([0,\infty))$. Moreover, the solution verifies the additional regularity $\pcdws{0}-\pcsdws{0}\in\mathfrak{X}_2^2([0,\infty))$ and the decay estimate
\begin{multline}\label{0de}
\|(\pcmd{0}-\pcsmd{0},\pcdl{0}-\pcsdl{0},\pcwnd{0}-\pcswnd{0})(t)\|_2+\|(\pcdws{0}-\pcsdws{0})(t)\|_4\\
\leq C\|(\md_{0}-\pcsmd{0},\dl_{0}-\pcsdl{0},\wnd_{0}-\pcswnd{0})\|_2\,e^{-\gamma_1 t},\quad\forall t\in [0,\infty).
\end{multline}
\end{lemma}

Now, we are in the position to state the main results in this paper. Firstly, the existence and uniqueness of the quantum stationary solution is summarized in the following theorem.
\begin{theorem}[Existence and uniqueness of the quantum stationary solution]\label{thm1}
Suppose that the doping profile and the boundary data satisfy conditions \eqref{nonflat} and \eqref{sbc}. For arbitrary positive constants $n_l$ and $\theta_{L}$, there exist three positive constants $\delta_3$, $\pc_1(\leq1)$ and $C$ such that if $\delta\leq\delta_3$ and $0<\pc\leq\pc_1$, then the BVP \eqref{1dsfqhd}$\sim$\eqref{sbc} has a unique solution $(\pcsmd{\pc},\pcsdl{\pc},\pcswnd{\pc},\pcsdws{\pc})\in H^4(\Omega)\times H^4(\Omega)\times H^3(\Omega)\times C^2(\overline{\Omega})$ satisfying the condition \eqref{psc} and the uniform estimates
\begin{subequations}\label{145.1}
\begin{gather}
0<b^2\leq\pcsmd{\pc}\leq B^2,\quad 0<\frac{1}{2}\theta_L\leq\pcswnd{\pc}\leq\frac{3}{2}\theta_L, \label{m145.1a}\\
\|\pcsmd{\pc}\|_2+\|(\pc\pd{x}{3}\pcsmd{\pc},\pc^2\pd{x}{4}\pcsmd{\pc})\|+|\pcsdws{\pc}|_2\leq C,\label{m145.1b}\\
|\pcsdl{\pc}|+\|\pcswnd{\pc}-\theta_{L}\|_3\leq C\delta,\label{m145.1c}
\end{gather}
\end{subequations}
where the positive constants $B$ and $b$ are defined as follows
\begin{equation}\label{75.1}
B:=\frac{3}{2}\sqrt{n_l}\,e^{2|D|_0/\theta_{L}},\quad b:=\frac{1}{2}\sqrt{n_l}\,e^{-(B^2+2|D|_0)/\theta_{L}}.
\end{equation}
\end{theorem}

The asymptotic stability of the quantum stationary solution is stated in the next theorem.
\begin{theorem}[Stability of the quantum stationary solution]\label{thm2}
Assume that the doping profile and the boundary data satisfy conditions \eqref{nonflat} and \eqref{bc}. Let the initial data $(\md_0,\dl_0,\wnd_0)\in H^4(\Omega)\times H^3(\Omega)\times H^2(\Omega)$ and satisfies the conditions \eqref{compatibility} and \eqref{ipsc}. For arbitrary positive constants $n_l$ and $\theta_{L}$, there exist four positive constants $\delta_4$, $\pc_2$, $\gamma_2$ and $C$ such that if $0<\pc\leq\pc_2$ and $\delta+\|(\md_0-\pcsmd{\pc},\dl_0-\pcsdl{\pc},\wnd_0-\pcswnd{\pc})\|_2+\|(\pc\pd{x}{3}(\md_0-\pcsmd{\pc}),\pc\pd{x}{3}(\dl_0-\pcsdl{\pc}),\pc^2\pd{x}{4}(\md_0-\pcsmd{\pc}))\|\leq\delta_4$, then the IBVP \eqref{1dfqhd}$\sim$\eqref{bc} has a unique global solution $(\pcmd{\pc},\pcdl{\pc},\pcwnd{\pc},\pcdws{\pc})$ satisfying the condition \eqref{psc} in $\big[\mathfrak{Y}_4([0,\infty))\cap H_{loc}^2(0,\infty;H^1(\Omega))\big]\times\big[\mathfrak{Y}_3([0,\infty))\cap H_{loc}^2(0,\infty;L^2(\Omega))\big]\times\big[\mathfrak{Y}_2([0,\infty))\cap H_{loc}^1(0,\infty;H^1(\Omega))\big]\times\mathfrak{Z}([0,\infty))$. Moreover, the solution verifies the additional regularity $\pcdws{\pc}-\pcsdws{\pc}\in\mathfrak{Y}_4^2([0,\infty))$ and the decay estimate
\begin{align}
&\|(\pcmd{\pc}-\pcsmd{\pc},\pcdl{\pc}-\pcsdl{\pc},\pcwnd{\pc}-\pcswnd{\pc})(t)\|_2 \notag\\
&+\|(\pc\pd{x}{3}(\pcmd{\pc}-\pcsmd{\pc}),\pc\pd{x}{3}(\pcdl{\pc}-\pcsdl{\pc}),\pc^2\pd{x}{4}(\pcmd{\pc}-\pcsmd{\pc}))(t)\|+\|(\pcdws{\pc}-\pcsdws{\pc})(t)\|_4 \notag\\
&\leq C\Big(\|(\md_0-\pcsmd{\pc},\dl_0-\pcsdl{\pc},\wnd_0-\pcswnd{\pc})\|_2 \notag\\
&\qquad\qquad\quad+\|(\pc\pd{x}{3}(\md_0-\pcsmd{\pc}),\pc\pd{x}{3}(\dl_0-\pcsdl{\pc}),\pc^2\pd{x}{4}(\md_0-\pcsmd{\pc}))\|\Big)\,e^{-\gamma_2 t},\quad\forall t\in [0,\infty).\label{de}
\end{align}
\end{theorem}

It is naturally expected that the solution $(\pcmd{\pc},\pcdl{\pc},\pcwnd{\pc},\pcdws{\pc})$ of the quantum system \eqref{1dfqhd} approaches the solution $(\pcmd{0},\pcdl{0},\pcwnd{0},\pcdws{0})$ of the limit system \eqref{1dfhd} as $\pc$ tends to zero. To justify this expectation, we first consider the convergence of the stationary solutions. Precisely, we show that the quantum stationary solution $(\pcsmd{\pc},\pcsdl{\pc},\pcswnd{\pc},\pcsdws{\pc})$ of the BVP \eqref{1dsfqhd}$\sim$\eqref{sbc} converges to the limit stationary solution $(\pcsmd{0},\pcsdl{0},\pcswnd{0},\pcsdws{0})$ of the BVP \eqref{1dsfhd}$\sim$\eqref{0sbc} as $\pc$ tends to zero. Then, we further study the convergence of the global solutions. The former result is summarized in the following theorem.
\begin{theorem}[Semi-classical limit of the stationary solutions]\label{thm3}
Suppose that the same conditions in Lemma \ref{lem1} and Theorem \ref{thm1} hold. For arbitrary positive constants $n_l$ and $\theta_{L}$, there exist two positive constants $\delta_5$ and $C$ such that if $\delta\leq\delta_5$, then for all $0<\pc\leq\pc_1$ (where $\pc_1$ is given in Theorem \ref{thm1}) the following convergence estimate
\begin{subequations}\label{ss39.1}
\begin{equation}\label{ss39.1a}
\|\pcsmd{\pc}-\pcsmd{0}\|_1+|\pcsdl{\pc}-\pcsdl{0}|+\|\pcswnd{\pc}-\pcswnd{0}\|_2+\|\pcsdws{\pc}-\pcsdws{0}\|_3\leq C\pc,
\end{equation}
holds true. Furthermore,
\begin{equation}\label{ss39.1b}
\big\|\big(\pd{x}{2}(\pcsmd{\pc}-\pcsmd{0}),\pc\pd{x}{3}\pcsmd{\pc},\pc^2\pd{x}{4}\pcsmd{\pc},\pd{x}{3}(\pcswnd{\pc}-\pcswnd{0}),\pd{x}{4}(\pcsdws{\pc}-\pcsdws{0})\big)\big\|\rightarrow0,\quad\text{as}\ \pc\rightarrow0.
\end{equation}
\end{subequations}
\end{theorem}

The semi-classical limit of the transient problem is stated in the next theorem.
\begin{theorem}[Semi-classical limit of the global solutions]\label{thm4}
Assume that the same conditions in Lemma \ref{lem2} and Theorem \ref{thm2} hold. For arbitrary positive constants $n_l$ and $\theta_{L}$, there exist four positive constants $\delta_6$, $\gamma_3$, $\gamma_4$ and $C$ such that if
\begin{multline}\label{sic}
\pc+\delta+\|(\md_0-\pcsmd{\pc},\dl_0-\pcsdl{\pc},\wnd_0-\pcswnd{\pc})\|_2\\
+\big\|\big(\pc\pd{x}{3}(\md_0-\pcsmd{\pc}),\pc\pd{x}{3}(\dl_0-\pcsdl{\pc}),\pc^2\pd{x}{4}(\md_0-\pcsmd{\pc})\big)\big\|\leq\delta_6,
\end{multline}
then the following convergence estimates
\begin{subequations}\label{gs}
\begin{equation}\label{gs164.5}
\|(\pcmd{\pc}-\pcmd{0},\pcdl{\pc}-\pcdl{0},\pcwnd{\pc}-\pcwnd{0})(t)\|_1+\|(\pcdws{\pc}-\pcdws{0})(t)\|_3\leq Ce^{\gamma_3 t}\pc^{1/2},\quad\forall t\in[0,\infty),
\end{equation}
and
\begin{equation}\label{gs1}
\sup_{t\in[0,\infty)}\Big(\|(\pcmd{\pc}-\pcmd{0},\pcdl{\pc}-\pcdl{0},\pcwnd{\pc}-\pcwnd{0})(t)\|_1+\|(\pcdws{\pc}-\pcdws{0})(t)\|_3\Big)\leq C\pc^{\gamma_4}
\end{equation}
hold true.
\end{subequations}
\end{theorem}

Now, we illustrate the main ideas and the key technical points in the proofs of the above theorems. Firstly, we apply the Schauder fixed-point theorem to solve the stationary problem \eqref{1dsfqhd}$\sim$\eqref{sbc}. To this end, we heuristically construct a fixed-point mapping $\mathcal{T}$, see \eqref{110.1}, through a careful observation on the structure of the stationary FQHD model \eqref{1dsfqhd}. Roughly speaking, in order to deal with the Bohm potential term in the stationary momentum equation \eqref{1dsfqhd2}, we introduce a transformation $\rsmd:=\sqrt{\smd}$ and reduce the stationary momentum equation to a parameter-dependent semilinear elliptic equation of the second order with the nonlocal terms by using the vanishing bohmenian-type boundary condition \eqref{sbc-b}. In order to treat the dispersive velocity term in the stationary energy equation \eqref{1dsfqhd3}, the desired mapping $\mathcal{T}$ has to be defined by solving two carefully designed nonlocal problems $(P1)$ and $(P2)$ in turn, see \eqref{109.1} and \eqref{109.2}.  The unique solvability of both $(P1)$ and $(P2)$ can be proved by using the Leray-Schauder fixed-point theorem and energy estimates. This ensure that the mapping $\mathcal{T}$ is well-defined. During the proof,  the main difficulty is to establish the uniform (in $\pc$) estimates \eqref{eu} and \eqref{eQ}.

Secondly, the existence of the global-in-time solution and the asymptotic stability of the stationary solution can be proved by the standard continuation argument based on the local existence and the uniform a priori estimate. Similar to the stationary problem, we also introduce a transformation $\rmd:=\sqrt{\md}$ to conveniently deal with the Bohm potential term in the momentum equation \eqref{1dfqhd2}. The local existence result is proved by combining the iteration method with the energy estimates. The unique solvability of the linearized problem \eqref{a3.1}$\sim$\eqref{a3.3} used to design the iteration scheme is shown in Appendix by Galerkin method, where we have used the existence result in \cite{NS08} for a fourth order wave equation. The uniform a priori estimate is established by refined energy method. The proof is very complicated due to the non-flatness of the stationary density and the appearance of the dispersive velocity term in the perturbed energy equation \eqref{10.2c}. During the proof, we find that the spatial derivatives of the perturbations $(\prmd,\pdl,\pwnd,\pdws)$ can be bounded by the temporal derivatives of the perturbations $(\prmd,\pdl,\pwnd)$ with the help of the special structure of the perturbed system \eqref{10.2}, see \eqref{46.5}. Therefore, we only need to establish the estimates of the temporal derivatives of the perturbations $(\prmd,\pdl,\pwnd)$ by using the homogeneous boundary condition \eqref{pbc}. We also find the interplay of the dissipative-dispersive effects in the FQHD model. Roughly speaking, the Bohm potential term in the perturbed momentum equation \eqref{10.2b} contributes the quantum dissipation rate $\|\pc\pd{t}{k}\prmd_{xx}(t)\|$, see \eqref{115.1}. The dispersive velocity term in the perturbed energy equation \eqref{10.2c} contributes the extra quantum dissipation rate $\|\pc\pd{t}{k}\prmd_{tx}(t)\|$, see \eqref{115.2}. The dissipative property of the dispersive velocity term plays a crucial role to close the uniform a priori estimate \eqref{127.2}.

Finally, we justify the semi-classical limit for both the stationary solutions and global solutions by using the energy method and compactness argument. For stationary solutions, in order to overcome the difficulties arising from the non-flatness of the stationary density, we need to introduce the transformations $\pcszmd{\pc}:=\ln \pcsmd{\pc}$ and $\pcszmd{0}:=\ln \pcsmd{0}$. In addition, we also have to technically estimate a bad integral term $I_2$ during establishing the error estimate of the stationary temperature error variable $\pcpswnd$. Actually, we find that the quantum stationary current density $\pcsdl{\pc}$ and the limit stationary current density $\pcsdl{0}$ possess the same explicit formula \eqref{104.1} due to the vanishing bohmenian-type boundary condition. Based on this fact, we can successfully overcome the difficulty in estimating the integral term $I_2$, see \eqref{ss17.2}. For global solutions, we have to pay more attention on the influences of the quantum corrected energy equation \eqref{157.2}, see \eqref{152.1} and \eqref{161.6} for example, the computations are very complicated. In the proof, the semi-classical limit of the stationary solutions plays important role.

The paper is organized as follows. In Section \ref{Sect.2}, we prove the existence and uniqueness of the stationary solution. Section \ref{Sect.3} is devoted to the global existence and stability analysis. In Subsection \ref{Subsect.3.1}, we show the local existence. In Subsections \ref{Subsect.3.2}$\sim$\ref{Subsect.3.5}, we reformulate the problem and establish the uniform a priori estimate. Section \ref{Sect.4} is devoted to the verification of the semi-classical limit. In Subsection \ref{Subsect.4.1}, we discuss the stationary case. In Subsection \ref{Subsect.4.2}, we study the non-stationary case.

\section{Existence and uniqueness of the stationary solution}\label{Sect.2}
In this section, we show Theorem \ref{thm1}. The proof is based on the Schauder fixed-point theorem (see Corollary 11.2 in \cite{GT98}), the Leray-Schauder fixed-point theorem (see Theorem 11.3 in \cite{GT98}) and the energy method.

\begin{proof}[\textbf{Proof of Theorem \ref{thm1}.}]
Since the proof is complicated, we divide it into several steps for clarification.

\emph{Step 1. Reformulation of the problem .}
It is convenient to make use of the transformation $\rsmd:=\sqrt{\smd}$. Inserting this transformation into the system \eqref{1dsfqhd}, dividing the the equation \eqref{1dsfqhd2} by $\rsmd^2$ and integrating the resultant on $[0,x)$ and then using the boundary condition \eqref{sbc}, applying the Green formula to the equation \eqref{1dsfqhd4} together with the boundary condition \eqref{sbc-d}, via the necessary calculations, the above procedures yield the following BVP with a constant current density $\sdl$ (which will be determined later, see \eqref{104.1} below),
\begin{subequations}\label{103.2}
\begin{numcases}{}
\pc^2\rsmd_{xx}=h(\rsmd,\swnd), \label{103.2a}\\
\frac{2}{3}\swnd_{xx}-\sdl\swnd_{x}+\frac{2}{3}\sdl\swnd(\ln \rsmd^2)_x-\rsmd^2(\swnd-\theta_L)=g(\rsmd,\swnd;\pc), \quad x\in\Omega,\label{103.2b}
\end{numcases}
\end{subequations}
with boundary conditions
\begin{subequations}\label{103.3}
\begin{gather}
\rsmd(0)=\rmd_{l},\qquad \rsmd(1)=\rmd_{r},\label{103.3a}\\
\swnd(0)=\wnd_{l},\qquad \swnd(1)=\wnd_{r},\label{103.3b}
\end{gather}
\end{subequations}
where
\begin{subequations}
\begin{gather}
F(a_1,a_2,a_3):=\frac{a_2^2}{2a_1^2}+a_3+a_3\ln a_1,\quad\rmd_{l}:=\sqrt{\md_{l}}, \quad\rmd_{r}:=\sqrt{\md_{r}},\label{103.5}\\
\sdws(x)=G[\rsmd^2](x):=\int_0^1G(x,y)(\rsmd^2-D)(y)dy+\phi_rx,\quad G(x,y):=\begin{cases}x(y-1),\ x<y\\y(x-1),\ x>y\end{cases},\label{104.2}\\
h(\rsmd,\swnd):=\rsmd\bigg[F(\rsmd^2,\sdl,\swnd)-F(\md_{l},\sdl,\theta_l)-\sdws-\int_0^x\swnd_{x}\ln\rsmd^2dy+\sdl\int_0^x\rsmd^{-2}dy\bigg],\label{103.4}\\
g(\rsmd,\swnd;\pc):=-\frac{1}{3}\frac{\sdl^2}{\rsmd^2}+\frac{\pc^2}{3}\sdl\bigg(\frac{12\rsmd_x^3}{\rsmd^3}-\frac{14\rsmd_x\rsmd_{xx}}{\rsmd^2}+\frac{2\rsmd_{xxx}}{\rsmd}\bigg).\label{103}
\end{gather}
\end{subequations}

Next, taking value $x=1$ in the equation \eqref{103.2a} and using the boundary condition \eqref{sbc}, we obtain the current-voltage relation
\begin{equation}\label{ucvr}
F(\md_{r},\sdl,\theta_r)-F(\md_{l},\sdl,\theta_l)-\phi_r-\int_0^1\swnd_{x}\ln\rsmd^2dy+\sdl\int_0^1\rsmd^{-2}dy=0.
\end{equation}
Easy to see, the the equation \eqref{ucvr} is a quadratic equation on $\sdl$. Based on the subsonic condition \eqref{psc}, we can uniquely solve $\sdl$ provided $\rsmd$, $\swnd$ are given and the strength parameter $\delta$ is small enough. Precisely, the constant stationary current density $\sdl$ satisfies the following explicit formula
\begin{align}\label{104.1}
&\sdl=J[\rsmd^2,\swnd]:=2\Big(\bar{b}+\int_0^1\swnd_{x}\ln\rsmd^2dy\Big)K[\rsmd^2,\swnd]^{-1},\\
&K[\rsmd^2,\swnd]:=\int_0^1\rsmd^{-2}dy+\sqrt{\Big(\int_0^1\rsmd^{-2}dy\Big)^2+2\Big(\bar{b}+\int_0^1\swnd_{x}\ln\rsmd^2dy\Big)\Big(n_{r}^{-2}-n_{l}^{-2}\Big)},\notag\\
&\bar{b}:=\phi_r-\theta_r+\theta_l-\theta_r\ln n_r+\theta_l\ln n_l.\notag
\end{align}

It is obvious that the BVP \eqref{103.2}$\sim$\eqref{103.3} combined with the explicit formulas \eqref{104.1} and \eqref{104.2} is equivalent to the original BVP \eqref{1dsfqhd}$\sim$\eqref{sbc} under the transformation $\smd=\rsmd^2$ for positive smooth solution $(\rsmd,\swnd)$. 

\emph{Step 2. Construction of the fixed-point mapping.} From now on, we focus on the unique solvability of the BVP \eqref{103.2}$\sim$\eqref{103.3}. The system \eqref{103.2} is a one-dimensional semilinear nonlocal elliptic system with a singular parameter $\pc\in(0,1]$ in the principal part of its first component equation \eqref{103.2a}. To solve it, we adopt the conventional framework based on the Schauder fixed-point theorem.

Observing the structure of the system \eqref{103.2}, we can construct the fixed-point mapping appropriately by the following procedure.

Firstly, we introduce a closed convex subset $\mathcal{U}[N_1,N_2]$ in the Banach space $C^2(\overline{\Omega})$ below, where $N_1$ and $N_2$ are positive constants to be determined later (see \eqref{N12} below),
\begin{equation}\label{120.1}
\mathcal{U}[N_1,N_2]:=\Big\{q\in C^2(\overline{\Omega})\,\Big|\, \|q-\theta_L\|_1\leq N_1\delta,\quad\|q_{xx}\|\leq N_2\delta,\quad q(0)=\theta_l,\ q(1)=\theta_r \Big\}.
\end{equation}

Next, we define the fixed-point mapping
\begin{align}\label{110.1}
\mathcal{T}:\ \mathcal{U}[N_1,N_2]&\longrightarrow H^3(\Omega)\notag\\
q&\longmapsto Q
\end{align}
by solving the following two problems in turn. For any fixed $q\in \mathcal{U}[N_1,N_2]$, we firstly solve the problem $(P1)$:
\begin{subequations}\label{109.1}
\begin{numcases}{(P1)\quad}
\pc^2u_{xx}=h(u,q), \quad x\in\Omega,\label{109.1a}\\
u(0)=w_l,\quad u(1)=w_r.
\end{numcases}
\end{subequations}
For problem $(P1)$, we claim the following fact, its proof will be given in the next step.
\begin{quotation}
\emph{Claim 1. For given $q\in \mathcal{U}[N_1,N_2]$, if $\delta$ and $\pc$ are small enough, then $(P1)$ has a unique solution $u=u[q]\in H^4(\Omega)$ satisfying the following uniform estimate with respect to $\pc$, that is,
\begin{subequations}\label{eu}
\begin{gather}
0<b\leq u(x)\leq B,\label{eu-a}\\
\|u\|_2+\|(\pc\pd{x}{3}u, \pc^2\pd{x}{4}u)\|\leq C,\label{eu-b}
\end{gather}
\end{subequations}
where the positive constants $b$ and $B$ are given by \eqref{75.1}, and the positive constant $C$ only depends on $n_l$, $\theta_L$ and $|D|_0$.}
\end{quotation}
Based on the Claim 1, for given function pair $(u,q)$, we further solve the problem $(P2)$:
\begin{subequations}\label{109.2}
\begin{numcases}{(P2)\quad}
\frac{2}{3}Q_{xx}-JQ_{x}+\frac{2}{3}J_*(\ln u^2)_x\theta_L+\frac{2}{3}J(\ln u^2)_x(Q-\theta_L)\notag\\
\qquad\qquad\qquad\qquad\qquad\qquad\ \,-u^2(Q-\theta_L)=g(u,q;\pc), \quad x\in\Omega,\label{109.2a}\\
Q(0)=\theta_l,\quad Q(1)=\theta_r,
\end{numcases}
\end{subequations}
where $J:=J[u^2,q]$ and $J_*:=2\big(\bar{b}+\int_0^1Q_x\ln u^2dx\big)K[u^2,q]^{-1}$. For problem $(P2)$, we also have a claim, its proof will be given in the Step 4.
\begin{quotation}
\emph{Claim 2. For given $(u,q)$ in Claim 1, if $\delta$ is small enough, then $(P2)$ has a unique solution $Q\in H^3(\Omega)$ satisfying the following uniform estimate with respect to $\pc$, that is,
\begin{subequations}\label{eQ}
\begin{align}
&\|Q-\theta_L\|_1\leq C_1\delta+C_2(b,B,N_1)\delta^2,\label{eQa}\\
&\|Q_{xx}\|\leq C_3(b,B,N_1)\delta,\label{eQb}\\
&\|Q_{xxx}\|\leq C_4(b,B,N_1,N_2)\delta,\label{eQc}
\end{align}
\end{subequations}
where the positive constant $C_1$ only depends on $n_l$, $\theta_L$ and $|D|_0$.}
\end{quotation}

\emph{Step 3. Proof of Claim 1.} 
Now, we begin to solve $(P1)$. In order to avoid vacuum $u=0$, we consider a truncation problem $(tP)$ induced by $(P1)$:
\begin{subequations}\label{111.1}
\begin{numcases}{(tP)\quad}
\pc^2u_{xx}=h(u_{\alpha\beta},q), \quad x\in\Omega,\label{111.1a}\\
u(0)=w_l,\quad u(1)=w_r,\label{111.1b}
\end{numcases}
\end{subequations}
where
\begin{equation*}
u_{\alpha\beta}:=\max\big\{\beta,\min\{\alpha,u\}\big\},\quad 0<\frac{1}{2}b=:\beta<\alpha:=2B.
\end{equation*}
This problem can be solved by Leray-Schauder fixed-point theorem. To this end, we define a fixed-point mapping $\mathcal{T}_1: r\mapsto R$ over $H^1(\Omega)$ by solving the linear problem:
\begin{subequations}\label{112.1}
\begin{numcases}{}
\pc^2R_{xx}=h(r_{\alpha\beta},q), \quad x\in\Omega,\label{112.1a}\\
R(0)=w_l,\quad R(1)=w_r.\label{112.1b}
\end{numcases}
\end{subequations}
In fact, for given $q\in\mathcal{U}[N_1,N_2]$ and $r\in H^1(\Omega)$, the right-side $h(r_{\alpha\beta},q)\in H^1(\Omega)$. Thus, the linear BVP \eqref{112.1} has a unique solution $R=:\mathcal{T}_1r\in H^3(\Omega)$ by the standard theory of the elliptic equations. In addition, the mapping $\mathcal{T}_1$ is a continuous and compact mapping from $H^1(\Omega)$ into itself. Next, we show that there exists a positive constant $M_1$ such that $\|v\|_1\leq M_1$ for an arbitrary $v\in\big\{f\in H^1(\Omega)\,|\,f=\lambda\mathcal{T}_1f, \ \forall\lambda\in[0,1] \big\}$. We may assume $\lambda>0$ as the case $\lambda=0$ is trivial. It is sufficient to show that $\|v\|_1\leq M_1$ for $v$ satisfying the following problem
\begin{subequations}\label{113.3}
\begin{numcases}{}
\pc^2v_{xx}=\lambda h(v_{\alpha\beta},q), \quad x\in\Omega,\label{113.3a}\\
v(0)=\lambda w_l,\quad v(1)=\lambda w_r.\label{113.3b}
\end{numcases}
\end{subequations}
Performing the procedure
\begin{equation*}\label{153.2}
\int_0^1\eqref{113.3a}\times(v-\lambda\bar{w})dx,\quad\text{where}\ \bar{w}(x):=w_l(1-x)+w_rx,
\end{equation*}
and using the Young inequality, the mean value theorem, the formula \eqref{104.1} and the estimate $|h(v_{\alpha\beta},q)|\leq C$, where $C$ is a positive constant which only depends on $\alpha$, $\beta$, $n_l$, $\theta_L$, and $|D|_0$. If $\delta$ is small enough, then these computations yield the desired estimate
\begin{equation}\label{153.4}
\|v\|_1\leq C\bigg(1+\frac{1}{\pc}\bigg)=:M_1.
\end{equation}
Based on the estimate \eqref{153.4}, we can directly apply the Leray-Schauder fixed-point theorem to the mapping $\mathcal{T}_1$, and see that $\mathcal{T}_1$ has a fixed-point $u=\mathcal{T}_1u\in H^3(\Omega)$ which is a strong solution to the truncation problem $(tP)$.

Next, we can further provide a maximum principle argument for any strong solution $u$ to the truncation problem $(tP)$. Consequently, this result can help us to remove the truncation in $(tP)$ and show that the solution $u$ to the truncation problem $(tP)$ exactly is a solution to the problem $(P1)$.

We first establish the upper bound of $u_{\alpha\beta}$. Before doing this, we note that if $\delta$ is small enough, then $q\in\mathcal{U}[N_1,N_2]$ implies
\begin{equation}\label{154.2}
0<\frac{1}{2}\theta_L\leq q(x)\leq\frac{3}{2}\theta_L,\quad\|q_x\|\leq N_1\delta,\quad\|q_{xx}\|\leq N_2\delta.
\end{equation} 
Now, we can establish the upper bound of $u_{\alpha\beta}$ by choosing the appropriate test functions in $H_0^1(\Omega)$. To this end, we define $\bar{n}:=\max\{n_l,n_r\}>0$, and perform the procedure
\begin{equation}\label{154.1}
\int_0^1-\eqref{111.1a}\times\bigg(\ln\frac{u_{\alpha\beta}^2}{\bar{n}}\bigg)_+^kdx,\quad k=1,2,3,\cdots,\quad\text{where}\ (\cdot)_+:=\max\{0,\cdot\}.
\end{equation}
The computations in terms of this procedure yield that
\begin{equation}\label{56.1}
\int_0^1-\pc^2u_{xx}\bigg(\ln\frac{u_{\alpha\beta}^2}{\bar{n}}\bigg)_+^kdx=\int_0^1-h(u_{\alpha\beta},q)\bigg(\ln\frac{u_{\alpha\beta}^2}{\bar{n}}\bigg)_+^kdx.
\end{equation}
The left-side of \eqref{56.1} can be estimated as follows by integration by parts,
\begin{align}
\eqref{56.1}_l&=\int_0^1\pc^2u_x\Bigg[\bigg(\ln\frac{u_{\alpha\beta}^2}{\bar{n}}\bigg)_+^k\Bigg]_xdx \notag\\
&=\int_0^12\pc^2k\frac{[(u_{\alpha\beta})_x]^2}{u_{\alpha\beta}}\bigg(\ln\frac{u_{\alpha\beta}^2}{\bar{n}}\bigg)_+^{k-1}dx\geq0.\label{57.1}
\end{align}
Based on the expression \eqref{103.4}, the estimate \eqref{154.2} and the Young inequality, we can estimate the right-side of \eqref{56.1} as follows, where $\bar{J}:=J[u_{\alpha\beta}^{2},q]$ satisfying the estimate $|\bar{J}|\leq C(\alpha,\beta,N_1)\delta$. Namely,
\allowdisplaybreaks
\begin{align}
\eqref{56.1}_r=&\int_0^1-u_{\alpha\beta}\bigg[F(u_{\alpha\beta}^2,\bar{J},q)-F(\md_{l},\bar{J},\theta_l)-G[u_{\alpha\beta}^{2}]\notag\\
&\qquad\qquad\qquad\qquad-\int_0^xq_{x}\ln u_{\alpha\beta}^2dy+\bar{J}\int_0^xu_{\alpha\beta}^{-2}dy\bigg]\bigg(\ln\frac{u_{\alpha\beta}^2}{\bar{n}}\bigg)_+^kdx\notag\\
=&\int_0^1-u_{\alpha\beta}\bigg[F(u_{\alpha\beta}^2,\bar{J},q)-q\ln\bar{n}+q\ln\bar{n}-F(\md_{l},\bar{J},\theta_l)-G[u_{\alpha\beta}^{2}]\notag\\
&\qquad\qquad\qquad\qquad\qquad\qquad\qquad\ -\int_0^xq_{x}\ln u_{\alpha\beta}^2dy+\bar{J}\int_0^xu_{\alpha\beta}^{-2}dy\bigg]\bigg(\ln\frac{u_{\alpha\beta}^2}{\bar{n}}\bigg)_+^kdx\notag\\
=&-\int_0^1u_{\alpha\beta}q\bigg(\ln\frac{u_{\alpha\beta}^2}{\bar{n}}\bigg)_+^{k+1}dx\notag\\
&+\int_0^1\bigg(\phi_rx-\int_0^1G(x,y)D(y)dy-\bar{J}\int_0^xu_{\alpha\beta}^{-2}dy+\frac{\bar{J}^2}{2n_l}\bigg)u_{\alpha\beta}\bigg(\ln\frac{u_{\alpha\beta}^2}{\bar{n}}\bigg)_+^kdx\notag\\
&+\int_0^1\bigg(\theta_l\ln n_l-q\ln\bar{n}+\theta_l-q+\int_0^xq_x\ln u_{\alpha\beta}^2dy \bigg)u_{\alpha\beta}\bigg(\ln\frac{u_{\alpha\beta}^2}{\bar{n}}\bigg)_+^kdx\notag\\
&+\int_0^1\bigg(\int_0^1\underbrace{G(x,y)}_{\leq0}u_{\alpha\beta}^2(y)dy-\frac{\bar{J}^2}{2u_{\alpha\beta}^4}\bigg)u_{\alpha\beta}\bigg(\ln\frac{u_{\alpha\beta}^2}{\bar{n}}\bigg)_+^kdx\notag\\
\leq&-\int_0^1\frac{1}{2}\theta_Lu_{\alpha\beta}\bigg(\ln\frac{u_{\alpha\beta}^2}{\bar{n}}\bigg)_+^{k+1}dx+\int_0^1\big(C(N_1)\delta+|D|_0\big)u_{\alpha\beta}\bigg(\ln\frac{u_{\alpha\beta}^2}{\bar{n}}\bigg)_+^kdx\notag\\
&+\int_0^1C(N_1)\delta u_{\alpha\beta}\bigg(\ln\frac{u_{\alpha\beta}^2}{\bar{n}}\bigg)_+^kdx+0\notag\\
\leq&-\int_0^1\frac{1}{2}\theta_Lu_{\alpha\beta}\bigg(\ln\frac{u_{\alpha\beta}^2}{\bar{n}}\bigg)_+^{k+1}dx+\int_0^12|D|_0u_{\alpha\beta}\bigg(\ln\frac{u_{\alpha\beta}^2}{\bar{n}}\bigg)_+^kdx \qquad\text{if}\ \delta\ll1\notag\\
=&-\int_0^1\frac{1}{2}\theta_Lu_{\alpha\beta}\bigg(\ln\frac{u_{\alpha\beta}^2}{\bar{n}}\bigg)_+^{k+1}dx+\int_0^12|D|_0\frac{2}{\theta_L}\frac{\theta_L}{2}u_{\alpha\beta}\bigg(\ln\frac{u_{\alpha\beta}^2}{\bar{n}}\bigg)_+^kdx\notag\\
=&-\int_0^1\frac{1}{2}\theta_Lu_{\alpha\beta}\bigg(\ln\frac{u_{\alpha\beta}^2}{\bar{n}}\bigg)_+^{k+1}dx+\int_0^1\frac{1}{2}\theta_Lu_{\alpha\beta}\underbrace{\bigg(\ln\frac{u_{\alpha\beta}^2}{\bar{n}}\bigg)_+^k\frac{4|D|_0}{\theta_L}}_{\text{by Young inequality}}dx\notag\\
\leq&-\int_0^1\frac{1}{2}\theta_Lu_{\alpha\beta}\bigg(\ln\frac{u_{\alpha\beta}^2}{\bar{n}}\bigg)_+^{k+1}dx\notag\\
&\qquad\qquad\qquad+\int_0^1\frac{1}{2}\theta_Lu_{\alpha\beta}\Bigg[\frac{k}{k+1}\bigg(\ln\frac{u_{\alpha\beta}^2}{\bar{n}}\bigg)_+^{k+1}+\frac{1}{k+1}\bigg(\frac{4|D|_0}{\theta_L}\bigg)^{k+1}\Bigg]dx\notag\\
=&-\frac{1}{k+1}\frac{1}{2}\theta_L\int_0^1u_{\alpha\beta}\bigg(\ln\frac{u_{\alpha\beta}^2}{\bar{n}}\bigg)_+^{k+1}dx+\frac{1}{k+1}\frac{1}{2}\theta_L\bigg(\frac{4|D|_0}{\theta_L}\bigg)^{k+1}\int_0^1\underbrace{u_{\alpha\beta}}_{\leq\alpha}dx\notag\\
\leq&\frac{\theta_L}{2(k+1)}\Bigg[-\int_0^1u_{\alpha\beta}\bigg(\ln\frac{u_{\alpha\beta}^2}{\bar{n}}\bigg)_+^{k+1}dx+\alpha\bigg(\frac{4|D|_0}{\theta_L}\bigg)^{k+1}\Bigg]. \label{59.3}
\end{align}
Inserting \eqref{57.1} and \eqref{59.3} into \eqref{56.1}, we have the estimate
\begin{equation}\label{60.2}
\int_0^1\sqrt{\bar{n}}\bigg(\ln\frac{u_{\alpha\beta}^2}{\bar{n}}\bigg)_+^{k+1}dx\leq\int_0^1u_{\alpha\beta}\bigg(\ln\frac{u_{\alpha\beta}^2}{\bar{n}}\bigg)_+^{k+1}dx\leq\alpha\bigg(\frac{4|D|_0}{\theta_L}\bigg)^{k+1},
\end{equation}
which implies 
\begin{equation}\label{61.2}
\bigg\|\bigg(\ln\frac{u_{\alpha\beta}^2}{\bar{n}}\bigg)_+\bigg\|_{L^{k+1}(\Omega)}\leq\bigg(\frac{\alpha}{\sqrt{\bar{n}}}\bigg)^{\frac{1}{k+1}}\frac{4|D|_0}{\theta_L},\quad k=1,2,3,\cdots.
\end{equation}
Let $k\rightarrow\infty$ in \eqref{61.2}, we immediately obtain
\begin{equation}\label{61.3}
\bigg\|\bigg(\ln\frac{u_{\alpha\beta}^2}{\bar{n}}\bigg)_+\bigg\|_{L^{\infty}(\Omega)}\leq\frac{4|D|_0}{\theta_L}.
\end{equation}
Note that $[\ln(u_{\alpha\beta}^2/\bar{n})]_+$ is nonnegative, then the estimate \eqref{61.3} implies that 
\begin{equation}\label{62.1}
u_{\alpha\beta}\leq\sqrt{\bar{n}}e^{2|D|_0/\theta_L}\leq B,\quad\text{if}\ \delta\ll1.
\end{equation}

Using the similar argument, we can establish the lower bound of $u_{\alpha\beta}$.  To this end, we define $\underbar{n}:=\min\{n_l,n_r\}>0$, and perform the procedure
\begin{equation}\label{159.1}
\int_0^1-\eqref{111.1a}\times\frac{1}{u_{\alpha\beta}}\bigg(\ln\frac{u_{\alpha\beta}^2}{\underbar{n}}\bigg)_-^{2k-1}dx,\quad k=1,2,3,\cdots,\quad\text{where}\ (\cdot)_-:=\min\{0,\cdot\}.
\end{equation}
The computations in terms of this procedure yield that
\begin{equation}\label{63.1}
\int_0^1-\pc^2\frac{u_{xx}}{u_{\alpha\beta}}\bigg(\ln\frac{u_{\alpha\beta}^2}{\underbar{n}}\bigg)_-^{2k-1}dx=\int_0^1-\frac{h(u_{\alpha\beta},q)}{u_{\alpha\beta}}\bigg(\ln\frac{u_{\alpha\beta}^2}{\underbar{n}}\bigg)_-^{2k-1}dx.
\end{equation}
The left-side of \eqref{63.1} can be estimated as follows by integration by parts,
\begin{align}
\eqref{63.1}_l&=\int_0^1\pc^2u_x\Bigg[\frac{1}{u_{\alpha\beta}}\bigg(\ln\frac{u_{\alpha\beta}^2}{\underbar{n}}\bigg)_-^{2k-1}\Bigg]_xdx \notag\\
&=\int_0^1\pc^2u_x\bigg(\frac{1}{u_{\alpha\beta}}\bigg)_x\bigg(\ln\frac{u_{\alpha\beta}^2}{\underbar{n}}\bigg)_-^{2k-1}dx+\int_0^1\pc^2\frac{u_x}{u_{\alpha\beta}}\Bigg[\bigg(\ln\frac{u_{\alpha\beta}^2}{\underbar{n}}\bigg)_-^{2k-1}\Bigg]_x\notag\\
&=-\int_0^1\pc^2\bigg[\frac{(u_{\alpha\beta})_x}{u_{\alpha\beta}}\bigg]^2\Bigg[\bigg(\ln\frac{u_{\alpha\beta}^2}{\underbar{n}}\bigg)_-^{2k-1}-2(2k-1)\bigg(\ln\frac{u_{\alpha\beta}^2}{\underbar{n}}\bigg)_-^{2k-2}\Bigg]dx\notag\\
&\geq0.\label{64.1}
\end{align}
The right-side of \eqref{63.1} can be estimated as follows,
\allowdisplaybreaks
\begin{align}
\eqref{63.1}_r=&\int_0^1-\bigg[F(u_{\alpha\beta}^2,\bar{J},q)-F(\md_{l},\bar{J},\theta_l)-G[u_{\alpha\beta}^{2}]\notag\\
&\qquad\qquad\qquad\qquad-\int_0^xq_{x}\ln u_{\alpha\beta}^2dy+\bar{J}\int_0^xu_{\alpha\beta}^{-2}dy\bigg]\bigg(\ln\frac{u_{\alpha\beta}^2}{\underbar{n}}\bigg)_-^{2k-1}dx\notag\\
=&\int_0^1-\bigg[F(u_{\alpha\beta}^2,\bar{J},q)-q\ln\underbar{n}+q\ln\underbar{n}-F(\md_{l},\bar{J},\theta_l)-G[u_{\alpha\beta}^{2}]\notag\\
&\qquad\qquad\qquad\qquad\qquad\qquad-\int_0^xq_{x}\ln u_{\alpha\beta}^2dy+\bar{J}\int_0^xu_{\alpha\beta}^{-2}dy\bigg]\bigg(\ln\frac{u_{\alpha\beta}^2}{\underbar{n}}\bigg)_-^{2k-1}dx\notag\\
=&-\int_0^1q\bigg(\ln\frac{u_{\alpha\beta}^2}{\underbar{n}}\bigg)_-^{2k}dx\notag\\
&+\int_0^1\Bigg[\phi_rx+\int_0^1G(x,y)(u_{\alpha\beta}^2-D)(y)dy-\bar{J}\int_0^xu_{\alpha\beta}^{-2}dy+\frac{\bar{J}^2}{2n_l}\notag\\
&\qquad\qquad+\theta_l\ln n_l-q\ln\underbar{n}+\theta_l-q+\int_0^xq_x\ln u_{\alpha\beta}^2dy-\frac{\bar{J}^2}{2u_{\alpha\beta}^4}\Bigg]\bigg(\ln\frac{u_{\alpha\beta}^2}{\underbar{n}}\bigg)_-^{2k-1}dx\notag\\
\leq&-\int_0^1\frac{1}{2}\theta_L\bigg(\ln\frac{u_{\alpha\beta}^2}{\underbar{n}}\bigg)_-^{2k}dx-\int_0^1\big(C(N_1)\delta+B^2+|D|_0\big)\bigg(\ln\frac{u_{\alpha\beta}^2}{\underbar{n}}\bigg)_-^{2k-1}dx\notag\\
\leq&-\int_0^1\frac{1}{2}\theta_L\bigg(\ln\frac{u_{\alpha\beta}^2}{\underbar{n}}\bigg)_-^{2k}dx-\int_0^1\big(B^2+2|D|_0\big)\bigg(\ln\frac{u_{\alpha\beta}^2}{\underbar{n}}\bigg)_-^{2k-1}dx\qquad\text{if}\ \delta\ll1 \notag\\
=&-\int_0^1\frac{1}{2}\theta_L\bigg(\ln\frac{u_{\alpha\beta}^2}{\underbar{n}}\bigg)_-^{2k}dx-\int_0^1\big(B^2+2|D|_0\big)\frac{2}{\theta_L}\frac{\theta_L}{2}\bigg(\ln\frac{u_{\alpha\beta}^2}{\underbar{n}}\bigg)_-^{2k-1}dx\notag\\
=&-\int_0^1\frac{1}{2}\theta_L\bigg(\ln\frac{u_{\alpha\beta}^2}{\underbar{n}}\bigg)_-^{2k}dx+\int_0^1\frac{1}{2}\theta_L\underbrace{\Bigg[-\bigg(\ln\frac{u_{\alpha\beta}^2}{\underbar{n}}\bigg)_-\Bigg]^{2k-1}\frac{2\big(B^2+2|D|_0\big)}{\theta_L}}_{\text{by Young inequality}}dx\notag\\
\leq&-\int_0^1\frac{1}{2}\theta_L\bigg(\ln\frac{u_{\alpha\beta}^2}{\underbar{n}}\bigg)_-^{2k}dx\notag\\
&\quad+\int_0^1\frac{1}{2}\theta_L\Bigg\{\frac{2k-1}{2k}\Bigg[-\bigg(\ln\frac{u_{\alpha\beta}^2}{\underbar{n}}\bigg)_-\Bigg]^{(2k-1)\cdot\frac{2k}{2k-1}}+\frac{1}{2k}\bigg[\frac{2\big(B^2+2|D|_0\big)}{\theta_L}\bigg]^{2k}\Bigg\}dx\notag\\
=&\frac{\theta_L}{4k}\Bigg\{-\int_0^1\bigg(\ln\frac{u_{\alpha\beta}^2}{\underbar{n}}\bigg)_-^{2k}dx+\bigg[\frac{2\big(B^2+2|D|_0\big)}{\theta_L}\bigg]^{2k}\Bigg\}. \label{66.3}
\end{align}
Inserting \eqref{64.1} and \eqref{66.3} into \eqref{63.1}, we have the estimate
\begin{equation}\label{67.1}
\bigg\|\bigg(\ln\frac{u_{\alpha\beta}^2}{\underbar{n}}\bigg)_-\bigg\|_{L^{2k}(\Omega)}\leq\frac{2\big(B^2+2|D|_0\big)}{\theta_L},\quad k=1,2,3,\cdots.
\end{equation}
Let $k\rightarrow\infty$ in \eqref{67.1}, we immediately obtain
\begin{equation}\label{67.2}
\bigg\|\bigg(\ln\frac{u_{\alpha\beta}^2}{\underbar{n}}\bigg)_-\bigg\|_{L^{\infty}(\Omega)}\leq\frac{2\big(B^2+2|D|_0\big)}{\theta_L}.
\end{equation}
Note that $[\ln(u_{\alpha\beta}^2/\underbar{n})]_-$ is nonpositive, then the estimate \eqref{67.2} implies that 
\begin{equation}\label{68.1}
u_{\alpha\beta}\geq\sqrt{\underbar{n}}e^{-(B^2+2|D|_0)/\theta_L}\geq b,\quad\text{if}\ \delta\ll1.
\end{equation}
Combining \eqref{62.1} with \eqref{68.1}, we have $b\leq u_{\alpha\beta}\leq B$, which means $u_{\alpha\beta}=u$. This gives the existence result of the problem $(P1)$. Differentiating the equation \eqref{109.1a} and using the regularity $u\in H^3(\Omega)$ to obtain the desired regularity $u\in H^4(\Omega)$ in Claim 1.

Before proving the uniqueness result of the problem $(P1)$, we need to establish the uniform estimate \eqref{eu} with respect to $\pc\in(0,1]$ for any $H^4$-solution $u$ of $(P1)$. To this end, the estimate \eqref{eu-a} can be proved similarly as the above derivation of $u_{\alpha\beta}=u$. 

Furthermore, performing the procedure
\begin{equation}\label{117.1}
\int_0^1\bigg[\eqref{109.1a}\times\frac{1}{u}\bigg]_x\times u_xdx,
\end{equation}
we have
\begin{equation}\label{117.1r}
\int_0^1\pc^2\bigg(\frac{u_{xx}}{u}\bigg)_xu_xdx=\int_0^1\bigg[2\bigg(q-\frac{J^2}{u^4}\bigg)\frac{u_x}{u}+q_x-\varphi_x+\frac{J}{u^2}\bigg]u_xdx,
\end{equation}
where $J:=J[u^2,q]$ and $\varphi(x):=G[u^2](x)$. The left-side of \eqref{117.1r} can be estimated by using integration by parts,
\begin{equation}\label{117.1r-l}
\eqref{117.1r}_l=-\int_0^1\pc^2\frac{(u_{xx})^2}{u}dx\leq0.
\end{equation}
Based on the estimates \eqref{eu-a}, \eqref{154.2} and $|J|\leq C(b,B,N_1)\delta$, the right-side of \eqref{117.1r} can be estimated as follows provided $\delta$ is small enough,
\begin{equation}\label{117.1r-r}
\eqref{117.1r}_r\geq\frac{\theta_L}{2B}\|u_x\|^2-C(B,\theta_L)\|(q_x,\varphi_x)\|^2-J\bigg(\frac{1}{\rmd_{r}}-\frac{1}{\rmd_{l}}\bigg)\geq\frac{\theta_L}{2B}\|u_x\|^2-C,
\end{equation}
where the positive constant $C$ only depends on $n_l$, $\theta_L$ and $|D|_0$. We have used the elliptic estimate $\|\varphi\|_2\leq C(\|u^2-D\|+\|\phi_rx\|)$ in the last inequality of \eqref{117.1r-r}. Inserting \eqref{117.1r-l} and \eqref{117.1r-r} into \eqref{117.1r}, we get
\begin{equation}\label{117.3}
\|u_x\|\leq C,
\end{equation}
where the positive constant $C$ only depends on $n_l$, $\theta_L$ and $|D|_0$ and is independent of $\pc\in(0,1]$.

Performing the procedure
\begin{equation}\label{118.0}
\int_0^1\bigg[\eqref{109.1a}\times\frac{1}{u}\bigg]_x\times\bigg(\frac{u_{xx}}{u}\bigg)_xdx,
\end{equation}
and using integration by parts, we get
\begin{multline}\label{80.1}
\int_0^1\pc^2\bigg[\bigg(\frac{u_{xx}}{u}\bigg)_x\bigg]^2dx+\int_0^12S\bigg(\frac{u_{xx}}{u}\bigg)^2dx\\
=-\int_0^1\bigg(\frac{2S}{u}\bigg)_x\frac{u_xu_{xx}}{u}dx-\int_0^1q_{xx}\frac{u_{xx}}{u}dx+\int_0^1(u^2-D)\frac{u_{xx}}{u}dx+\int_0^1\frac{2J}{u^4}u_xu_{xx}dx,
\end{multline}
where $S:=q-J^2/u^4$. The left-side of \eqref{80.1} can be estimated as
\begin{equation}\label{80.2}
\eqref{80.1}_l\geq\frac{\theta_L}{2B^2}\|u_{xx}\|^2.
\end{equation}
The right-side of \eqref{80.1} can be estimated by H\"older, Sobolev and Cauchy-Schwarz inequalities as
\begin{align}
\eqref{80.1}_r&\leq C|u_x|_0\|u_x\|\|u_{xx}\|+C(\|q_{xx}\|+\|u^2-D\|+\|u_x\|)\|u_{xx}\|\notag\\
&\leq C(\|u_x\|^2+2\|u_x\|\|u_{xx}\|)^{1/2}\|u_{xx}\|+C\|u_{xx}\|\notag\\
&\leq\mu\|u_{xx}\|^2+C_\mu(\|u_x\|^2+2\|u_x\|\|u_{xx}\|)+C\|u_{xx}\|\notag\\
&\leq\mu\|u_{xx}\|^2+C_\mu(1+\|u_{xx}\|)\label{81.3}
\end{align}
where we have used the estimate \eqref{117.3}. The positive constant $C_\mu$ only depends on $n_l$, $\theta_L$, $|D|_0$ and $\mu$, where $\mu$ is a small number and will be determined later. Inserting \eqref{80.2} and \eqref{81.3} into \eqref{80.1}, and let $\mu\ll1$, we obtain $\|u_{xx}\|^2\leq C(1+\|u_{xx}\|)$. Solving this inequality with respect to $\|u_{xx}\|$ to obtain the estimate
\begin{equation}\label{82.4}
\|u_{xx}\|\leq C,
\end{equation}
where the positive constant $C$ only depends on $n_l$, $\theta_L$ and $|D|_0$ and is independent of $\pc\in(0,1]$.

Substituting the uniform estimates \eqref{eu-a}, \eqref{117.3} and \eqref{82.4} in the equality \eqref{80.1}, we have the estimate $\pc\|(u_{xx}/u)_x\|\leq C$. Note that
\[
\pc\pd{x}{3}u=\pc u\bigg(\frac{u_{xx}}{u}\bigg)_x+\pc\frac{u_xu_{xx}}{u}.
\]
We immediately get the following uniform estimate
\begin{equation}\label{85.2}
\|\pc\pd{x}{3}u\|\leq C.
\end{equation}

Furthermore, applying $\pd{x}{2}$ to the equation \eqref{109.1a} and taking the $L^2$-norm of the resultant equality, we finally have the following uniform estimate
\begin{equation}\label{86.1}
\|\pc^2\pd{x}{4}u\|\leq C.
\end{equation}

From the estimates \eqref{eu-a}, \eqref{117.3}, \eqref{82.4}, \eqref{85.2} and \eqref{86.1}, we have established the desired uniform estimate \eqref{eu-b} with respect to $\pc\in(0,1]$ for any strong solution $u$ to $(P1)$.

Based on the uniform estimate \eqref{eu}, now we can prove the uniqueness of solution to $(P1)$ by the energy method. To this end, we assume that $u_1$ and $u_2$ are two solutions to $(P1)$. Let $z_i:=\ln u_i^2$, $J_i:=J[e^{z_i},q]$, $S_i:=q-J_i^2/e^{2z_i}$, $\varphi_i:=G[e^{z_i}]$, $i=1,2$. Taking the difference $J_1$ and $J_2$, and applying the mean value theorem and \eqref{eu-a} to the explicit formula \eqref{104.1}, we have
\begin{equation}\label{138.0}
|J_i|\leq C(b,B,N_1)\delta,\qquad |J_1-J_2|\leq C\delta\|z_x\|,
\end{equation}
where $z:=z_1-z_2$ and $C$ is a positive constant which only depends on $n_l$, $\theta_L$ and $|D|_0$. Due to the procedure 
\begin{equation}
\bigg[\eqref{109.1a}\times\frac{1}{u_1}\bigg]_x-\bigg[\eqref{109.1a}\times\frac{1}{u_2}\bigg]_x
\end{equation}
and the transformation $u_i=e^{z_i/2}$, the difference $z$ satisfies
\begin{equation}\label{95.3}
-\bigg(\frac{J_1^2}{e^{2z_1}}-\frac{J_2^2}{e^{2z_2}}\bigg)z_{1x}+S_2z_x-\frac{\pc^2}{2}\bigg[z_{xx}+\frac{z_{1x}^2}{2}-\frac{z_{2x}^2}{2}\bigg]_x=(\varphi_1-\varphi_2)_x-\bigg(\frac{J_1}{e^{z_1}}-\frac{J_2}{e^{z_2}}\bigg).
\end{equation}
Multiplying \eqref{95.3} by $z_x$, integrating the resultant equality and using the boundary conditions
\begin{equation}\label{94.5}
z_i(0)=\ln\md_l,\quad z_i(1)=\ln\md_r,\quad \bigg(z_{ixx}+\frac{z_{ix}^2}{2}\bigg)(0)=\bigg(z_{ixx}+\frac{z_{ix}^2}{2}\bigg)(1)=0
\end{equation} 
to obtain that
\begin{multline}\label{96.1}
\int_0^1S_2z_x^2dx+\int_0^1\frac{\pc^2}{2}z_{xx}^2dx+\int_0^1\underbrace{(e^{z_1}-e^{z_2})z}_{\geq0}dx\\
=\int_0^1\bigg(\frac{J_1^2}{e^{2z_1}}-\frac{J_2^2}{e^{2z_2}}\bigg)z_{1x}z_xdx-\int_0^1\frac{\pc^2(z_{1x}+z_{2x})}{4}z_xz_{xx}dx-\int_0^1\bigg(\frac{J_1}{e^{z_1}}-\frac{J_2}{e^{z_2}}\bigg)z_xdx.
\end{multline}
The left-side of \eqref{96.1} can be estimated as
\begin{equation}\label{96.2}
\eqref{96.1}_l\geq\frac{\theta_L}{4}\|z_x\|^2+\frac{\pc^2}{2}\|z_{xx}\|^2,
\end{equation}
where we have used the estimates \eqref{154.2}, \eqref{eu} and \eqref{138.0}. The right-side of \eqref{96.1} can be estimated by H\"older, Poincar\'e and Cauchy-Schwarz inequalities and by the estimates \eqref{eu} and \eqref{138.0} as
\begin{align}
\eqref{96.1}_r&\leq C\Big(|J_1-J_2|\|z_x\|+|J_2|\|z\|\|z_x\|\Big)+C\pc^2\|z_x\|\|z_{xx}\|\notag\\
&\leq C(B,b,N_1)\delta\|z_x\|^2+\frac{\pc^2}{4}\|z_{xx}\|^2+C\pc^2\|z_x\|^2\notag\\
&=\Big(C(B,b,N_1)\delta+C\pc^2\Big)\|z_x\|^2+\frac{\pc^2}{4}\|z_{xx}\|^2.\label{97.3}
\end{align}
Substituting \eqref{96.2} and \eqref{97.3} in \eqref{96.1}, we see from letting $\delta$ and $\pc$ small enough that $\|z\|^2\leq0$. Thus we have shown $u_1\equiv u_2$. This complete the proof of Claim 1.

\emph{Step 4. Proof of Claim 2.}
For given function pair $(u,q)$ in the Claim 1, we discuss the unique solvability of the problem $(P2)$ by Leray-Schauder fixed-point theorem and energy method again. To this end, we define a fixed-point mapping $\mathcal{T}_2: q_1\mapsto Q_1$ over $H^1(\Omega)$ by solving the linear problem,
\begin{subequations}\label{126.1}
\begin{numcases}{}
\frac{2}{3}Q_{1xx}-JQ_{1x}+\frac{2}{3}J_{1*}(\ln u^2)_x\theta_L+\frac{2}{3}J(\ln u^2)_x(q_1-\theta_L)\notag\\
\qquad\qquad\qquad\qquad\qquad\qquad\ \,-u^2(Q_1-\theta_L)=g(u,q;\pc), \quad x\in\Omega,\label{126.1a}\\
Q_1(0)=\theta_l,\quad Q_1(1)=\theta_r,
\end{numcases}
\end{subequations}
where $J:=J[u^2,q]$ and $J_{1*}:=2\big(\bar{b}+\int_0^1q_{1x}\ln u^2dx\big)K[u^2,q]^{-1}$. In fact, for given $(u,q)$ in the Claim 1 and $q_1\in H^1(\Omega)$, the linear problem \eqref{126.1} is uniquely solvable in $H^3(\Omega)$ owing to the standard theory of the elliptic equations. By using the standard argument, we can further show that the mapping $\mathcal{T}_2$ is continuous and compact from $H^1(\Omega)$ into itself. Hence, it is sufficient to show that there exists a positive constant $M_2$ such that $\|\Theta\|_1\leq M_2$ for any $\Theta\in\big\{f\in H^1(\Omega)\,|\,f=\lambda\mathcal{T}_2f, \ \forall\lambda\in[0,1] \big\}$. We may assume $\lambda>0$ as the case $\lambda=0$ is trivial. Namely, for the function $\Theta$ verifying
\begin{subequations}\label{se127.2}
\begin{numcases}{}
\frac{2}{3}\Theta_{xx}-J\Theta_{x}+\frac{2}{3}\lambda J_*(\ln u^2)_x\theta_L+\frac{2}{3}\lambda J(\ln u^2)_x(\Theta-\theta_L)\notag\\
\qquad\qquad\qquad\qquad\qquad\qquad\ \,-u^2(\Theta-\lambda\theta_L)=\lambda g(u,q;\pc), \quad x\in\Omega, \label{127.2a}\\
\Theta(0)=\lambda\theta_l,\quad \Theta(1)=\lambda\theta_r,\qquad\forall\lambda\in[0,1],\label{127.2b}\\
J:=J[u^2,q],\quad J_*:=2\bigg(\bar{b}+\int_0^1\Theta_x\ln u^2dx\bigg)K[u^2,q]^{-1},\notag
\end{numcases}
\end{subequations}
we need to show the estimate $\|\Theta\|_1\leq M_2$.

Substituting $\Theta_\lambda:=\Theta-\lambda\bar{\theta}$ into the equation \eqref{127.2a}, where $\bar{\theta}(x):=\theta_l(1-x)+\theta_rx$, multiplying the resultant equation by $\Theta_\lambda$ and integrating it by parts over the domain $\Omega$ give
\begin{align}
&\frac{2}{3}\|\Theta_{\lambda x}\|^2+b^2\|\Theta_\lambda\|^2\notag\\
\leq&-\frac{2}{3}\lambda\theta_LJ_*\int_0^1\Theta_{\lambda x}\ln u^2dx-\int_0^1\bigg[\frac{2}{3}\lambda J\ln u^2\big(\Theta_\lambda^2\big)_x+J\Theta_{\lambda x}\Theta_\lambda\bigg]dx\notag\\
&-\frac{2}{3}\lambda J\int_0^1\ln u^2\Big[(\lambda\bar{\theta}-\theta_L)\Theta_\lambda\Big]_xdx-\lambda\int_0^1\Big[J\bar{\theta}_x+u^2(\bar{\theta}-\theta_L)+g(u,q;\pc)\Big]\Theta_\lambda dx\notag\\ 
\leq&-\frac{4\lambda\theta_L}{3K[u^2,q]}\bigg(\bar{b}+\int_0^1\big(\Theta_{\lambda x}+\lambda\bar{\theta}_x\big)\ln u^2dx\bigg)\int_0^1\Theta_{\lambda x}\ln u^2dx+C(b,B,N_1)\delta\|\Theta_{\lambda}\|_1^2\notag\\
&+\mu\|\Theta_{\lambda}\|_1^2+C(\mu,b,B,N_1)\delta^2\big(\delta^2+\|\lambda\bar{\theta}-\theta_L\|^2\big)+C(\mu,b,B)\delta^2\notag\\
\leq&\underbrace{-\frac{4\lambda\theta_L}{3K[u^2,q]}\bigg(\int_0^1\Theta_{\lambda x}\ln u^2dx\bigg)^2}_{\leq0}+\big[\mu+C(b,B,N_1)\delta\big]\|\Theta_{\lambda}\|_1^2\notag\\
&+C(\mu,b,B,N_1)\delta^2\big(\delta^2+\|\lambda\bar{\theta}-\theta_L\|^2\big)+C(\mu,b,B,\theta_L)\delta^2\notag\\
\leq&\big[\mu+C(b,B,N_1)\delta\big]\|\Theta_{\lambda}\|_1^2+C(\mu,b,B,N_1)\delta^2\big(\delta^2+\|\lambda\bar{\theta}-\theta_L\|^2\big)\notag\\
&+C(\mu,b,B,\theta_L)\delta^2,\label{133.0}
\end{align}
where we have used the expression \eqref{103}, the estimates \eqref{eu} and $|J|\leq C(b,B,N_1)\delta$, and the Young inequality. Taking $\mu$ and $\delta$ small enough in \eqref{133.0}, we have
\begin{equation}\label{133.3}
\|\Theta_{\lambda}\|_1^2\leq C(b,B,\theta_L)\delta^2+C(b,B,N_1)\delta^2\big(\delta^2+\|\lambda\bar{\theta}-\theta_L\|^2\big),
\end{equation}
which immediately means that
\begin{align}
\|\Theta\|_1=&\|\Theta_\lambda+\lambda\bar{\theta}\|_1\leq\|\Theta_\lambda\|_1+\lambda\|\bar{\theta}\|_1\notag\\
\leq&\sqrt{C(b,B,\theta_L)+C(b,B,N_1)(1+\theta_L^2)}+2\theta_L=:M_2.\label{es134.1}
\end{align}
Thus, the mapping $\mathcal{T}_2$ has a fixed point $Q=\mathcal{T}_2Q\in H^3(\Omega)$ by Leray-Schauder fixed-point theorem and the elliptic regularity theory. Hence, we have shown the existence of the solution $Q$ to the problem $(P2)$.

The uniqueness of the solution $Q$ follows from the energy method. Let $Q_i\in H^1(\Omega)$, $i=1,2$ be two solutions to $(P2)$ corresponding to the same function pair $(u,q)$. Define $\bar{Q}:=Q_1-Q_2$, which satisfies
\begin{subequations}\label{135.2}
\begin{numcases}{}
\frac{2}{3}\bar{Q}_{xx}-J\bar{Q}_{x}+\frac{4\theta_L}{3K[u^2,q]}\bigg(\int_0^1\bar{Q}_x\ln u^2dx\bigg)(\ln u^2)_x\notag\\
\qquad\qquad\qquad\qquad\qquad\qquad\ \,+\frac{2}{3}J(\ln u^2)_x\bar{Q}-u^2\bar{Q}=0, \quad x\in\Omega,\label{135.2a}\\
\bar{Q}(0)=\bar{Q}(1)=0.\label{135.2b}
\end{numcases}
\end{subequations}
Multiplying the equation \eqref{135.2a} by $-\bar{Q}$ and integrating the resultant equality over $\Omega$. In a similar way as the derivation of \eqref{133.0}, we have
\begin{align}
\frac{2}{3}\|\bar{Q}_x\|^2+b^2\|\bar{Q}\|^2\leq&-\frac{4\theta_L}{3K[u^2,q]}\bigg(\int_0^1\bar{Q}_x\ln u^2dx\bigg)^2+C(b,B,N_1)\delta\|\bar{Q}\|_1^2\notag\\
\leq&C(b,B,N_1)\delta\|\bar{Q}\|_1^2.\label{136.0}
\end{align}
Taking $\delta$ small enough in \eqref{136.0}, we see that $\|\bar{Q}\|_1^2\leq0$. Thus we have proven $Q_1\equiv Q_2$.

On the other hand, letting $\Theta=Q$ and $\lambda=1$ in the estimate \eqref{133.3}, we obtain
\begin{align}
\|Q-\theta_L\|_1\leq&\|Q-\bar{\theta}\|_1+\|\bar{\theta}-\theta_L\|_1\notag\\
\leq&C(b,B,\theta_L)\delta+C(b,B,N_1)\delta\big(\delta+\|\bar{\theta}-\theta_L\|\big)+\|\bar{\theta}-\theta_L\|_1\notag\\
\leq&C(b,B,\theta_L)\delta+C(b,B,N_1)\delta^2\notag\\
=&C_1\delta+C_2(b,B,N_1)\delta^2,\label{133.5}
\end{align}
which exactly is the desired estimate \eqref{eQa}. Solving the equation $\pd{x}{k}\eqref{109.2a}$ with respect to $\pd{x}{k}Q_{xx}$ for $k=0,1$ and directly taking the $L^2$-norm, we get the desired estimates \eqref{eQb} and \eqref{eQc} with the aid of the estimates \eqref{eu}, \eqref{133.5} and $|J|, |J_*|\leq C(b,B,N_1)\delta$. Consequently, the proof of Claim 2 is completed.

\emph{Step 5. End of the proof.} 
Firstly, based on the estimate \eqref{eQ} we can determine the constants $N_1$ and $N_2$ by letting
\begin{equation}\label{N12}
N_1:=2C_1,\quad N_2:=C_3(b,B,2C_1).
\end{equation}
If $\delta$ is small enough, that is,
\begin{equation*}
\delta\leq\frac{C_1}{C_2(b,B,2C_1)},
\end{equation*}
then we see from the estimate \eqref{eQ} that $\mathcal{T}$ maps $\mathcal{U}[N_1,N_2]$ into itself. Combining the estimates \eqref{eu} and \eqref{eQ} with the Sobolev compact embedding theorem, via a standard argument, we see that the mapping $\mathcal{T}$ is continuous in the norm of $C^2(\overline{\Omega})$ and the image $\mathcal{T}\big(\mathcal{U}[N_1,N_2]\big)$ is precompact in $C^2(\overline{\Omega})$. Therefore, applying the Schauder fixed-point theorem to the mapping $\mathcal{T}:\mathcal{U}[N_1,N_2]\rightarrow\mathcal{U}[N_1,N_2]$, we obtain a fixed-point $\swnd\in\mathcal{U}[N_1,N_2]$ of the mapping $\mathcal{T}$. According to the construction of the mapping $\mathcal{T}$ above, we can easily see that $(\rsmd:=u[\swnd],\swnd)$ is a desired solution to the BVP \eqref{103.2}$\sim$\eqref{103.3}.

In addition, the solution $(\smd,\sdl,\swnd,\sdws)$ to the original BVP \eqref{1dsfqhd}$\sim$\eqref{sbc} is constructed from the the solution $(\rsmd,\swnd)$ to the BVP \eqref{103.2}$\sim$\eqref{103.3}. In fact, we define a function $\smd:=\rsmd^2$, the constant $\sdl:=J[\rsmd^2,\swnd]$ and a function $\sdws:=G[\rsmd^2]$, where $J[\cdot,\cdot]$ and $G[\cdot]$ are given in \eqref{104.1} and \eqref{104.2}. Then, we see that $(\smd,\sdl,\swnd,\sdws)\in H^4(\Omega)\times H^4(\Omega)\times H^3(\Omega)\times C^2(\overline{\Omega})$ is a desired solution to the BVP \eqref{1dsfqhd}$\sim$\eqref{sbc}. Moreover, this stationary solution satisfies the condition \eqref{psc} and the estimate \eqref{145.1}, thanks to the estimates \eqref{eu} and \eqref{eQ}.

Finally, using the same methods in Step 3 and Step 4, we can prove the local uniqueness of the stationary solution $(\smd,\sdl,\swnd,\sdws)$ to the BVP \eqref{1dsfqhd}$\sim$\eqref{sbc} if the parameters $\delta$ and $\pc$ are small enough and the solution additionally satisfies \eqref{psc} and \eqref{145.1}. The computations are standard but tedious, we omit the details.
\end{proof}

\section{Asymptotic stability of the stationary solution}\label{Sect.3}
In this section, we show Theorem \ref{thm2} by applying the standard continuation argument based on the local existence result and the uniform a priori estimate. To simplify the notations, we remove the superscript $\pc$ and denote the solution $(\pcmd{\pc},\pcdl{\pc},\pcwnd{\pc},\pcdws{\pc})$ in Theorem \ref{thm2} as $(\md,\dl,\wnd,\dws)$.

\subsection{Local existence}\label{Subsect.3.1}
In this subsection, we discuss the existence of the local-in-time solution. The proof is based on the iteration method and the energy estimates.

It is also convenient to make use of the transformation $\rmd:=\sqrt{\md}$ in the IBVP \eqref{1dfqhd}$\sim$\eqref{bc}. Then, we derive the equivalent IBVP for $(\rmd,\dl,\wnd,\dws)$ as follows
\begin{subequations}\label{a2.2}
\begin{numcases}{}
2\rmd\rmd_t+\dl_x=0, \label{a2.2-1}\\
\dl_t+2S[\rmd^2,\dl,\wnd]\rmd\rmd_x+\frac{2\dl}{\rmd^2}\dl_x+\rmd^2\wnd_x-\pc^2\rmd^2\Bigg(\frac{\rmd_{xx}}{\rmd}\Bigg)_x=\rmd^2\dws_x-\dl, \label{a2.2-2}\\
\rmd^2\wnd_t+\dl\wnd_x+\frac{2}{3}\rmd^2\wnd\Bigg(\frac{\dl}{\rmd^2}\Bigg)_x-\frac{2}{3}\wnd_{xx}-\frac{\pc^2}{3}\Bigg[\rmd^2\Bigg(\frac{\dl}{\rmd^2}\Bigg)_{xx}\Bigg]_x=\frac{1}{3}\frac{\dl^2}{\rmd^2}-\rmd^2(\wnd-\theta_{L}), \label{a2.2-3}\\
\dws_{xx}=\rmd^2-D(x), \qquad\forall t>0,\ \forall x\in\Omega,\label{a2.2-4}
\end{numcases}
\end{subequations}
with the initial condition
\begin{equation}\label{a2.3}
(\rmd,\dl,\wnd)(0,x)=(\rmd_0,\dl_0,\wnd_0)(x),\quad \rmd_0:=\sqrt{\md_0},
\end{equation}
and the boundary conditions
\begin{subequations}\label{a2.4}
\begin{gather}
\rmd(t,0)=\rmd_{l},\qquad \rmd(t,1)=\rmd_{r},\label{a2.4-1}\\
\rmd_{xx}(t,0)=\rmd_{xx}(t,1)=0,\label{a2.4-2}\\
\wnd(t,0)=\wnd_{l},\qquad \wnd(t,1)=\wnd_{r},\label{a2.4-3}\\
\dws(t,0)=0,\qquad \dws(t,1)=\phi_r.\label{a2.4-4}
\end{gather}
\end{subequations}

In the following discussion, we borrow the ideas in the papers \cite{NS08,NS09} which have shown the local existence theorems for the isothermal QHD model and the FHD model. Also, see \cite{KNN99,KNN03} for the general hyperbolic-elliptic coupled systems.

Now we are in the position to state the local existence.
\begin{lemma}\label{lema2}
Suppose that the initial data $(\rmd_0,\dl_0,\wnd_0)\in H^4(\Omega)\times H^3(\Omega)\times H^2(\Omega)$ and the boundary data satisfy the compatible condition
\begin{gather}
\rmd_0(0)=\rmd_l,\quad \rmd_0(1)=\rmd_r,\quad \wnd_0(0)=\theta_l,\quad \wnd_0(1)=\theta_r,\notag\\
\dl_{0x}(0)=\dl_{0x}(1)=\rmd_{0xx}(0)=\rmd_{0xx}(1)=0\label{rcompatibility}
\end{gather}
and the condition
\begin{equation}\label{ripsc}
\inf_{x\in\Omega}\rmd_{0}>0,\qquad \inf_{x\in\Omega}\wnd_0>0, \qquad\inf_{x\in\Omega}S[\rmd_{0}^2,\dl_{0},\wnd_0]>0.
\end{equation}
Then there exists a constant $T_*>0$ such that the IBVP \eqref{a2.2}$\sim$\eqref{a2.4} has a unique solution $(\rmd,\dl,\wnd,\dws)\in\big[\mathfrak{Y}_4([0,T_*])\cap H^2(0,T_*;H^1(\Omega))\big]\times\big[\mathfrak{Y}_3([0,T_*])\cap H^2(0,T_*;L^2(\Omega))\big]\times\big[\mathfrak{Y}_2([0,T_*])\cap H^1(0,T_*;H^1(\Omega))\big]\times\mathfrak{Z}([0,T_*])$ satisfying
\begin{equation}\label{rpsc}
\inf_{x\in\Omega}\rmd>0,\qquad \inf_{x\in\Omega}\wnd>0, \qquad\inf_{x\in\Omega}S[\rmd^2,\dl,\wnd]>0.
\end{equation}
\end{lemma}
To show Lemma \ref{lema2}, we first study the linear IBVP for the unknowns $(\hat{\rmd},\hat{\dl},\hat{\wnd})$
\begin{subequations}\label{a3.1}
\begin{numcases}{}
2\rmd\hat{\rmd}_t+\hat{\dl}_x=0, \label{a3.1-1}\\
\hat{\dl}_t+2S[\rmd^2,\dl,\wnd]\rmd\hat{\rmd}_x+\frac{2\dl}{\rmd^2}\hat{\dl}_x+\rmd^2\hat{\wnd}_x-\pc^2\rmd^2\Bigg(\frac{\hat{\rmd}_{xx}}{\rmd}\Bigg)_x=\rmd^2\dws_x-\dl, \label{a3.1-2}\\
\rmd^2\hat{\wnd}_t+\dl\hat{\wnd}_x+\frac{2}{3}\Bigg(\frac{\dl}{\rmd^2}\Bigg)_x\rmd^2\hat{\wnd}-\frac{2}{3}\hat{\wnd}_{xx}-\frac{\pc^2}{3}\Bigg[\rmd^2\Bigg(\frac{\dl}{\rmd^2}\Bigg)_{xx}\Bigg]_x=\frac{1}{3}\frac{\dl^2}{\rmd^2}-\rmd^2(\hat{\wnd}-\theta_{L}), \label{a3.1-3}\\
\dws:=\Phi[\rmd^2], \qquad\forall t>0,\ \forall x\in\Omega,\label{a3.1-4}
\end{numcases}
\end{subequations}
with the initial condition
\begin{equation}\label{a3.2}
(\hat{\rmd},\hat{\dl},\hat{\wnd})(0,x)=(\rmd_0,\dl_0,\wnd_0)(x),
\end{equation}
and the boundary conditions
\begin{subequations}\label{a3.3}
\begin{gather}
\hat{\rmd}(t,0)=\rmd_{l},\qquad \hat{\rmd}(t,1)=\rmd_{r},\label{a3.3-1}\\
\hat{\rmd}_{xx}(t,0)=\hat{\rmd}_{xx}(t,1)=0,\label{a3.3-2}\\
\hat{\wnd}(t,0)=\wnd_{l},\qquad \hat{\wnd}(t,1)=\wnd_{r},\label{a3.3-3}
\end{gather}
\end{subequations}
where the function $\dws$ is defined by \eqref{efep}. Let the functions $(\rmd,\dl,\wnd)$ in the coefficients in \eqref{a3.1} satisfy
\begin{subequations}\label{a4.1}
\begin{gather}
(\rmd,\dl,\wnd)(0,x)=(\rmd_0,\dl_0,\wnd_0)(x),\quad\forall x\in\Omega, \label{a4.1-1}\\
\rmd\in\mathfrak{Y}_4([0,T])\cap H^2(0,T;H^1(\Omega)),\quad \dl\in\mathfrak{Y}_3([0,T])\cap H^2(0,T;L^2(\Omega)),\notag\\
\wnd\in\mathfrak{Y}_2([0,T])\cap H^1(0,T;H^1(\Omega)),\label{a4.1-2}\\
\rmd(t,x),\  \wnd(t,x),\  S[\rmd^2,\dl,\wnd](t,x)\geq m,\quad\forall (t,x)\in[0,T]\times\Omega,\label{a4.1-3}\\
\|\rmd(t)\|_4^2+\|\dl(t)\|_3^2+\|(\rmd_t,\wnd)(t)\|_2^2+\|\dl_t(t)\|_1^2+\|(\rmd_{tt},\wnd_t)(t)\|^2\notag\\
+\int_0^t\|(\rmd_{ttx},\dl_{tt},\wnd_{tx})(\tau)\|^2d\tau\leq M,\quad\forall t\in[0,T],\label{a4.1-4}
\end{gather}
\end{subequations}
where $T$, $m$ and $M$ are positive constants. We denote by $X(T;m,M)$ the set of functions $(\rmd,\dl,\wnd)$ satisfying \eqref{a4.1}, and we abbreviate $X(T;m,M)$ by $X(\cdot)$ without confusion. The property of $\dws$ is that
\[
\dws\in\mathfrak{Z}([0,T]),\quad \|\pd{t}{i}\dws(t)\|_2^2\leq M,\quad\forall t\in[0,T], \ i=0,1,2.
\]
Then the next lemma means that for suitably chosen constants $T$, $m$ and $M$, the set $X(\cdot)$ is invariant under the mapping $(\rmd,\dl,\wnd)\mapsto(\hat{\rmd},\hat{\dl},\hat{\wnd})$ defined by solving the linear IBVP \eqref{a3.1}$\sim$\eqref{a3.3}. We discuss the solvability of this linear problem in Appendix. Since the next lemma is proved similarly as in \cite{NS08,NS09}, we omit the proof.
\begin{lemma}\label{lema3}
Under the same assumptions in Lemma \ref{lema2}, there exist positive constants $T$, $m$ and $M$ with the following property: If $(\rmd,\dl,\wnd)\in X(\cdot)$, then the linear IBVP \eqref{a3.1}$\sim$\eqref{a3.3} admits a unique solution $(\hat{\rmd},\hat{\dl},\hat{\wnd})$ in the same set $X(\cdot)$.
\end{lemma}
Using Lemma \ref{lema3}, we can show Lemma \ref{lema2}.
\begin{proof}[\textbf{Proof of Lemma \ref{lema2}}]
We define the approximation sequence $\{(\rmd^k,\dl^k,\wnd^k)\}_{k=0}^\infty$ by letting $(\rmd^0,\dl^0,\wnd^0)=(\rmd_0,\dl_0,\wnd_0)$ and solving
\begin{subequations}\label{a5.1}
\begin{numcases}{}
2\rmd^k\rmd^{k+1}_t+\dl^{k+1}_x=0, \label{a5.1-1}\\
\dl^{k+1}_t+2S[(\rmd^k)^2,\dl^k,\wnd^k]\rmd^k\rmd^{k+1}_x+\frac{2\dl^k}{(\rmd^k)^2}\dl^{k+1}_x+(\rmd^k)^2\wnd^{k+1}_x\notag\\
\qquad\qquad\qquad\qquad\qquad\qquad\qquad\quad-\pc^2(\rmd^k)^2\Bigg(\frac{\rmd^{k+1}_{xx}}{\rmd^k}\Bigg)_x=(\rmd^k)^2\dws^k_x-\dl^k, \label{a5.1-2}\\
(\rmd^k)^2\wnd^{k+1}_t+\dl^k\wnd^{k+1}_x+\frac{2}{3}\Bigg(\frac{\dl^k}{(\rmd^k)^2}\Bigg)_x(\rmd^k)^2\wnd^{k+1}-\frac{2}{3}\wnd^{k+1}_{xx}\notag\\
\qquad\qquad\qquad\quad-\frac{\pc^2}{3}\Bigg[(\rmd^k)^2\Bigg(\frac{\dl^k}{(\rmd^k)^2}\Bigg)_{xx}\Bigg]_x=\frac{1}{3}\frac{(\dl^k)^2}{(\rmd^k)^2}-(\rmd^k)^2(\wnd^{k+1}-\theta_{L}), \label{a5.1-3}\\
\dws^k:=\Phi[(\rmd^k)^2], \qquad\forall t>0,\ \forall x\in\Omega,\label{a5.1-4}
\end{numcases}
\end{subequations}
with the initial condition
\begin{equation}\label{a5.2}
(\rmd^{k+1},\dl^{k+1},\wnd^{k+1})(0,x)=(\rmd_0,\dl_0,\wnd_0)(x),
\end{equation}
and the boundary conditions
\begin{subequations}\label{a5.3}
\begin{gather}
\rmd^{k+1}(t,0)=\rmd_{l},\qquad \rmd^{k+1}(t,1)=\rmd_{r},\label{a5.3-1}\\
\rmd^{k+1}_{xx}(t,0)=\rmd^{k+1}_{xx}(t,1)=0,\label{a5.3-2}\\
\wnd^{k+1}(t,0)=\wnd_{l},\qquad \wnd^{k+1}(t,1)=\wnd_{r},\label{a5.3-3}
\end{gather}
\end{subequations}

Thanks to Lemma \ref{lema3}, the sequence $\{(\rmd^k,\dl^k,\wnd^k)\}_{k=0}^\infty$ is well defined and contained in $X(\cdot)$. Consequently, $(\rmd^k,\dl^k,\wnd^k)$ satisfies the estimates \eqref{a4.1-3} and \eqref{a4.1-4}. Next, applying the standard energy method to the system satisfied by the difference $(\rmd^{k+1}-\rmd^k,\dl^{k+1}-\dl^k,\wnd^{k+1}-\wnd^k)$, we see that $\{(\rmd^k,\dl^k,\wnd^k)\}_{k=0}^\infty$ is the Cauchy sequence in $\mathfrak{Y}_2([0,T_*])\times\mathfrak{Y}_1([0,T_*])\times\big[\mathfrak{Y}_2([0,T_*])\cap H^1(0,T_*;H^1(\Omega))\big]$ for small enough $0<T_*\leq T$. In showing this fact, we obtain the estimates of the higher-order derivatives in the time variable $t$ and then rewrite them into those in the spatial variable $x$ by using the linear equations. Thus, there exists a function $(\rmd,\dl,\wnd)\in\mathfrak{Y}_2([0,T_*])\times\mathfrak{Y}_1([0,T_*])\times\big[\mathfrak{Y}_2([0,T_*])\cap H^1(0,T_*;H^1(\Omega))\big]$ such that $(\rmd^k,\dl^k,\wnd^k)\rightarrow(\rmd,\dl,\wnd)$ strongly in $\mathfrak{Y}_2([0,T_*])\times\mathfrak{Y}_1([0,T_*])\times\big[\mathfrak{Y}_2([0,T_*])\cap H^1(0,T_*;H^1(\Omega))\big]$ as $k\rightarrow\infty$. Moreover, it holds $(\rmd,\dl)\in\big[\mathfrak{Y}_4([0,T_*])\cap H^2(0,T_*;H^1(\Omega))\big]\times\big[\mathfrak{Y}_3([0,T_*])\cap H^2(0,T_*;L^2(\Omega))\big]$ by the standard argument (see \cite{NS08,NS09} for example). Define $\dws:=\Phi[\rmd^2]$ by the limit function $\rmd$ and the explicit formula \eqref{efep}, we see that $(\rmd,\dl,\wnd,\dws)$ is the desired solution to the IBVP \eqref{a2.2}$\sim$\eqref{a2.4}. Notice that this solution also satisfies \eqref{rpsc}.
\end{proof}

\subsection{A priori estimate}\label{Subsect.3.2}
To show the asymptotic stability of the stationary solution $(\rsmd,\sdl,\swnd,\sdws)$, we introduce the perturbations around the stationary solution $(\rsmd,\sdl,\swnd,\sdws)$ below
\begin{align}
\prmd(t,x)&:=\rmd(t,x)-\rsmd(x),&\pdl(t,x)&:=\dl(t,x)-\sdl,\notag\\
\pwnd(t,x)&:=\wnd(t,x)-\swnd(x),&\pdws(t,x)&:=\dws(t,x)-\sdws(x).\label{10.1}
\end{align}

Taking the difference between the transient system \eqref{1dfqhd} and the stationary system \eqref{1dsfqhd} via the following procedure
\begin{equation*}\label{10.0}
\eqref{1dfqhd1}-\eqref{1dsfqhd1},\quad\eqref{1dfqhd2}/\rmd^2-\eqref{1dsfqhd2}/\rsmd^2,\quad\eqref{1dfqhd3}-\eqref{1dsfqhd3},\quad\eqref{1dfqhd4}-\eqref{1dsfqhd4},
\end{equation*}
we can derive the perturbed system for the perturbations $(\prmd,\pdl,\pwnd,\pdws)$ as
\begin{subequations}\label{10.2}
\begin{numcases}{}
2(\prmd+\rsmd)\prmd_t+\pdl_x=0, \label{10.2a}\\
\bigg[\frac{\pdl+\sdl}{(\prmd+\rsmd)^2}\bigg]_t+\frac{1}{2}\Bigg\{\bigg[\frac{\pdl+\sdl}{(\prmd+\rsmd)^2}\bigg]^2-\bigg(\frac{\sdl}{\rsmd^2}\bigg)^2\Bigg\}_x+\pwnd\Big[\ln(\prmd+\rsmd)^2\Big]_x\notag\\
\quad+\swnd\Big[\ln(\prmd+\rsmd)^2-\ln\rsmd^2\Big]_x+\pwnd_x-\pc^2\bigg[\frac{(\prmd+\rsmd)_{xx}}{\prmd+\rsmd}-\frac{\rsmd_{xx}}{\rsmd}\bigg]_x\notag\\
\qquad\qquad\qquad\qquad\qquad\qquad\qquad\qquad\qquad\qquad=\pdws_x-\bigg[\frac{\pdl+\sdl}{(\prmd+\rsmd)^2}-\frac{\sdl}{\rsmd^2}\bigg], \label{10.2b}\\
(\prmd+\rsmd)^2\pwnd_t-\frac{2}{3}\pwnd_{xx}+\frac{2}{3}\swnd\pdl_x-\frac{4\sdl\swnd}{3\rsmd}\prmd_x\notag\\
\qquad\qquad\quad-\frac{\pc^2}{3}\Bigg\{(\prmd+\rsmd)^2\bigg[\frac{\pdl+\sdl}{(\prmd+\rsmd)^2}\bigg]_{xx}-\rsmd^2\bigg(\frac{\sdl}{\rsmd^2}\bigg)_{xx}\Bigg\}_x=H(t,x), \label{10.2c}\\
\pdws_{xx}=(\prmd+2\rsmd)\prmd,\label{10.2d}
\end{numcases}
\end{subequations}
where the right-side term of the perturbed energy equation \eqref{10.2c} is defined by
\begin{equation}\label{11.1}
H(t,x):=\frac{4\swnd(\prmd+\rsmd)_x}{3(\prmd+\rsmd)}\pdl-\rsmd^2\pwnd+H_1(t,x),
\end{equation}
and
\begin{align}
H_1(t,x):=&\frac{4\pwnd(\prmd+\rsmd)_x}{3(\prmd+\rsmd)}\pdl-(\prmd+2\rsmd)(\pwnd+\swnd-\theta_L)\prmd\notag\\
&-(\pwnd+\swnd)_x\pdl-\sdl\pwnd_x-\frac{2\pdl_x}{3}\pwnd+\frac{4\sdl\swnd(\prmd+\rsmd)_x}{3\rsmd^2}\prmd+\frac{4\sdl(\prmd+\rsmd)_x}{3(\prmd+\rsmd)}\pwnd\notag\\
&+\frac{4\sdl\swnd(\prmd+2\rsmd)(\prmd+\rsmd)_x}{3\rsmd^2(\prmd+\rsmd)}\prmd+\frac{\pdl+2\sdl}{3(\prmd+\rsmd)^2}\pdl-\frac{\sdl^2(\prmd+2\rsmd)}{3\rsmd^2(\prmd+\rsmd)^2}\prmd.\label{11.0}
\end{align}
The initial and the boundary conditions to the system \eqref{10.2} are derived from \eqref{ic}, \eqref{bc} and \eqref{sbc} as
\begin{gather}
\prmd(0,x)=\prmd_0(x):=\rmd_0(x)-\rsmd(x),\quad\pdl(0,x)=\pdl_0(x):=\dl_0(x)-\sdl,\notag\\
\pwnd(0,x)=\pwnd_0(x):=\wnd_0(x)-\swnd(x),\label{pic}
\end{gather}
and
\begin{subequations}\label{pbc}
\begin{gather}
\prmd(t,0)=\prmd(t,1)=0,\label{pbc1}\\
\prmd_{xx}(t,0)=\prmd_{xx}(t,1)=0,\label{pbc2}\\
\pwnd(t,0)=\pwnd(t,1)=0,\label{pbc3}\\
\pdws(t,0)=\pdws(t,1)=0.\label{pbc4}
\end{gather}
\end{subequations}

Theorem \ref{thm1} and Lemma \ref{lema2} ensure the local existence of the solution $(\prmd,\pdl,\pwnd,\pdws)$ to the IBVP \eqref{10.2}$\sim$\eqref{pbc}. It is summarized in the next corollary.
\begin{corollary}\label{cor1}
Suppose that $(\prmd_0,\pdl_0,\pwnd_0)\in H^4(\Omega)\times H^3(\Omega)\times H^2(\Omega)$ and $(\prmd_0+\rsmd,\pdl_0+\sdl,\pwnd_0+\swnd)$ satisfies \eqref{rcompatibility} and \eqref{ripsc}. Then there exists a constant $T_*>0$ such that the IBVP \eqref{10.2}$\sim$\eqref{pbc} has a unique solution $(\prmd,\pdl,\pwnd,\pdws)\in\big[\mathfrak{Y}_4([0,T_*])\cap H^2(0,T_*;H^1(\Omega))\big]\times\big[\mathfrak{Y}_3([0,T_*])\cap H^2(0,T_*;L^2(\Omega))\big]\times\big[\mathfrak{Y}_2([0,T_*])\cap H^1(0,T_*;H^1(\Omega))\big]\times\mathfrak{Y}_4^2([0,T_*])$ with the property that $(\prmd+\rsmd,\pdl+\sdl,\pwnd+\swnd)$ satisfies \eqref{rpsc}.
\end{corollary}
To show the global existence of the solution, the key step is to derive the a priori estimate \eqref{127.2} for the local solution in Corollary \ref{cor1}. The next three subsections are devoted to the proof of Proposition \ref{prop1}, where the following notations are frequently used.
\begin{equation}\label{16.1}
\pcN:=\sup_{t\in[0,T]}\pcn,\quad \pcn:=\|(\prmd,\pdl,\pwnd)(t)\|_2+\|(\pc\pd{x}{3}\prmd,\pc\pd{x}{3}\pdl,\pc^2\pd{x}{4}\prmd)(t)\|.
\end{equation}
\begin{proposition}\label{prop1}
Let $(\prmd,\pdl,\pwnd,\pdws)$ be a solution to the IBVP \eqref{10.2}$\sim$\eqref{pbc} which belongs to $\big[\mathfrak{Y}_4([0,T])\cap H^2(0,T;H^1(\Omega))\big]\times\big[\mathfrak{Y}_3([0,T])\cap H^2(0,T;L^2(\Omega))\big]\times\big[\mathfrak{Y}_2([0,T])\cap H^1(0,T;H^1(\Omega))\big]\times\mathfrak{Y}_4^2([0,T])$. Then there exist positive constants $\delta_0$, $C$ and $\gamma$ such that if $\pcN+\delta+\pc\leq\delta_0$, then the following estimate holds for $t\in[0,T]$,
\begin{equation}\label{127.2}
\pcn+\|\pdws(t)\|_4\leq Cn_\pc(0)e^{-\gamma t},
\end{equation}
where $C$ and $\gamma$ are two positive constants independent of $\delta$, $\pc$ and $T$.
\end{proposition}

Using the Sobolev inequality, the estimate \eqref{145.1}, the perturbed system \eqref{10.2} and the notation \eqref{16.1}, we can derive some frequently used estimates in the next lemma. Since the proof is straightforward and tedious, we omit the details.
\begin{lemma}
Under the same assumptions as in Proposition \ref{prop1}, the following estimates hold for $t\in[0,T]$,
\begin{gather}
|\rsmd|_1+\big|\big(\pc^{1/2}\rsmd_{xx},\pc^{3/2}\rsmd_{xxx}\big)\big|_0\leq C,\quad |\sdl|+|\swnd-\theta_L|_2\leq C\delta,\label{17.1}\\
|(\prmd,\pdl,\pwnd)(t)|_1+\big|\big(\pc^{1/2}\prmd_{xx},\pc^{1/2}\pdl_{xx},\pc^{3/2}\prmd_{xxx},\prmd_t,\pdl_t\big)(t)\big|_0\leq C\pcN,\label{16.3+16.5}\\
\|\pd{t}{i}\pdws(t)\|_2\leq C\bigg[\|\pd{t}{i}\prmd(t)\|+\frac{i(i-1)}{2}\pcN\|\prmd_t(t)\|\bigg],\quad i=0,1,2,\label{17.2}\\
\|\pdws_{tx}(t)\|\leq\|\pdl(t)\|,\quad\|\pdws(t)\|_4\leq C\|\prmd(t)\|_2,\label{17.3}\\
\|\pd{x}{l}\pdl_x(t)\|\leq C\|\prmd_t(t)\|_l,\quad\|\pd{x}{l}\prmd_t(t)\|\leq C\|\pdl_x(t)\|_l,\quad l=0,1,2,\label{32.4+44.2}\\
\|\pdl_{tx}(t)\|\leq C\|(\prmd_t,\prmd_{tt})(t)\|,\quad\|\pdl_{txx}(t)\|\leq C\|(\prmd_{tt},\prmd_{tx},\prmd_{ttx})(t)\|,\label{38.2c+38.2d}
\end{gather}
where the positive constant $C$ is independent of $\delta$, $\pc$ and $T$.
\end{lemma}

\subsection{Basic estimate}\label{Subsect.3.3}
In this subsection, we derive the following basic estimate.
\begin{lemma}\label{be}
Suppose the same assumptions as in Proposition \ref{prop1} hold. Then there exist positive constants $\delta_0$, $c$ and $C$ such that if $\pcN+\delta+\pc\leq\delta_0$, it holds that for $t\in[0,T]$,
\begin{equation}\label{25.1}
\frac{d}{dt}\Xi(t)+c\Pi(t)\leq C\Gamma(t),
\end{equation}
where
\begin{equation}\label{25.2}
\Xi(t):=\int_0^1\bigg\{\bigg[\frac{1}{2\rmd^2}\pdl^2+\swnd\rmd^2\Psi\bigg(\frac{\rsmd^2}{\rmd^2}\bigg)+\pc^2\prmd_x^2+\frac{1}{2}\pdws_x^2\bigg]+\frac{3\rmd^2}{4\swnd}\pwnd^2-\alpha\bigg(\frac{\dl}{\rmd^2}-\frac{\sdl}{\rsmd^2}\bigg)\pdws_x\bigg\}dx,
\end{equation}
here $\alpha\in(0,1)$ is a small constant which will be determined later and $\Psi(s):=s-1-\ln s$ for $s>0$,
\begin{equation}\label{D} 
\Pi(t):=\|(\prmd,\pc\prmd_x,\pdl,\pwnd,\pwnd_x)(t)\|^2,
\end{equation}
and
\begin{equation}\label{R} 
\Gamma(t):=\big(\pcN+\delta+\pc^{3/2}\big)\|(\prmd_x,\pdl_x)(t)\|^2+\pc^3\|(\prmd_{xx},\pdl_{xx})(t)\|^2.
\end{equation}
Furthermore, if $\alpha$ is small enough, then the following equivalent relation holds true,
\begin{equation}\label{25.3}
c\|(\prmd,\pdl,\pwnd,\pc\prmd_x)(t)\|^2\leq\Xi(t)\leq C\|(\prmd,\pdl,\pwnd,\pc\prmd_x)(t)\|^2,
\end{equation}
where the constants $c$ and $C$ are independent of $\delta$, $\pc$ and $T$.
\end{lemma}
\begin{proof}
Firstly, multiplying the equation \eqref{10.2b} by $\pdl$ and applying Leibniz formula to the resultant equality together with the equations \eqref{10.2a} and \eqref{10.2d}, we obtain
\begin{equation}\label{18.1}
\pd{t}{}\bigg[\frac{1}{2\rmd^2}\pdl^2+\swnd\rmd^2\Psi\bigg(\frac{\rsmd^2}{\rmd^2}\bigg)+\pc^2\prmd_x^2+\frac{1}{2}\pdws_x^2\bigg]+\frac{1}{\rsmd^2}\pdl^2+\frac{2\rsmd_x}{\rmd}\pwnd\pdl+\pwnd_x\pdl=\pd{x}{}R_{1}(t,x)+R_2(t,x),
\end{equation}
where
\begin{equation}\label{18.3a}
R_1(t,x):=\pdws\pdws_{tx}+\pdws\pdl-\swnd\Big(\ln\rmd^2-\ln\rsmd^2\Big)\pdl+\pc^2\bigg[\bigg(\frac{\rmd_{xx}}{\rmd}-\frac{\rsmd_{xx}}{\rsmd}\bigg)\pdl+2\prmd_t\prmd_x\bigg],
\end{equation}
\begin{align}
R_2(t,x):=&-\frac{\pdl+2\sdl}{2\rmd^4}\pdl\pdl_x-\frac{1}{2}\Bigg[\bigg(\frac{\dl}{\rmd^2}\bigg)^2-\bigg(\frac{\sdl}{\rsmd^2}\bigg)^2\Bigg]_x\pdl\notag\\
&-\frac{2\prmd_x}{\rmd}\pwnd\pdl+\swnd_x\Big(\ln\rmd^2-\ln\rsmd^2\Big)\pdl+\frac{(\rmd+\rsmd)\dl}{\rmd^2\rsmd^2}\prmd\pdl+\pc^2\frac{\rsmd_{xx}}{\rsmd\rmd}\prmd\pdl_x.\label{18.3b}
\end{align}
Applying the estimates \eqref{17.1}$\sim$\eqref{16.3+16.5} and Cauchy-Schwarz inequality to \eqref{18.3b}, we have the following pointwise estimate
\begin{equation}\label{18.4b}
R_2(t,x)\leq C\big(\pcN+\delta+\pc^{3/2}\big)|(\prmd,\pdl,\prmd_x,\pdl_x)(t,x)|^2.
\end{equation}

In addition, multiplying the equation \eqref{10.2c} by $3\pwnd/(2\swnd)$ and applying Leibniz formula to the resultant equality, we obtain
\begin{equation}\label{19.1}
\pd{t}{}\bigg(\frac{3\rmd^2}{4\swnd}\pwnd^2\bigg)+\frac{3\rsmd^2}{2\swnd}\pwnd^2+\frac{1}{\swnd}\pwnd_x^2-\frac{2\rsmd_x}{\rmd}\pdl\pwnd-\pdl\pwnd_x=\pd{x}{}R_3(t,x)+R_4(t,x),
\end{equation}
where
\begin{equation}\label{19.2a}
R_3(t,x):=\frac{1}{\swnd}\pwnd\pwnd_x-\pdl\pwnd+\frac{\pc^2}{2\swnd}\bigg[\rmd^2\bigg(\frac{\dl}{\rmd^2}\bigg)_{xx}-\rsmd^2\bigg(\frac{\sdl}{\rsmd^2}\bigg)_{xx}\bigg]\pwnd,
\end{equation}
\begin{align}
R_4(t,x):=&\frac{2\prmd_x}{\rmd}\pdl\pwnd-\frac{3\pdl_x}{4\swnd}\pwnd^2+\frac{\swnd_x}{\swnd^2}\pwnd\pwnd_x+\frac{2\sdl}{\rsmd}\prmd_x\pwnd+\frac{3}{2\swnd}H_1(t,x)\pwnd\notag\\
&+\frac{\pc^2\swnd_x}{2\swnd^2}\bigg[\rmd^2\bigg(\frac{\dl}{\rmd^2}\bigg)_{xx}-\rsmd^2\bigg(\frac{\sdl}{\rsmd^2}\bigg)_{xx}\bigg]\pwnd-\frac{\pc^2}{2\swnd}\bigg[\rmd^2\bigg(\frac{\dl}{\rmd^2}\bigg)_{xx}-\rsmd^2\bigg(\frac{\sdl}{\rsmd^2}\bigg)_{xx}\bigg]\pwnd_x.\label{19.2b}
\end{align}
Applying the estimates \eqref{17.1}$\sim$\eqref{16.3+16.5} and Cauchy-Schwarz inequality to \eqref{19.2b} together with \eqref{32.1+33.1} for $k=0$, we have the following pointwise estimate
\begin{align}
R_4(t,x)\leq&C\big(\pcN+\delta+\pc^3\big)|(\prmd,\pdl,\prmd_x,\pdl_x)(t,x)|^2\notag\\
&+C\pc|(\pwnd,\pwnd_x)(t,x)|^2+C\pc^3|(\prmd_{xx},\pdl_{xx})(t,x)|^2.\label{20.2b}
\end{align}

Since the stationary density $\rsmd$ is non-flat, see \eqref{17.1}, we have to capture the dissipation rate of the perturbed density $\prmd$ in establishing the basic estimate. To this end, multiplying the equation \eqref{10.2b} by $-\pdws_x$ and applying Leibniz formula to the resultant equality, we obtain
\begin{equation}\label{22.1}
-\pd{t}{}\bigg[\bigg(\frac{\dl}{\rmd^2}-\frac{\sdl}{\rsmd^2}\bigg)\pdws_x\bigg]+\swnd(\rmd+\rsmd)\Big(\ln\rmd^2-\ln\rsmd^2\Big)\prmd+\frac{\rmd+\rsmd}{\rmd}\pc^2\prmd_x^2+\pdws_x^2=\pd{x}{}R_5(t,x)+R_6(t,x),
\end{equation}
where
\begin{equation}\label{22.2a}
R_5(t,x):=\swnd\Big(\ln\rmd^2-\ln\rsmd^2\Big)\pdws_x-\pc^2\bigg(\frac{\rmd_{xx}}{\rmd}-\frac{\rsmd_{xx}}{\rsmd}\bigg)\pdws_x+\pc^2\frac{\rmd+\rsmd}{\rmd}\prmd_x\prmd,
\end{equation}
\begin{align} 
R_6:=&-\bigg(\frac{\dl}{\rmd^2}-\frac{\sdl}{\rsmd^2}\bigg)\pdws_{tx}+\frac{1}{2}\Bigg[\bigg(\frac{\dl}{\rmd^2}\bigg)^2-\bigg(\frac{\sdl}{\rsmd^2}\bigg)^2\Bigg]_x\pdws_x+\Big(\ln\rmd^2\Big)_x\pwnd\pdws_x\notag\\
&-\swnd_x\Big(\ln\rmd^2-\ln\rsmd^2\Big)\pdws_x+\pwnd_x\pdws_x+\bigg(\frac{\dl}{\rmd^2}-\frac{\sdl}{\rsmd^2}\bigg)\pdws_x+\pc^2\frac{(\rmd+\rsmd)\prmd}{\rmd^2}\prmd_x^2\notag\\
&-\pc^2\frac{\rsmd_{xx}(\rmd+\rsmd)}{\rmd\rsmd}\prmd^2+\pc^2\frac{(\rmd+\rsmd)\rsmd_x}{\rmd^2}\prmd_x\prmd-\pc^2\frac{(\rmd+\rsmd)_x}{\rmd}\prmd_x\prmd.\label{22.2b}
\end{align}
Similarly, we also have the pointwise estimate
\begin{align}
R_6(t,x)\leq&(\mu+C\delta)|\pdws_x(t,x)|^2+C\big(\pcN+\delta\big)|(\prmd,\prmd_x,\pdl_x)(t,x)|^2\notag\\
&+C_\mu|(\pwnd,\pwnd_x,\pdl)(t,x)|^2+C|(\pdl,\pdws_{tx})(t,x)|^2+C\big(\pcN+\pc\big)|(\prmd,\pc\prmd_x)(t,x)|^2,\label{22.3b}
\end{align}
In particular, applying the mean value theorem to the second term on the left-side of \eqref{22.1}, and using the estimates \eqref{17.1}$\sim$\eqref{16.3+16.5}, the second and the third terms on the left-side of \eqref{22.1} can be further treated as
\begin{equation}\label{23.0}
\text{(the 2nd and 3rd terms)}\geq c|(\prmd,\pc\prmd_x)(t,x)|^2,
\end{equation}
and the quantity in the first term on the left-side of \eqref{22.1} can be estimated as
\begin{equation}\label{23.2b}
\bigg|-\bigg(\frac{\dl}{\rmd^2}-\frac{\sdl}{\rsmd^2}\bigg)\pdws_x\bigg|\leq C|(\prmd,\pdl,\pdws_x)(t,x)|^2.
\end{equation}
Substituting \eqref{22.3b} and \eqref{23.0} into \eqref{22.1}, letting $\mu$ and $\pcN+\delta+\pc$ be small enough, we have
\begin{align}
-\pd{t}{}\bigg[\bigg(\frac{\dl}{\rmd^2}-\frac{\sdl}{\rsmd^2}\bigg)\pdws_x\bigg]+c|(\prmd,\pc\prmd_x)(t,x)|^2\leq&\pd{x}{}R_5(t,x)+C|(\pdws_{tx},\pdl,\pwnd,\pwnd_x)(t,x)|^2\notag\\
&+C\big(\pcN+\delta\big)|(\prmd_x,\pdl_x)(t,x)|^2.\label{23.1}
\end{align}

Finally, from the following procedure
\begin{equation*}
\int_0^1\Big[\eqref{18.1}+\eqref{19.1}+\alpha\eqref{23.1}\Big]dx,
\end{equation*}
where $\alpha$ is an arbitrary positive constant to be determined, we obtain
\begin{multline}\label{24.2}
\frac{d}{dt}\Xi(t)+\int_0^1\bigg(\frac{1}{\rsmd^2}\pdl^2+\frac{3\rsmd^2}{2\swnd}\pwnd^2+\frac{1}{\swnd}\pwnd_x^2\bigg)dx+c\alpha\|(\prmd,\pc\prmd_x)(t)\|^2\\
\leq\int_0^1\Big[R_2(t,x)+R_4(t,x)\Big]dx+C\alpha\|(\pdws_{tx},\pdl,\pwnd,\pwnd_x)(t)\|^2+C\alpha\big(\pcN+\delta\big)\|(\prmd_x,\pdl_x)(t)\|^2,
\end{multline}
where we have used the fact
\begin{equation}\label{18.4a+20.1b+23.2a}
\int_0^1\pd{x}{}\Big[R_1(t,x)+R_3(t,x)+\alpha R_5(t,x)\Big]=0,
\end{equation}
which follows from the boundary conditions \eqref{pbc}. Applying the estimates \eqref{17.1}, \eqref{17.3}, \eqref{18.4b}, \eqref{20.2b} and \eqref{23.2b} to the inequality \eqref{24.2}, and then letting $\alpha$ and $\pcN+\delta+\pc$ be sufficiently small. These procedures yield the desired estimates \eqref{25.1} and \eqref{25.3}.
\end{proof}

\subsection{Higher order estimates}\label{Subsect.3.4}
This subsection is devoted to the derivation of the higher order estimates. In order to use the homogeneous boundary condition \eqref{pbc}, we first establish the estimates of the temporal derivatives of the perturbations $(\prmd,\pdl,\pwnd)$. And then, we find that the spatial derivatives of the perturbations $(\prmd,\pdl,\pwnd,\pdws)$ can be bounded by the temporal derivatives of the perturbations $(\prmd,\pdl,\pwnd)$ with the help of the special structure of the perturbed system \eqref{10.2}. To justify the above mentioned computations, we need to use the mollifier arguments with respect to the time variable $t$ because the regularity of the local solution $(\prmd,\pdl,\pwnd)$ is insufficient. However, we omit these arguments since they are standard.

It is convenient to intoduce the notations
\begin{gather*}\label{39.1}
A_{-1}(t):=\|(\prmd,\pdl,\pwnd,\pwnd_x)(t)\|,\\
A_k(t):=A_{-1}(t)+\sum_{i=0}^k\|(\pd{t}{i}\prmd_t,\pd{t}{i}\prmd_x,\pc\pd{t}{i}\prmd_{xx})(t)\|,\quad k=0,1.
\end{gather*}

Differentiating \eqref{1dfqhd2} with respect to $x$ and multiplying the result by $1/\rmd$. Similarly, differentiating \eqref{1dsfqhd2} with respect to $x$ and multiplying the result by $1/\rsmd$. Furthermore, taking the difference between the two resultant  equalities and substituting the equations \eqref{10.2a} and \eqref{10.2d} in the resultant equation. Then applying the operator $\pd{t}{k}$ for $k=0,1$ to the result, we obtain the equation
\begin{align}
2\pd{t}{k}\prmd_{tt}-&2\swnd\pd{t}{k}\prmd_{xx}+\pc^2\pd{t}{k}\prmd_{xxxx}+2\pd{t}{k}\prmd_t-\rsmd\pd{t}{k}\pwnd_{xx}-2\rsmd_{xx}\pd{t}{k}\pwnd\notag\\
=&\frac{2(\pdl+\sdl)}{(\prmd+\rsmd)^3}\pd{t}{k}\pdl_{xx}-\frac{2(\pdl+\sdl)^2}{(\prmd+\rsmd)^4}\pd{t}{k}\prmd_{xx}+2\pwnd\pd{t}{k}\prmd_{xx}+\prmd\pd{t}{k}\pwnd_{xx}\notag\\
&+\pc^2\frac{(k+1)\prmd_{xx}+2\rsmd_{xx}}{\prmd+\rsmd}\pd{t}{k}\prmd_{xx}+\pd{t}{k}P(t,x)+O_k(t,x),\quad k=0,1,\label{12.3}
\end{align}
where
\begin{align*}
P(t,x):=&-(\prmd+\rsmd)(\prmd+2\rsmd)\prmd-(\rsmd^2-D)\prmd-\frac{2(\prmd+\rsmd)_x^2(\pwnd+\swnd)}{(\prmd+\rsmd)\rsmd}\prmd\notag\\
&+\frac{2\rsmd_x^2}{\rsmd}\pwnd+4\rsmd_x\pwnd_x-2(\prmd+\rsmd)_x\pdws_x\notag\\ 
&+\swnd_{xx}\prmd+\frac{6(\prmd+\rsmd)_x^2(\pdl+\sdl)^2}{(\prmd+\rsmd)^5\rsmd^5}\big[\rsmd^5-(\prmd+\rsmd)^5\big]+\frac{6\rsmd_x^2(\pdl+2\sdl)}{\rsmd^5}\pdl\notag\\
&-\frac{2(\pdl+\sdl)^2\big[\rsmd^4-(\prmd+\rsmd)^4\big]}{(\prmd+\rsmd)^4\rsmd^4}\rsmd_{xx}-\frac{2(\pdl+2\sdl)\pdl}{\rsmd^4}\rsmd_{xx}-\frac{\pc^2\rsmd_{xx}^2}{(\prmd+\rsmd)\rsmd}\prmd\\ 
&+\frac{6(\prmd+2\rsmd)_x(\pdl+\sdl)^2}{\rsmd^5}\prmd_x+\frac{2(\pwnd+\swnd)}{\rsmd}\prmd_x^2+\frac{2(\prmd+2\rsmd)_x\pwnd}{\rsmd}\prmd_x+4(\pwnd+\swnd)_x\prmd_x\notag\\
&-\frac{2}{\prmd+\rsmd}\prmd_t^2+\frac{2}{(\prmd+\rsmd)^3}\pdl_x^2-\frac{8(\prmd+\rsmd)_x(\pdl+\sdl)}{(\prmd+\rsmd)^4}\pdl_x+2\bigg(\frac{2\rsmd_x\swnd}{\rsmd}-\sdws_x\bigg)\prmd_x, 
\end{align*}
and 
\begin{align*}
&O_0(t,x):=0,& O_1(t,x):=&-\frac{6(\pdl+\sdl)}{(\prmd+\rsmd)^4}\prmd_t\pdl_{xx}+\frac{2}{(\prmd+\rsmd)^3}\pdl_t\pdl_{xx}+\frac{8(\pdl+\sdl)^2}{(\prmd+\rsmd)^5}\prmd_t\prmd_{xx}\notag\\
&&&-\frac{4(\pdl+\sdl)}{(\prmd+\rsmd)^4}\pdl_t\prmd_{xx}+2\pwnd_t\prmd_{xx}+\prmd_t\pwnd_{xx}-\frac{\pc^2(\prmd_{xx}+2\rsmd_{xx})\prmd_{xx}}{(\prmd+\rsmd)^2}\prmd_t. 
\end{align*}
According to \eqref{17.1}$\sim$\eqref{38.2c+38.2d}, we show the estimate
\begin{align}
\|\pd{t}{k}P(t)\|+\|O_k(t)\|\leq&C\|(\pd{t}{k}\prmd,\pd{t}{k}\pwnd,\pd{t}{k}\pwnd_{x})(t)\|\notag\\
&+C\big(\pcN+\delta+\pc^{1/2}\big)\|(\pd{t}{k}\prmd_{t},\pd{t}{k}\prmd_{x},\pd{t}{k}\pdl)(t)\|,\quad k=0,1,\label{26.1+26.4+27.2+27.3+29.1}
\end{align}
where $C$ is a positive constant independent of $\delta$, $\pc$ and $T$. In deriving the $L^2$-norm estimate of $\|\pd{t}{k}P(t)\|$, we have used the equation \eqref{1dsfqhd2} and the estimate \eqref{17.1} to deal with the coefficient of the last term in the expression of $P(t,x)$ as
\begin{equation*}
\Bigg|2\bigg(\frac{2\rsmd_x\swnd}{\rsmd}-\sdws_x\bigg)\Bigg|=\Bigg|2\bigg[\frac{2\sdl^2}{\rsmd^5}\rsmd_x-\swnd_x+\pc^2\bigg(\frac{\rsmd_{xx}}{\rsmd}\bigg)_x-\frac{\sdl}{\rsmd^2}\bigg]\Bigg|\leq C\big(\delta+\pc^{1/2}\big).
\end{equation*}

Applying the operator $\pd{t}{k}$ for $k=0,1$ to \eqref{10.2c}, we have
\begin{multline}\label{15.1}
(\prmd+\rsmd)^2\pd{t}{k}\pwnd_t-\frac{2}{3}\pd{t}{k}\pwnd_{xx}+\frac{2}{3}\swnd\pd{t}{k}\pdl_x-\frac{4\sdl\swnd}{3\rsmd}\pd{t}{k}\prmd_x\\
=\pd{x}{}\mathcal{V}_k(t,x)+\pd{t}{k}H(t,x)+L_k(t,x),\quad k=0,1,
\end{multline}
where
\begin{align}
\mathcal{V}_k(t,x):=&\frac{\pc^2}{3}\pd{t}{k}\Bigg\{(\prmd+\rsmd)^2\bigg[\frac{\pdl+\sdl}{(\prmd+\rsmd)^2}\bigg]_{xx}-\rsmd^2\bigg(\frac{\sdl}{\rsmd^2}\bigg)_{xx}\Bigg\}\notag\\ 
=&\frac{\pc^2}{3}\pd{t}{k}\pdl_{xx}-\frac{2\pc^2\sdl}{3\rsmd}\pd{t}{k}\prmd_{xx}+\pd{t}{k}\mathcal{K}(t,x),\quad k=0,1,\label{32.1+33.1}
\end{align}
\begin{align}
\mathcal{K}(t,x):=\frac{\pc^2}{3}\bigg[-\frac{4(\prmd+\rsmd)_x}{\prmd+\rsmd}&\pdl_x+\frac{6(\prmd+\rsmd)_x^2}{(\prmd+\rsmd)^2}\pdl-\frac{6\sdl(\prmd+2\rsmd)(\prmd+\rsmd)_x^2}{(\prmd+\rsmd)^2\rsmd^2}\prmd\notag\\
+&\frac{6\sdl(\prmd+2\rsmd)_x}{\rsmd^2}\prmd_x-\frac{2(\prmd+\rsmd)_{xx}}{\prmd+\rsmd}\pdl+\frac{2\sdl(\prmd+\rsmd)_{xx}}{(\prmd+\rsmd)\rsmd}\prmd\bigg],\label{32.2a}
\end{align}
\begin{equation}
L_0(t,x):=0,\quad L_1(t,x):=-2(\prmd+\rsmd)\prmd_t\pwnd_t.
\end{equation}
For convenience, we further calculate the $\pd{x}{}\mathcal{V}_k(t,x)$ as
\begin{equation}\label{34.1}
\pd{x}{}\mathcal{V}_0(t,x)=\frac{\pc^2}{3}\pdl_{xxx}-\frac{2\pc^2\sdl}{3\rsmd}\prmd_{xxx}+\underbrace{\frac{2\pc^2\sdl\rsmd_x}{\rsmd^2}\prmd_{xx}+\pd{x}{}\mathcal{K}(t,x)}_{=:\mathcal{K}_1(t,x)}, 
\end{equation}
and
\begin{equation}\label{37.2}
\pd{x}{}\mathcal{V}_1(t,x)=\frac{\pc^2}{3}\pdl_{txxx}-\frac{2\pc^2(\pdl+\sdl)}{3(\prmd+\rsmd)}\prmd_{txxx}+\mathcal{K}_2(t,x),
\end{equation}
where
\allowdisplaybreaks
\begin{align}
\mathcal{K}_2(t,x):=\frac{\pc^2}{3}\bigg[-&\frac{4(\prmd+\rsmd)_x}{\prmd+\rsmd}\pdl_{txx}+\frac{4(\prmd+\rsmd)_x\pdl_{xx}}{(\prmd+\rsmd)^2}\prmd_t-\frac{4\pdl_{xx}}{\prmd+\rsmd}\prmd_{tx}\notag\\ 
&+\frac{10(\prmd+\rsmd)_x^2}{(\prmd+\rsmd)^2}\pdl_{tx}-\frac{20(\prmd+\rsmd)_x^2\pdl_x}{(\prmd+\rsmd)^3}\prmd_t+\frac{20(\prmd+\rsmd)_x\pdl_x}{(\prmd+\rsmd)^2}\prmd_{tx}\notag\\ 
&-\frac{6(\prmd+\rsmd)_{xx}}{\prmd+\rsmd}\pdl_{tx}+\frac{6(\prmd+\rsmd)_{xx}\pdl_x}{(\prmd+\rsmd)^2}\prmd_t-\frac{6\pdl_x}{\prmd+\rsmd}\prmd_{txx}\notag\\
&-\frac{12(\prmd+\rsmd)_x^3}{(\prmd+\rsmd)^3}\pdl_t+\frac{36(\pdl+\sdl)(\prmd+\rsmd)_x^3}{(\prmd+\rsmd)^4}\prmd_t-\frac{36(\pdl+\sdl)(\prmd+\rsmd)_x^2}{(\prmd+\rsmd)^3}\prmd_{tx}\notag\\ 
&+\frac{14(\prmd+\rsmd)_{x}(\prmd+\rsmd)_{xx}}{(\prmd+\rsmd)^2}\pdl_t-\frac{28(\pdl+\sdl)(\prmd+\rsmd)_{x}(\prmd+\rsmd)_{xx}}{(\prmd+\rsmd)^3}\prmd_t\notag\\
&\qquad\qquad\qquad+\frac{14(\pdl+\sdl)(\prmd+\rsmd)_{xx}}{(\prmd+\rsmd)^2}\prmd_{tx}+\frac{14(\pdl+\sdl)(\prmd+\rsmd)_x}{(\prmd+\rsmd)^2}\prmd_{txx}\notag\\ 
&-\frac{2(\prmd+\rsmd)_{xxx}}{\prmd+\rsmd}\pdl_t+\frac{2(\pdl+\sdl)(\prmd+\rsmd)_{xxx}}{(\prmd+\rsmd)^2}\prmd_t\bigg].\label{37.1} 
\end{align}
According to \eqref{17.1}$\sim$\eqref{38.2c+38.2d}, we further show the estimates
\begin{gather}
\|H(t)\|\leq C\|(\pdl,\pwnd)(t)\|+C\big(\pcN+\delta\big)\|(\prmd,\pwnd_x)(t)\|,\label{29.2a}\\
\|\pd{t}{}H(t)\|+\|L_1(t)\|\leq C\|(\pdl_t,\pwnd_t)(t)\|+C\big(\pcN+\delta\big)\|(\prmd_t,\prmd_{tt},\prmd_{tx},\pwnd_{tx})(t)\|,\label{30.1a+30.2}\\
\|\mathcal{K}(t)\|\leq C\pc^{3/2}\|(\prmd,\prmd_t,\prmd_x,\pdl)(t)\|,\label{32.2b}\\
\|\pd{t}{}\mathcal{K}(t)\|\leq C\pc^{3/2}\|\prmd_t(t)\|+C\pc^2\|(\prmd_{tt},\prmd_{tx},\pdl_t)(t)\|+C\big(\pcN+\delta\big)\pc\|\pc\prmd_{txx}(t)\|,\label{33.2b}\\
\|\mathcal{K}_1(t)\|\leq C\pc^{1/2}\|(\prmd,\pdl)(t)\|+C\pc^{3/2}\|(\prmd_t,\prmd_x)(t)\|+C\pc^2\|(\prmd_{tx},\prmd_{xx})(t)\|,\label{34.2c}\\
\|\mathcal{K}_2(t)\|\leq C\pc^{1/2}\|(\prmd_t,\pdl_t)(t)\|+C\pc^{3/2}\|(\prmd_{tt},\prmd_{tx})(t)\|+C\pc\|(\pc\prmd_{ttx},\pc\prmd_{txx})(t)\|,\label{37.1b}
\end{gather}
where the positive constant $C$ is independent of $\delta$, $\pc$ and $T$.

The following lemma is important for the strategy in which we establish the a priori estimate \eqref{127.2}. This means that the spatial derivatives of the perturbations $(\prmd,\pdl,\pwnd)$ can be controlled by their temporal derivatives.
\begin{lemma}\label{lem5} 
Under the same assumptions as in Proposition \ref{prop1}, the following equivalent relation holds for $t\in[0,T]$,
\begin{equation}\label{46.5}
c\big(A_1(t)+\|\pwnd_t(t)\|\big)\leq\pcn\leq C\big(A_1(t)+\|\pwnd_t(t)\|\big),
\end{equation}
where the two positive constants $c$ and $C$ are independent of $\delta$, $\pc$ and $T$.
\end{lemma}
\begin{proof}
We only show the right-side inequality in \eqref{46.5} because the left-side inequality in \eqref{46.5} can be established by the similar method and the corresponding computations are much easier than the right-side one. According to the definitions of the notations $\pcn$, $A_{-1}(t)$, $A_0(t)$ and $A_1(t)$, and using the estimates \eqref{17.1}$\sim$\eqref{38.2c+38.2d}, \eqref{26.1+26.4+27.2+27.3+29.1} with $k=0$, \eqref{29.2a} and \eqref{34.2c}, the equation \eqref{12.3} with $k=0$, the equation \eqref{15.1} with $k=0$ and the equality \eqref{34.1}, we have
\begin{align}
\pcn=&\|(\prmd,\pdl,\pwnd)(t)\|_2+\|(\pc\pd{x}{3}\prmd,\pc\pd{x}{3}\pdl,\pc^2\pd{x}{4}\prmd)(t)\|\notag\\
\leq&CA_1(t)+\|\pc^2\pd{x}{4}\prmd(t)\|+\|\pwnd_{xx}(t)\|+\|\prmd_{xx}(t)\|+\|\pc\pd{x}{3}\prmd(t)\|\notag\\
=&CA_1(t)+\|\pwnd_{xx}(t)\|+\|\prmd_{xx}(t)\|+\|\pc\pd{x}{3}\prmd(t)\|\notag\\
&+\Big\|\big[\eqref{12.3}_r|_{k=0}-\big(2\prmd_{tt}-2\swnd\prmd_{xx}+2\prmd_t-\rsmd\pwnd_{xx}-2\rsmd_{xx}\pwnd\big)\big](t)\Big\|\notag\\
\leq&CA_1(t)+\|\pwnd_{xx}(t)\|+\|\prmd_{xx}(t)\|+\|\pc\pd{x}{3}\prmd(t)\|\notag\\
&+C\|(\prmd_t,\prmd_{tt},\pdl_{xx},P,\pwnd,\pwnd_{x},\pwnd_{xx},\prmd_{xx})(t)\|\notag\\ 
\leq&C\big(A_1(t)+\|\pwnd_{xx}(t)\|+\|\prmd_{xx}(t)\|+\|\pc\pd{x}{3}\prmd(t)\|\big)\notag\\
=&C\big(A_1(t)+\|\prmd_{xx}(t)\|+\|\pc\pd{x}{3}\prmd(t)\|\big)\notag\\
&+C\bigg\|-\frac{3}{2}\bigg[-(\prmd+\rsmd)^2\pwnd_t-\frac{2\swnd}{3}\pdl_x+\frac{4\sdl\swnd}{3\rsmd}\prmd_x+\frac{\pc^2}{3}\pd{x}{3}\pdl-\frac{2\pc^2\sdl}{3\rsmd}\pd{x}{3}\prmd+\mathcal{K}_1+H\bigg](t)\bigg\|\notag\\
\leq&C\big(A_1(t)+\|\prmd_{xx}(t)\|+\|\pc\pd{x}{3}\prmd(t)\|\big)+C\big(A_1(t)+\|\pwnd_t(t)\|+\pc\|\pc\pd{x}{3}\prmd(t)\|\big)\notag\\ 
\leq&C\big(A_1(t)+\|\pwnd_t(t)\|+\|\prmd_{xx}(t)\|+\|\pc\pd{x}{3}\prmd(t)\|\big).\label{42.2}
\end{align}

Moreover, multiplying the equation \eqref{12.3} with $k=0$ by $-\prmd_{xx}$ and integrating by parts with using the boundary condition \eqref{pbc2}, we obtain
\begin{align}
&\theta_{L}\|\prmd_{xx}(t)\|^2+\|\pc\pd{x}{3}\prmd(t)\|^2\notag\\
\leq&-\int_0^1\Big[\eqref{12.3}_r|_{k=0}-(2\prmd_{tt}+2\prmd_t-\rsmd\pwnd_{xx}-2\rsmd_{xx}\pwnd)\Big]\prmd_{xx}dx\notag\\
\leq&\big[\mu+C\big(\pcN+\delta+\pc^{3/2}\big)\big]\|\prmd_{xx}(t)\|^2+C_\mu\|(\prmd_t,\prmd_{tt},\prmd_{tx},\pwnd,\pwnd_x,\pwnd_{xx},P)(t)\|^2\notag\\
\leq&\big[\mu+C\big(\pcN+\delta+\pc^{3/2}\big)\big]\|\prmd_{xx}(t)\|^2+C_\mu\big(A_1^2(t)+\|\pwnd_{xx}(t)\|^2\big)\notag\\
\leq&\big[\mu+C\big(\pcN+\delta+\pc^{3/2}\big)\big]\|\prmd_{xx}(t)\|^2+C_\mu\big(A_1^2(t)+\|\pwnd_t(t)\|^2+\pc^2\|\pc\pd{x}{3}\prmd(t)\|^2\big),\label{42.0}
\end{align}
Let $\mu$ and $\pcN+\delta+\pc$ small enough, the inequality \eqref{42.0} implies
\begin{equation}\label{42.1}
\|\prmd_{xx}(t)\|+\|\pc\pd{x}{3}\prmd(t)\|\leq C\big(A_1(t)+\|\pwnd_t(t)\|\big)
\end{equation}
Substituting \eqref{42.1} into \eqref{42.2}, we have $\pcn\leq C(A_1(t)+\|\pwnd_t(t)\|)$.
\end{proof}

For convenience of later use, we estimate the $L^2$-norm of $\pd{t}{k}\pdl_t$ for $k=0,1$ in the next lemma.
\begin{lemma} 
Under the same assumptions as in Proposition \ref{prop1}, the following estimates hold for $t\in[0,T]$,
\begin{equation}\label{P79-70.3a}
\|\pdl_t(t)\|\leq C\|(\prmd,\pdl,\pwnd,\pwnd_x,\prmd_x)(t)\|+C\big(\pcN+\delta\big)\|\prmd_t(t)\|+C\pc^{1/2}\big(A_1(t)+\|\pwnd_t(t)\|\big),
\end{equation}
and
\begin{subequations}\label{P79-72.2+P79-78.1a}
\begin{equation}\label{P79-72.2}
\pd{t}{}\pdl_t=\pc^2\rmd\prmd_{txxx}+Y_1(t,x),
\end{equation}
\begin{align}
\|Y_1(t)\|\leq&C\|(\prmd,\pdl,\pwnd,\pwnd_x,\prmd_t,\prmd_x,\prmd_{tx},\pwnd_{tx})(t)\|\notag\\
&+C\big(\pcN+\delta+\pc^{1/2}\big)\|(\prmd_{tt},\pc\prmd_{xx},\pc\prmd_{txx})(t)\|,\label{P79-78.1a}
\end{align}
\end{subequations}
where $Y_1(t,x)$ is given by \eqref{72.3} and the positive constant $C$ is independent of $\delta$, $\pc$ and $T$.
\end{lemma}
\begin{proof} 
Solving the equation \eqref{10.2b} with respect to $\pdl_t$, we have
\begin{equation}\label{es68.1}
\pdl_t=\pc^2\rmd^2\bigg(\frac{\rmd_{xx}}{\rmd}-\frac{\rsmd_{xx}}{\rsmd}\bigg)_x+Y(t,x),
\end{equation}
where
\begin{align}
Y(t,x):=&\frac{2\dl}{\rmd}\prmd_t-\frac{\rmd^2}{2}\Bigg[\bigg(\frac{\dl}{\rmd^2}\bigg)^2-\bigg(\frac{\sdl}{\rsmd^2}\bigg)^2\Bigg]_x-\rmd^2\pwnd\big(\ln\rmd^2\big)_x\notag\\
&-\rmd^2\swnd\big(\ln\rmd^2-\ln\rsmd^2\big)_x-\rmd^2\pwnd_x+\rmd^2\pdws_x-\rmd^2\bigg(\frac{\dl}{\rmd^2}-\frac{\sdl}{\rsmd^2}\bigg).
\end{align}
Taking the $L^2$-norm of \eqref{es68.1} directly, and applying the estimates \eqref{17.1}, \eqref{16.3+16.5}, \eqref{17.3} and \eqref{42.1} to the resultant equality, we obtain the desired estimate \eqref{P79-70.3a}.

Next, differentiating the equation \eqref{es68.1} with respect to the time variable $t$, we get the equality \eqref{P79-72.2}, where $Y_1(t,x)$ is defined by
\begin{align}
Y_1(t,x):=&-\pc^2\rmd_x\prmd_{txx}-\pc^2\rmd_{xxx}\prmd_t+\frac{2\pc^2\rmd_{xx}\rmd_x}{\rmd}\prmd_t\notag\\
&-\pc^2\rmd_{xx}\prmd_{tx}+2\pc^2\rmd\prmd_t\bigg(\frac{\rmd_{xx}}{\rmd}-\frac{\rsmd_{xx}}{\rsmd}\bigg)_x+\pd{t}{}Y(t,x).\label{72.3}
\end{align}
Similarly, taking the $L^2$-norm of \eqref{72.3} directly, and applying the estimates \eqref{17.1}, \eqref{16.3+16.5}, \eqref{17.3}, \eqref{38.2c+38.2d}, \eqref{42.1} and \eqref{P79-70.3a} to the resultant equality, we have the desired estimate \eqref{P79-78.1a}.
\end{proof}

Now, we begin to derive the higher order estimates to complete the a priori estimate \eqref{127.2}. From the following lemma, we can see that the Bohm potential term in the momentum equation contributes the quantum dissipation rate $\|\pc\pd{t}{k}\prmd_{xx}(t)\|$.
\begin{lemma} 
Suppose the same assumptions as in Proposition \ref{prop1} hold. Then there exist positive constants $\delta_0$, $c$ and $C$ such that if $\pcN+\delta+\pc\leq\delta_0$, it holds that for $t\in[0,T]$,
\begin{equation}\label{115.1}
\frac{d}{dt}\Xi_1^{(k)}(t)+c\Pi_1^{(k)}(t)\leq C\Gamma_1^{(k)}(t),\quad k=0,1,
\end{equation}
where
\begin{gather}
\Xi_1^{(k)}(t):=\int_0^1\Big[\big(\pd{t}{k}\prmd\big)^2+2\pd{t}{k}\prmd_t\pd{t}{k}\prmd\Big]dx,\quad \Pi_1^{(k)}(t):=\|(\pd{t}{k}\prmd_x,\pc\pd{t}{k}\prmd_{xx})(t)\|^2,\notag\\
\Gamma_1^{(k)}(t):=\|(\pd{t}{k}\prmd_t,\pd{t}{k}\prmd,\pd{t}{k}\pwnd,\pd{t}{k}\pwnd_x)(t)\|^2+\big(\pcN+\delta+\pc^{1/2}\big)A_k^2(t),\label{115.1a}
\end{gather}
and the constants $c$ and $C$ are independent of $\delta$, $\pc$ and $T$.
\end{lemma}
\begin{proof} 
Multiplying the equation \eqref{12.3} by $\pd{t}{k}\prmd$ and integrating the resultant equality by parts over the domain $\Omega=(0,1)$ together with the homogeneous boundary conditions \eqref{pbc}, we get
\begin{equation}\label{48.1}
\frac{d}{dt}\Xi_1^{(k)}(t)+\int_0^1\Big[2\swnd\big(\pd{t}{k}\prmd_x\big)^2+\big(\pc\pd{t}{k}\prmd_{xx}\big)^2\Big]dx=\mathcal{I}_1^{(k)}(t),\quad k=0,1,
\end{equation}
where the integral term $\mathcal{I}_1^{(k)}(t)$ is defined by
\begin{align}
\mathcal{I}_1^{(k)}(t):=\int_0^1\bigg\{&2\big(\pd{t}{k}\prmd_t\big)^2-2\swnd_x\pd{t}{k}\prmd_x\pd{t}{k}\prmd-\Big(\rsmd_x\pd{t}{k}\pwnd_x\pd{t}{k}\prmd+\rsmd\pd{t}{k}\pwnd_x\pd{t}{k}\prmd_x\Big)\notag\\
&+2\rsmd_{xx}\pd{t}{k}\pwnd\pd{t}{k}\prmd+\bigg(\frac{6\rmd_x\dl}{\rmd^4}\pd{t}{k}\pdl_x\pd{t}{k}\prmd-\frac{2\pdl_x}{\rmd^3}\pd{t}{k}\pdl_x\pd{t}{k}\prmd-\frac{2\dl}{\rmd^3}\pd{t}{k}\pdl_x\pd{t}{k}\prmd_x\bigg)\notag\\
&-\bigg[\frac{8\rmd_x\dl^2}{\rmd^5}\pd{t}{k}\prmd_x\pd{t}{k}\prmd-\frac{4\dl\pdl_x}{\rmd^4}\pd{t}{k}\prmd_x\pd{t}{k}\prmd-\frac{2\dl^2}{\rmd^4}\big(\pd{t}{k}\prmd_x\big)^2\bigg]\notag\\
&-\Big[2\pwnd_x\pd{t}{k}\prmd_x\pd{t}{k}\prmd+2\pwnd\big(\pd{t}{k}\prmd_x\big)^2\Big]-\Big(\prmd_x\pd{t}{k}\pwnd_x\pd{t}{k}\prmd+\prmd\pd{t}{k}\pwnd_x\pd{t}{k}\prmd_x\Big)\notag\\
&\qquad\quad+\pc^2\frac{(k+1)\prmd_{xx}+2\rsmd_{xx}}{\rmd}\pd{t}{k}\prmd_{xx}\pd{t}{k}\prmd+\Big(\pd{t}{k}P+O_k\Big)\pd{t}{k}\prmd\bigg\}dx.\label{48.0}
\end{align}
According to the estimates \eqref{17.1}, \eqref{16.3+16.5}, \eqref{38.2c+38.2d}, \eqref{26.1+26.4+27.2+27.3+29.1} and \eqref{P79-70.3a}, and then using Cauchy-Schwarz, Young and H\"older inequalities, we have the estimates
\begin{equation}\label{52.1}
\int_0^1\Big[2\swnd\big(\pd{t}{k}\prmd_x\big)^2+\big(\pc\pd{t}{k}\prmd_{xx}\big)^2\Big]dx\geq c\Pi_1^{(k)}(t),
\end{equation}
and
\begin{align}
\mathcal{I}_1^{(k)}(t)\leq&2\|\pd{t}{k}\prmd_t(t)\|^2+C\delta\|(\pd{t}{k}\prmd_x,\pd{t}{k}\prmd)(t)\|^2+\mu\|\pd{t}{k}\prmd_x(t)\|^2+C_\mu\|(\pd{t}{k}\pwnd_x,\pd{t}{k}\prmd)(t)\|^2\notag\\
&+C\|(\pd{t}{k}\pwnd,\pd{t}{k}\pwnd_x,\pd{t}{k}\prmd)(t)\|^2+C\big(\pcN+\delta\big)\|(\pd{t}{k}\prmd,\pd{t}{k}\prmd_x,\pd{t}{k}\pdl_x)(t)\|^2\notag\\
&+C\pc^{1/2}\|(\pd{t}{k}\prmd,\pc\pd{t}{k}\prmd_{xx})(t)\|^2+C\|(\pd{t}{k}P,O_k,\pd{t}{k}\prmd)(t)\|^2\notag\\
\leq&\mu\|\pd{t}{k}\prmd_x(t)\|^2+C_\mu\|(\pd{t}{k}\prmd_t,\pd{t}{k}\prmd,\pd{t}{k}\pwnd,\pd{t}{k}\pwnd_x)(t)\|^2\notag\\
&+C\big(\pcN+\delta+\pc^{1/2}\big)A_k^2(t).\label{52.2}
\end{align}
Substituting \eqref{52.1} and \eqref{52.2} into \eqref{48.1} and taking $\mu$ small enough, we obtain the desired estimate \eqref{115.1}.
\end{proof}

Next, the following lemma is the most difficult part in establishing the higher order estimates. From this lemma, we can find that the dispersive velocity term in the energy equation contributes the extra quantum dissipation rate $\|\pc\pd{t}{k}\prmd_{tx}(t)\|$, see \eqref{100.2}. It plays a similar role like the additional dissipation rate $\|\pwnd_{tx}(t)\|$ contributed by the diffusion term in the energy equation, see \eqref{126.2}.
\begin{lemma} 
Suppose the same assumptions as in Proposition \ref{prop1} hold. Then there exist positive constants $\delta_0$, $c$ and $C$ such that if $\pcN+\delta+\pc\leq\delta_0$, it holds that for $t\in[0,T]$,
\begin{equation}\label{115.2}
\frac{d}{dt}\Xi_2^{(k)}(t)+c\Pi_2^{(k)}(t)\leq C\Gamma_2^{(k)}(t),\quad k=0,1,
\end{equation}
where
\begin{multline*}
\Xi_2^{(k)}(t):=\int_0^1\Bigg\{\big(\pd{t}{k}\prmd_t\big)^2+\bigg(\wnd-\frac{\dl^2}{\rmd^4}\bigg)\big(\pd{t}{k}\prmd_x\big)^2+\frac{1}{2}\big(\pc\pd{t}{k}\prmd_{xx}\big)^2\\
-\frac{3\rmd^3}{2}\pd{t}{k}\pwnd\pd{t}{k}\prmd_t-k\bigg[\frac{9\rmd^5\pc}{8}\pwnd_t\big(\pc\prmd_{txx}\big)+\frac{3\rmd^4\pc^2}{8}\big(\pc\prmd_{txx}\big)^2\bigg]\Bigg\}dx,
\end{multline*}
\begin{equation*}
\Pi_2^{(k)}(t):=\|(\pd{t}{k}\prmd_t,\pc\pd{t}{k}\prmd_{tx})(t)\|^2,\quad \Gamma_2^{(k)}(t):=\big(\mu+\pcN+\delta+\pc^{1/2}\big)\|\pd{t}{k}\prmd_x(t)\|^2+\Upsilon^{(k)}(t),
\end{equation*}
\begin{equation*}
\Upsilon^{(0)}(t):=C_\mu\|(\prmd,\pdl,\pwnd,\pwnd_x)(t)\|^2+\big(\pcN+\delta+\pc^{1/2}\big)\|(\prmd_{tt},\prmd_{tx},\pc\prmd_{xx},\pc\prmd_{txx},\pwnd_t)(t)\|^2,
\end{equation*}
\begin{align}
\Upsilon^{(1)}(t):=C_\mu&\|(\pwnd_t,\pwnd_{tx})(t)\|^2+\|(\prmd,\pdl,\pwnd,\pwnd_x,\prmd_t,\prmd_x)(t)\|^2\notag\\
&+\big(\pcN+\delta+\pc^{1/2}\big)\|(\prmd_{tt},\pc\prmd_{xx},\pc\prmd_{txx})(t)\|^2,\label{115.2a}
\end{align}
and the constants $c$ and $C$ are independent of $\delta$, $\pc$ and $T$. Here $\mu$ is an arbitrary positive constant to be determined and $C_\mu$ is a generic constant which only depends on $\mu$.
\end{lemma}
\begin{proof} 
Multiplying the equation \eqref{12.3} by $\pd{t}{k}\prmd_t$ and integrating the resultant equality by parts over the domain $\Omega$ together with the homogeneous boundary conditions \eqref{pbc}, we get
\begin{multline}\label{60.1}
\frac{d}{dt}\int_0^1\bigg[\big(\pd{t}{k}\prmd_t\big)^2+\bigg(\wnd-\frac{\dl^2}{\rmd^4}\bigg)\big(\pd{t}{k}\prmd_x\big)^2+\frac{1}{2}\big(\pc\pd{t}{k}\prmd_{xx}\big)^2\bigg]dx\\
+2\|\pd{t}{k}\prmd_t(t)\|^2-\int_0^1\rmd\pd{t}{k}\pwnd_{xx}\pd{t}{k}\prmd_tdx=\mathcal{I}_2^{(k)}(t),\quad k=0,1,
\end{multline}
where the integral term $\mathcal{I}_2^{(k)}(t)$ is given by
\begin{align}
\mathcal{I}_2^{(k)}(t):=\int_0^1\bigg\{&\bigg[-2\bigg(\wnd-\frac{\dl^2}{\rmd^4}\bigg)_x\pd{t}{k}\prmd_x\pd{t}{k}\prmd_t+\bigg(\wnd-\frac{\dl^2}{\rmd^4}\bigg)_t\big(\pd{t}{k}\prmd_x\big)^2\bigg]\notag\\
&+2\rsmd_{xx}\pd{t}{k}\pwnd\pd{t}{k}\prmd_t+\frac{2\dl}{\rmd^3}\pd{t}{k}\pdl_{xx}\pd{t}{k}\prmd_t+\pc^2\frac{(k+1)\prmd_{xx}+2\rsmd_{xx}}{\rmd}\pd{t}{k}\prmd_{xx}\pd{t}{k}\prmd_t\notag\\
&+\Big[\pd{t}{k}P(t,x)+O_k(t,x)\Big]\pd{t}{k}\prmd_t\bigg\}dx
\end{align}
and can be estimated by the standard method as follows
\begin{align}
\mathcal{I}_2^{(k)}(t)\leq&C(\pcN+\delta)\big(\|(\pd{t}{k}\prmd_t,\pd{t}{k}\prmd_x)(t)\|^2+\|\pwnd_t(t)\|_k^2\big)\notag\\ 
&+\mu\|\pd{t}{k}\prmd_t(t)\|^2+C_\mu\|\pd{t}{k}\pwnd(t)\|_1^2\notag\\ 
&+C(\pcN+\delta)\|(\pd{t}{k}\prmd_t,k\prmd_{tx})(t)\|^2\notag\\ 
&+C\pc^{1/2}\|(\pd{t}{k}\prmd_t,\pc\pd{t}{k}\prmd_{xx})(t)\|^2\notag\\ 
&+\mu\|\pd{t}{k}\prmd_t(t)\|^2+C_\mu\|(\pd{t}{k}\prmd,\pd{t}{k}\pwnd,\pd{t}{k}\pwnd_x)(t)\|^2\notag\\
&\qquad+C_\mu\big(\pcN+\delta+\pc^{1/2}\big)\|(\pd{t}{k}\pdl,\pd{t}{k}\prmd_t,\pd{t}{k}\prmd_x)(t)\|^2\notag\\ 
\leq&2\mu\|\pd{t}{k}\prmd_t(t)\|^2+C_\mu\|(\pd{t}{k}\prmd,\pd{t}{k}\pwnd,\pd{t}{k}\pwnd_x)(t)\|^2\notag\\
&+C_\mu\big(\pcN+\delta+\pc^{1/2}\big)\Big(\|(\pd{t}{k}\pdl,\pd{t}{k}\prmd_t,\pd{t}{k}\prmd_x,\pc\pd{t}{k}\prmd_{xx})(t)\|^2+\|\pwnd_t(t)\|_k^2\Big),\label{107.2r}
\end{align}
with the aid of the estimates \eqref{17.1}, \eqref{16.3+16.5}, \eqref{38.2c+38.2d} and \eqref{26.1+26.4+27.2+27.3+29.1}, and the H\"older, Young and Sobolev inequalities.

Now, we have to deal with the last integral term on the left-side of the equality \eqref{60.1}. It is the most difficult part in the proof due to the dispersive velocity term in the energy equation and the Bohm potential term in the momentum equation. Precisely, solving the equation \eqref{15.1} with respect to $\pd{t}{k}\pwnd_{xx}$ and substituting the result in the last integral term on the left-side of \eqref{60.1} give
\begin{align}
-\int_0^1\rmd\pd{t}{k}\pwnd_{xx}\pd{t}{k}\prmd_tdx=&\int_0^1\rmd\frac{3}{2}\bigg[-\rmd^2\pd{t}{k}\pwnd_t-\frac{2}{3}\swnd\pd{t}{k}\pdl_x+\frac{4\sdl\swnd}{3\rsmd}\pd{t}{k}\prmd_x\notag\\
&\qquad\qquad\qquad+\pd{x}{}\mathcal{V}_k(t,x)+\pd{t}{k}H(t,x)+L_k(t,x)\bigg]\pd{t}{k}\prmd_tdx\notag\\
=&-\int_0^1\frac{3}{2}\rmd^3\pd{t}{k}\pwnd_t\pd{t}{k}\prmd_tdx-\int_0^1\rmd\swnd\pd{t}{k}\pdl_x\pd{t}{k}\prmd_tdx\notag\\
&+\int_0^1\frac{2\rmd\sdl\swnd}{\rsmd}\pd{t}{k}\prmd_x\pd{t}{k}\prmd_tdx+\int_0^1\frac{3}{2}\rmd\pd{x}{}\mathcal{V}_k(t,x)\pd{t}{k}\prmd_tdx\notag\\
&+\int_0^1\frac{3}{2}\rmd\Big[\pd{t}{k}H(t,x)+L_k(t,x)\Big]\pd{t}{k}\prmd_tdx\notag\\
=&\mathfrak{T}^{(k)}_1(t)+\mathfrak{T}^{(k)}_2(t)+\mathfrak{T}^{(k)}_3(t)+\mathfrak{T}^{(k)}_4(t)+\mathfrak{T}^{(k)}_5(t). \label{64.2}
\end{align}
The integrals $\mathfrak{T}^{(k)}_2(t)$, $\mathfrak{T}^{(k)}_3(t)$ and $\mathfrak{T}^{(k)}_5(t)$ are relatively easier to be estimated than to deal with the integrals $\mathfrak{T}^{(k)}_1(t)$ and $\mathfrak{T}^{(k)}_4(t)$. Before treating them one by one, we first derive the following equality which follows from the equation \eqref{10.2a},
\begin{equation}\label{65.3}
\pd{t}{k}\prmd_{tt}=-\frac{1}{2\rmd}\pd{t}{k}\pdl_{tx}+\mathcal{B}_k(t,x),\quad k=0,1,
\end{equation}
where
\begin{equation*}\label{65.2}
\mathcal{B}_0(t,x):=\frac{1}{2\rmd^2}\prmd_t\pdl_x,\quad \mathcal{B}_1(t,x):=\frac{1}{\rmd^2}\prmd_t\pdl_{tx}-\frac{1}{\rmd^3}\prmd_t^2\pdl_x+\frac{1}{2\rmd^2}\prmd_{tt}\pdl_x,
\end{equation*}
satisfying the estimate
\begin{equation}\label{66.2}
\|\mathcal{B}_k(t)\|\leq C\pcN\|(\pd{t}{k}\prmd_t,\prmd_t)(t)\|.
\end{equation}

Now, we begin to estimate $\mathfrak{T}^{(k)}_l(t)$, $l=1,\cdots,5$. Firstly, using the estimates \eqref{17.1}, \eqref{16.3+16.5}, \eqref{29.2a} and \eqref{30.1a+30.2}, the equation \eqref{10.2a} and the Young inequality, via the standard computations, we have
\begin{align}
\mathfrak{T}^{(k)}_2(t)+\mathfrak{T}^{(k)}_3(t)+\mathfrak{T}^{(k)}_5(t)\geq&\int_0^1\rmd\swnd\big(2\rmd\pd{t}{k}\prmd_t+k2\prmd_t^2\big)\pd{t}{k}\prmd_tdx-C\delta\|(\pd{t}{k}\prmd_t,\pd{t}{k}\prmd_x)(t)\|^2\notag\\
&-\mu\|\pd{t}{k}\prmd_t(t)\|^2-C_\mu\|(\pd{t}{k}\prmd,\pd{t}{k}\pdl,\pd{t}{k}\pwnd,\pd{t}{k}\pwnd_x)(t)\|^2\notag\\
&-kC_\mu(\pcN+\delta)\|(\prmd_{tt},\prmd_{tx})(t)\|^2\notag\\
\geq&c\|\pd{t}{k}\prmd_t(t)\|^2-C\|(\pd{t}{k}\prmd,\pd{t}{k}\pdl,\pd{t}{k}\pwnd,\pd{t}{k}\pwnd_x)(t)\|^2\notag\\
&-C(\pcN+\delta)\|\pd{t}{k}\prmd_x(t)\|^2.\label{92.2+93.1+101.1} 
\end{align}

In addition, we continue to estimate $\mathfrak{T}^{(k)}_1(t)$ by using \eqref{65.3}, \eqref{66.2} and the integration by parts.
\begin{align}
\mathfrak{T}^{(k)}_1(t)=&-\frac{d}{dt}\int_0^1\frac{3\rmd^3}{2}\pd{t}{k}\pwnd\pd{t}{k}\prmd_tdx+\int_0^1\frac{9\rmd^2\prmd_t}{2}\pd{t}{k}\pwnd\pd{t}{k}\prmd_tdx+\int_0^1\frac{3\rmd^3}{2}\pd{t}{k}\pwnd\pd{t}{k}\prmd_{tt}dx\notag\\
\geq&-\frac{d}{dt}\int_0^1\frac{3\rmd^3}{2}\pd{t}{k}\pwnd\pd{t}{k}\prmd_tdx-C\pcN\|(\pd{t}{k}\prmd_t,\pd{t}{k}\pwnd)(t)\|^2\notag\\
&+\int_0^1\frac{3\rmd^3}{2}\pd{t}{k}\pwnd\bigg[-\frac{1}{2\rmd}\pd{t}{k}\pdl_{tx}+\mathcal{B}_k(t,x)\bigg]dx\notag\\
\geq&-\frac{d}{dt}\int_0^1\frac{3\rmd^3}{2}\pd{t}{k}\pwnd\pd{t}{k}\prmd_tdx-C\pcN\|(\pd{t}{k}\prmd_t,\prmd_t,\pd{t}{k}\pwnd)(t)\|^2\notag\\
&+\int_0^1\frac{3\rmd^2}{4}\pd{t}{k}\pwnd_x\pd{t}{k}\pdl_tdx+\int_0^1\frac{3\rmd\rmd_x}{2}\pd{t}{k}\pwnd\pd{t}{k}\pdl_tdx,\quad k=0,1.\label{ge67.1}
\end{align}
Moreover, we have to separately deal with the last two terms on the right-side of \eqref{ge67.1} for $k=0$ and $k=1$. In fact,
\begin{subequations}\label{ge80.1+90.3}
\begin{equation}\label{ge80.1}
\int_0^1\frac{3\rmd^2}{4}\pwnd_x\pdl_tdx\geq-\mu\|\pdl_t(t)\|^2-C_\mu\|\pwnd_x(t)\|^2,\qquad\text{for}\ k=0,
\end{equation}
\begin{align}
\int_0^1\frac{3\rmd^2}{4}\pwnd_{tx}\pd{t}{}\pdl_tdx=&\int_0^1\frac{3\rmd^2}{4}\pwnd_{tx}\Big[\pc^2\rmd\prmd_{txxx}+Y_1(t,x)\Big]dx\notag\\
=&-\int_0^1\frac{3\rmd^3\pc^2}{4}\pwnd_{txx}\prmd_{txx}dx-\int_0^1\frac{9\rmd^2\rmd_x\pc}{4}\pwnd_{tx}(\pc\prmd_{txx})dx+\int_0^1\frac{3\rmd^2}{4}\pwnd_{tx}Y_1dx\notag\\
\geq&\int_0^1\frac{9\rmd^3\pc^2}{8}\bigg[-\rmd^2\pwnd_{tt}-\frac{2}{3}\swnd\pdl_{tx}+\frac{4\sdl\swnd}{3\rsmd}\prmd_{tx}+\pd{x}{}\mathcal{V}_1+\pd{t}{}H+L_1\bigg]\prmd_{txx}dx\notag\\
&-C\pc\|(\pc\prmd_{txx},\pwnd_{tx})(t)\|^2-\mu\|Y_1(t)\|^2-C_\mu\|\pwnd_{tx}(t)\|^2\notag\\
\geq&-\int_0^1\frac{9\rmd^5\pc^2}{8}\pwnd_{tt}\prmd_{txx}dx+\int_0^1\frac{9\rmd^3\pc^2}{8}\pd{x}{}\mathcal{V}_1\prmd_{txx}dx\notag\\
&-C\pc\|(\pdl_t,\prmd_t,\prmd_{tt},\prmd_{tx},\pc\prmd_{txx})(t)\|^2-\mu\|Y_1(t)\|^2-C_\mu\|\pwnd_{tx}(t)\|^2\notag\\
\geq&-\frac{d}{dt}\int_0^1\frac{9\rmd^5\pc}{8}\pwnd_t(\pc\prmd_{txx})dx-\frac{d}{dt}\int_0^1\frac{3\rmd^4\pc^2}{8}(\pc\prmd_{txx})^2dx\notag\\
&-C\pc\|(\pdl_t,\prmd_t,\prmd_{tt},\prmd_{tx},\pc\prmd_{ttx},\pc\prmd_{txx})(t)\|^2\notag\\
&-\mu\|Y_1(t)\|^2-C_\mu\|\pwnd_{tx}(t)\|^2,\quad\qquad\qquad\qquad\text{for}\ k=1,\label{90.3}
\end{align}
\end{subequations}
where we have used the equality \eqref{P79-72.2}, the equation \eqref{15.1} with $k=1$, the estimates \eqref{38.2c+38.2d} and \eqref{30.1a+30.2}, the Cauchy-Schwarz inequality, and the following computations,
\begin{align}
-\int_0^1\frac{9\rmd^5\pc^2}{8}\pwnd_{tt}\prmd_{txx}dx=&-\frac{d}{dt}\int_0^1\frac{9\rmd^5\pc}{8}\pwnd_t(\pc\prmd_{txx})dx+\int_0^1\frac{45\rmd^4\prmd_t\pc}{8}\pwnd_t(\pc\prmd_{txx})dx\notag\\
&+\int_0^1\frac{9\rmd^5\pc^2}{8}\pwnd_t\prmd_{ttxx}dx\notag\\
=&-\frac{d}{dt}\int_0^1\frac{9\rmd^5\pc}{8}\pwnd_t(\pc\prmd_{txx})dx+\int_0^1\frac{45\rmd^4\prmd_t\pc}{8}\pwnd_t(\pc\prmd_{txx})dx\notag\\
&-\int_0^1\frac{45\rmd^4\rmd_x\pc}{8}\pwnd_t(\pc\prmd_{ttx})dx-\int_0^1\frac{9\rmd^5\pc}{8}\pwnd_{tx}(\pc\prmd_{ttx})dx\notag\\
\geq&-\frac{d}{dt}\int_0^1\frac{9\rmd^5\pc}{8}\pwnd_t(\pc\prmd_{txx})dx-C\pc\|(\pc\prmd_{ttx},\pc\prmd_{txx},\pwnd_{tx})(t)\|^2,\label{a1-82.1}
\end{align}
and
\begin{align}
&\int_0^1\frac{9\rmd^3\pc^2}{8}\pd{x}{}\mathcal{V}_1\prmd_{txx}dx\notag\\
=&\int_0^1\frac{9\rmd^3\pc^2}{8}\bigg(\frac{\pc^2}{3}\pdl_{txxx}-\frac{2\pc^2\dl}{3\rmd}\prmd_{txxx}+\mathcal{K}_2\bigg)\prmd_{txx}dx\notag\\
\geq&-\int_0^1\frac{3\rmd^3\pc^2}{8}\Big(2\rmd\prmd_{ttxx}+4\rmd_x\prmd_{ttx}+4\prmd_{tx}^2+4\prmd_t\prmd_{txx}+2\rmd_{xx}\prmd_{tt}\Big)\prmd_{txx}dx\notag\\
&+\int_0^1\bigg(\frac{3\pc^4}{8}\rmd^2\dl\bigg)_x\prmd_{txx}^2dx-C\pc\|(\mathcal{K}_2,\pc\prmd_{txx})(t)\|^2\notag\\
\geq&-\frac{d}{dt}\int_0^1\frac{3\rmd^4\pc^2}{8}(\pc\prmd_{txx})^2dx-C\pc\|(\pdl_t,\prmd_t,\prmd_{tt},\prmd_{tx},\pc\prmd_{ttx},\pc\prmd_{txx})(t)\|^2,\label{a4-88.2}
\end{align}
with the aid of the equality \eqref{37.2}, the equation \eqref{10.2a}, the integration by parts, the estimates \eqref{17.1}, \eqref{16.3+16.5} and \eqref{37.1b}, and the Sobolev, H\"older and Cauchy-Schwarz inequalities. Similarly, we continue to estimate the last term on the right-side of \eqref{ge67.1} as follows
\begin{subequations}\label{91.1+91.2}
\begin{equation}\label{91.1}
\int_0^1\frac{3\rmd\rmd_x}{2}\pwnd\pdl_tdx\geq-\mu\|\pdl_t(t)\|^2-C_\mu\|\pwnd(t)\|^2, \quad\text{for}\ k=0,
\end{equation}
\begin{align}
\int_0^1\frac{3\rmd\rmd_x}{2}\pwnd_t\pd{t}{}\pdl_tdx=&\int_0^1\frac{3\rmd\rmd_x}{2}\pwnd_t\Big(\pc^2\rmd\prmd_{txxx}+Y_1\Big)dx\notag\\
\geq&-\int_0^1\bigg(\frac{3\rmd^2\rmd_x\pc^2}{2}\pwnd_t\bigg)_x\prmd_{txx}dx-\mu\|Y_1(t)\|^2-C_\mu\|\pwnd_t(t)\|^2\notag\\
\geq&-C\pc^{1/2}\|(\pc\prmd_{txx},\pwnd_{tx})(t)\|^2-\mu\|Y_1(t)\|^2-C_\mu\|\pwnd_t(t)\|^2, \quad\text{for}\ k=1.\label{91.2}
\end{align}
\end{subequations}

Next, we estimate the integral $\mathfrak{T}^{(k)}_4(t)$ by using the integration by parts and the equality \eqref{32.1+33.1},
\begin{align}
\mathfrak{T}^{(k)}_4(t)=&-\int_0^1\bigg(\frac{3}{2}\rmd\pd{t}{k}\prmd_t\bigg)_x\mathcal{V}_kdx\notag\\
=&-\int_0^1\frac{3}{2}\rmd\pd{t}{k}\prmd_{tx}\mathcal{V}_kdx-\int_0^1\frac{3}{2}\rmd_x\pd{t}{k}\prmd_t\mathcal{V}_kdx\notag\\
\geq&-\int_0^1\frac{3}{2}\rmd\pd{t}{k}\prmd_{tx}\bigg(\frac{\pc^2}{3}\pd{t}{k}\pdl_{xx}-\frac{2\pc^2\sdl}{3\rsmd}\pd{t}{k}\prmd_{xx}+\pd{t}{k}\mathcal{K}\bigg)dx\notag\\
&-C\pc\|(\pd{t}{k}\pdl,\pd{t}{k}\prmd,\pd{t}{k}\prmd_t,\pd{t}{k}\prmd_x,\pc\pd{t}{k}\prmd_{tx},\pc\pd{t}{k}\prmd_{xx})(t)\|^2\notag\\ 
\geq&-\int_0^1\frac{\rmd\pc^2}{2}\pd{t}{k}\prmd_{tx}\pd{t}{k}\pdl_{xx}dx\notag\\
&-C\big(\pcN+\delta+\pc^{1/2}\big)\|(\pd{t}{k}\pdl,\pd{t}{k}\prmd,\pd{t}{k}\prmd_t,\pd{t}{k}\prmd_x,\pc\pd{t}{k}\prmd_{tx},\pc\pd{t}{k}\prmd_{xx})(t)\|^2\notag\\
=&\int_0^1\frac{\rmd\pc^2}{2}\pd{t}{k}\prmd_{tx}\Big(2\rmd\pd{t}{k}\prmd_{tx}+2\rmd_x\pd{t}{k}\prmd_t+k4\prmd_t\prmd_{tx}\Big)dx\notag\\
&-C\big(\pcN+\delta+\pc^{1/2}\big)\|(\pd{t}{k}\pdl,\pd{t}{k}\prmd,\pd{t}{k}\prmd_t,\pd{t}{k}\prmd_x,\pc\pd{t}{k}\prmd_{tx},\pc\pd{t}{k}\prmd_{xx})(t)\|^2\notag\\
\geq&\int_0^1\rmd^2\big(\pc\pd{t}{k}\prmd_{tx}\big)^2dx\notag\\
&-C\big(\pcN+\delta+\pc^{1/2}\big)\|(\pd{t}{k}\pdl,\pd{t}{k}\prmd,\pd{t}{k}\prmd_t,\pd{t}{k}\prmd_x,\pc\pd{t}{k}\prmd_{tx},\pc\pd{t}{k}\prmd_{xx})(t)\|^2\notag\\
\geq&c\|\pc\pd{t}{k}\prmd_{tx}(t)\|^2-C\big(\pcN+\delta+\pc^{1/2}\big)\|(\pd{t}{k}\pdl,\pd{t}{k}\prmd,\pd{t}{k}\prmd_t,\pd{t}{k}\prmd_x,\pc\pd{t}{k}\prmd_{xx})(t)\|^2,\label{100.2}
\end{align}
where we have used the equation \eqref{10.2a} and the estimates \eqref{32.2b}$\sim$\eqref{33.2b}.

Finally, substituting \eqref{107.2r}, \eqref{64.2}, \eqref{92.2+93.1+101.1}, \eqref{ge67.1}, \eqref{ge80.1+90.3}, \eqref{91.1+91.2} and \eqref{100.2} into \eqref{60.1}, applying the estiamtes \eqref{P79-70.3a} and \eqref{P79-78.1a} to the resultant inequality, making $\mu$ and $\pcN+\delta+\pc$ small enough and rewriting the result as a unified form in $k=0,1$, we obtain the desired estimate \eqref{115.2}.
\end{proof}

In order to close the uniform a priori estimate, we continue to derive the higher order estimates of the perturbed temperature $\pwnd$. The dispersive velocity term in the energy equation makes the corresponding computations more complex.
\begin{lemma} 
Suppose the same assumptions as in Proposition \ref{prop1} hold. Then there exist positive constants $\delta_0$, $c$ and $C$ such that if $\pcN+\delta+\pc\leq\delta_0$, it holds that for $t\in[0,T]$,
\begin{equation}\label{116.2}
\frac{d}{dt}\Xi_3(t)+c\Pi_3(t)\leq C\Gamma_3(t),
\end{equation}
where
\begin{gather}
\Xi_3(t):=\int_0^1\bigg(\frac{1}{3}\pwnd_x^2+\frac{2\swnd}{3}\pdl_x\pwnd\bigg)dx,\quad\Pi_3(t):=\|\pwnd_t(t)\|^2,\notag\\
\Gamma_3(t):=\mu\|\prmd_x(t)\|^2+C_\mu A_{-1}^2(t)+\big(\pcN+\delta+\pc\big)\Big(A_1^2(t)+\|\pwnd_{tx}(t)\|^2\Big),\label{116.2a}
\end{gather}
and
\begin{equation}\label{117.2}
\frac{d}{dt}\Xi_4(t)+c\Pi_4(t)\leq C\Gamma_4(t),
\end{equation}
where
\begin{gather}
\Xi_4(t):=\int_0^1\frac{\rmd^2}{2}\pwnd_t^2dx,\quad\Pi_4(t):=\|\pwnd_{tx}(t)\|^2,\notag\\
\Gamma_4(t):=\|(\prmd,\pdl,\pwnd,\pwnd_x,\prmd_x,\pwnd_t)(t)\|^2+\big(\pcN+\delta+\pc\big)\Big(A_1^2(t)+\|\pc\prmd_{ttx}(t)\|^2\Big),\label{117.2a}
\end{gather}
and the constants $c$ and $C$ are independent of $\delta$, $\pc$ and $T$. Here $\mu$ is an arbitrary positive constant to be determined and $C_\mu$ is a generic constant which only depends on $\mu$.
\end{lemma}
\begin{proof} 
Multiplying the equation \eqref{15.1} with $k=0$ by $\pwnd_t$ and integrating the resultant equality by parts over $\Omega$ together with the boundary conditions \eqref{pbc}, we get
\begin{equation}\label{53.1}
\frac{d}{dt}\Xi_3(t)+\int_0^1\rmd^2\pwnd_t^2dx=\mathcal{I}_3(t),
\end{equation}
where the integral term $\mathcal{I}_3(t)$ is defined by
\begin{equation}
\mathcal{I}_3(t):=\int_0^1\bigg[-\bigg(\frac{2\swnd_x}{3}\pdl_t\pwnd+\frac{2\swnd}{3}\pdl_t\pwnd_x\bigg)+\frac{4\sdl\swnd}{3\rsmd}\prmd_x\pwnd_t+\pd{x}{}\mathcal{V}_0(t,x)\pwnd_t+H(t,x)\pwnd_t\bigg]dx,
\end{equation}
and can be estimated as
\begin{align}
\mathcal{I}_3(t)\leq&\int_0^1\pd{x}{}\mathcal{V}_0(t,x)\pwnd_tdx+(\mu+C\delta)\|\pdl_t(t)\|^2+C_\mu\|(\pwnd,\pwnd_x)(t)\|^2\notag\\
&+C\delta\|(\prmd_x,\pwnd_t)(t)\|^2+\mu\|\pwnd_t(t)\|^2+C_\mu\|H(t)\|^2\notag\\
\leq&\int_0^1\pd{x}{}\mathcal{V}_0(t,x)\pwnd_tdx+(\mu+\delta+\pc)\|\pwnd_t(t)\|^2\notag\\
&+\mu\|\prmd_x(t)\|^2+C_\mu A_{-1}^2(t)+C\big(\pcN+\delta+\pc\big)A_1^2(t),\label{55.2}
\end{align}
with the aid of the estimates \eqref{17.1}, \eqref{29.2a} and \eqref{P79-70.3a}, the H\"older and Young inequalities. Furthermore, applying the integration by parts and the equality \eqref{32.1+33.1} with $k=0$ to the first integral term on the right-side of \eqref{55.2}, we have
\begin{align}
\int_0^1\pd{x}{}\mathcal{V}_0(t,x)\pwnd_tdx=&-\int_0^1\mathcal{V}_0(t,x)\pwnd_{tx}dx\notag\\
=&-\int_0^1\bigg(\frac{\pc^2}{3}\pdl_{xx}-\frac{2\pc^2\sdl}{3\rsmd}\prmd_{xx}+\mathcal{K}(t,x)\bigg)\pwnd_{tx}dx\notag\\
\leq&C\pc\|(\pdl_{xx},\pc\prmd_{xx},\pwnd_{tx})(t)\|^2+\|\mathcal{K}(t)\|\|\pwnd_{tx}(t)\|\notag\\
\leq&C\pc\|(\prmd,\pdl,\prmd_t,\prmd_x,\prmd_{tx},\pc\prmd_{xx},\pwnd_{tx})(t)\|^2\notag\\
\leq&C\pc\big(A_1^2(t)+\|\pwnd_{tx}(t)\|^2\big),\label{54.2}
\end{align}
with the aid of the estimates \eqref{32.4+44.2} and \eqref{32.2b}, the H\"older and Cauchy-Schwarz inequalities. Substituting \eqref{54.2} into \eqref{55.2}, we obtain
\begin{equation}\label{55.3a}
\mathcal{I}_3(t)\leq(\mu+\delta+\pc)\|\pwnd_t(t)\|^2+C\Gamma_3(t).
\end{equation}
On the other hand, it is easy to see that
\begin{equation}\label{55.3b}
\int_0^1\rmd^2\pwnd_t^2dx\geq c\Pi_3(t).
\end{equation}
Substituting \eqref{55.3a} and \eqref{55.3b} into \eqref{53.1}, and letting $\mu$ and $\pcN+\delta+\pc$ sufficiently small, we have proved the desired estimate \eqref{116.2}.

Next, multiplying the equation \eqref{15.1} with $k=1$ by $\pwnd_t$ and integrating the resultant equality by parts over $\Omega$ together with the boundary conditions \eqref{pbc}, we get
\begin{equation}\label{ge56.1}
\frac{d}{dt}\Xi_4(t)+\frac{2}{3}\Pi_4(t)=\mathcal{I}_4(t),
\end{equation}
where the integral term $\mathcal{I}_4(t)$ is given by
\begin{align}
\mathcal{I}_4(t):=\int_0^1\bigg[&\rmd\prmd_t\pwnd_t^2+\bigg(\frac{2\swnd_x}{3}\pdl_t\pwnd_t+\frac{2\swnd}{3}\pdl_t\pwnd_{tx}\bigg)\notag\\
&+\frac{4\sdl\swnd}{3\rsmd}\prmd_{tx}\pwnd_t+\pd{x}{}\mathcal{V}_1(t,x)\pwnd_t+\big(\pd{t}{}H+L_1\big)(t,x)\pwnd_t\bigg]dx,\label{56.0}
\end{align}
and can be estimated as
\begin{align}
\mathcal{I}_4(t)\leq&\int_0^1\pd{x}{}\mathcal{V}_1(t,x)\pwnd_tdx+C\big(\pcN+\delta\big)\|(\prmd_{tx},\pwnd_t)(t)\|^2\notag\\
&+\mu\|\pwnd_{tx}(t)\|^2+C_\mu\|\pdl_t(t)\|^2+\big(\|\pd{t}{}H(t)\|+\|L_1(t)\|\big)\|\pwnd_t(t)\|\notag\\
\leq&\int_0^1\pd{x}{}\mathcal{V}_1(t,x)\pwnd_tdx+\big[\mu+C(\pcN+\delta)\big]\|\pwnd_{tx}(t)\|^2\notag\\
&+C_\mu\|(\pdl_t,\pwnd_t)(t)\|^2+C(\pcN+\delta)\|(\prmd_t,\prmd_{tt},\prmd_{tx})(t)\|^2,\label{59.2}
\end{align}
with the aid of the estimates \eqref{17.1}, \eqref{16.3+16.5} and \eqref{30.1a+30.2}, the H\"older and Young inequalities. Furthermore, applying the integration by parts and the equality \eqref{32.1+33.1} with $k=1$ to the first integral term on the right-side of \eqref{59.2}, we have
\begin{align}
\int_0^1\pd{x}{}\mathcal{V}_1(t,x)\pwnd_tdx=&-\int_0^1\mathcal{V}_1(t,x)\pwnd_{tx}dx\notag\\
=&-\int_0^1\bigg(\frac{\pc^2}{3}\pdl_{txx}-\frac{2\pc^2\sdl}{3\rsmd}\prmd_{txx}+\pd{t}{}\mathcal{K}(t,x)\bigg)\pwnd_{tx}dx\notag\\
\leq&C\pc\|(\pc\pdl_{txx},\pc\prmd_{txx},\pwnd_{tx})(t)\|^2+\|\pd{t}{}\mathcal{K}(t)\|\|\pwnd_{tx}(t)\|\notag\\
\leq&C\pc\big(\|\pwnd_{tx}(t)\|^2+\|\pdl_t(t)\|^2+A_1^2(t)+\|\pc\prmd_{ttx}(t)\|^2\big),\label{58.1}
\end{align}
with the aid of the estimates \eqref{38.2c+38.2d} and \eqref{33.2b}, the H\"older and Cauchy-Schwarz inequalities. Substituting \eqref{58.1} into \eqref{59.2} together with the estimate \eqref{P79-70.3a}, we obtain
\begin{align}\label{I4}
\mathcal{I}_4(t)\leq&\big[\mu+C(\pcN+\delta+\pc)\big]\|\pwnd_{tx}(t)\|^2+C(\pcN+\delta+\pc)\big(A_1^2(t)+\|\pc\prmd_{ttx}(t)\|^2\big)\notag\\
&+C_\mu\|(\pdl_t,\pwnd_t)(t)\|^2\notag\\
\leq&\big[\mu+C(\pcN+\delta+\pc)\big]\|\pwnd_{tx}(t)\|^2+C_\mu\Gamma_4(t).
\end{align}
Substituting \eqref{I4} into \eqref{ge56.1}, and letting $\mu$ and $\pcN+\delta+\pc$ sufficiently small, we have shown the desired estimate \eqref{117.2}.
\end{proof}

\subsection{Decay estimate}\label{Subsect.3.5}
Based on the basic estimate \eqref{25.1} and the higher order estimates \eqref{115.1}, \eqref{115.2}, \eqref{116.2} and \eqref{117.2}, it is not difficult to find that we have captured the strong enough dissipation mechanism to show the decay estimate \eqref{127.2} in Proposition \ref{prop1}.
\begin{proof}[\textbf{Proof of Proposition \ref{prop1}.}]
From the procedure
\begin{equation*}\label{118.1}
\eqref{25.1}+\beta\Big[\alpha\eqref{115.1}+\eqref{115.2}\Big]\Big|_{k=0}+\beta\Big[\eqref{116.2}+\beta\eqref{117.2}\Big]+\beta^3\Big[\alpha\eqref{115.1}+\eqref{115.2}\Big]\Big|_{k=1},
\end{equation*}
where $\alpha$ is the positive constant in \eqref{25.2} and $\beta$ is another positive constant, both of them will be determined later, we have the energy inequality
\begin{equation}\label{118.2}
\frac{d}{dt}\mathbb{E}(t)+\mathbb{D}(t)\leq0,\quad\forall t\in[0,T],
\end{equation}
where the total energy $\mathbb{E}(t)$ is defined by
\begin{equation}\label{119.2}
\mathbb{E}(t):=\Xi(t)+\beta\Big[\alpha\Xi_1^{(0)}(t)+\Xi_2^{(0)}(t)\Big]+\beta\Big[\Xi_3(t)+\beta\Xi_4(t)\Big]+\beta^3\Big[\alpha\Xi_1^{(1)}(t)+\Xi_2^{(1)}(t)\Big],
\end{equation}
and the total dissipation rate $\mathbb{D}(t)$ is given by
\begin{align}
\mathbb{D}(t):=&\Big[c\Pi(t)-C\Gamma(t)\Big]+\beta\Big\{\alpha\Big[c\Pi_1^{(0)}(t)-C\Gamma_1^{(0)}(t)\Big]+\Big[c\Pi_2^{(0)}(t)-C\Gamma_2^{(0)}(t)\Big]\Big\}\notag\\
&+\beta\Big\{\Big[c\Pi_3(t)-C\Gamma_3(t)\Big]+\beta\Big[c\Pi_4(t)-C\Gamma_4(t)\Big]\Big\}\notag\\
&+\beta^3\Big\{\alpha\Big[c\Pi_1^{(1)}(t)-C\Gamma_1^{(1)}(t)\Big]+\Big[c\Pi_2^{(1)}(t)-C\Gamma_2^{(1)}(t)\Big]\Big\}.\label{122.1}
\end{align}
Substituting the specific definitions \eqref{25.2}$\sim$\eqref{R}, \eqref{115.1a}, \eqref{115.2a}, \eqref{116.2a} and \eqref{117.2a} into \eqref{119.2} and \eqref{122.1}, and then taking $\alpha$, $\mu$, $\beta$ and $\pcN+\delta+\pc$ sufficiently small in the following order $0<\pcN+\delta+\pc\ll\beta^3\ll\beta^2\ll\beta\ll\mu\ll\alpha\ll1$, via the elaborate calculations, we obtain the estimates
\begin{equation}\label{121.2}
c\big(A_1^2(t)+\|\pwnd_t(t)\|^2\big)\leq\mathbb{E}(t)\leq C\big(A_1^2(t)+\|\pwnd_t(t)\|^2\big),
\end{equation}
and
\begin{align}
\mathbb{D}(t)\geq&c\big(A_1^2(t)+\|\pwnd_t(t)\|^2+\|(\pwnd_{tx},\pc\prmd_{ttx})(t)\|^2\big)\notag\\
\geq&c\big(A_1^2(t)+\|\pwnd_t(t)\|^2\big),\label{126.2}\\
\mathbb{D}(t)\leq&C\big(A_1^2(t)+\|\pwnd_t(t)\|^2+\|(\pwnd_{tx},\pc\prmd_{ttx})(t)\|^2\big),\notag
\end{align}
where the positive constants $c$ and $C$ are independent of $\delta$, $\pc$ and $T$.

Applying \eqref{121.2} and \eqref{126.2} to the inequality \eqref{118.2}, we see that there exists a positive constant $\gamma$ which is independent of $\delta$, $\pc$ and $T$ such that the following inequality holds,
\begin{equation}\label{127.1}
\frac{d}{dt}\mathbb{E}(t)+2\gamma\mathbb{E}(t)\leq0,\quad\forall t\in[0,T].
\end{equation}
Finally, applying Gronwall inequality to \eqref{127.1} and using the elliptic estimate \eqref{17.3} and the equivalent relations \eqref{121.2} and \eqref{46.5}, we have the desired decay estimate \eqref{127.2}.
\end{proof}

Once Proposition \ref{prop1} is proved, Theorem \ref{thm2} immediately follows.
\begin{proof}[\textbf{Proof of Theorem \ref{thm2}.}]
The existence of the global-in-time solution to the initial-boundary value problem \eqref{1dfqhd}$\sim$\eqref{bc} follows from the continuation argument with Corollary \ref{cor1} and Proposition \ref{prop1}. The decay estimate \eqref{de} is derived by the transformations $\md=\rmd^2$, $\smd=\rsmd^2$ and the estimate \eqref{127.2}. 
\end{proof}

\section{Semi-classical limit}\label{Sect.4}
In this section, we prove Theorem \ref{thm3} in Subsection \ref{Subsect.4.1} and Theorem \ref{thm4} in Subsection \ref{Subsect.4.2}, respectively. 

\subsection{Stationary case}\label{Subsect.4.1}
In this subsection, we discuss the semi-classical limit of the stationary solutions based on the existence and uniqueness results in Lemma \ref{lem1} and Theorem \ref{thm1}. Since both of the quantum stationary density $\pcsmd{\pc}$ and the limit one $\pcsmd{0}$ are non-flat, it is convenient to introduce the logarithmic transformations $\pcszmd{\pc}:=\ln \pcsmd{\pc}$ and $\pcszmd{0}:=\ln \pcsmd{0}$ in the following discussion. We also introduce the error variables as follows
\begin{equation}\label{ss7.1}
\pcpszmd:=\pcszmd{\pc}-\pcszmd{0},\quad \pcpsdl:=\pcsdl{\pc}-\pcsdl{0},\quad\pcpswnd:=\pcswnd{\pc}-\pcswnd{0},\quad\pcpsdws:=\pcsdws{\pc}-\pcsdws{0}.
\end{equation}
These quantities $(\pcszmd{\pc},\pcsdl{\pc},\pcswnd{\pc},\pcsdws{\pc})$, $(\pcszmd{0},\pcsdl{0},\pcswnd{0},\pcsdws{0})$ and $(\pcpszmd,\pcpsdl,\pcpswnd,\pcpsdws)$ satisfy the following properties
\begin{equation}\label{psbc}
\pcpszmd\in H_0^1(\Omega)\cap C^2(\overline{\Omega}),\quad \pcpswnd\in H_0^1(\Omega)\cap H^3(\Omega),\quad \pcpsdws\in H_0^1(\Omega)\cap C^2(\overline{\Omega}),
\end{equation}
\begin{equation}\label{psbcxx}
\bigg[\pcszmd{\pc}_{xx}+\frac{(\pcszmd{\pc}_x)^2}{2}\bigg](0)=\bigg[\pcszmd{\pc}_{xx}+\frac{(\pcszmd{\pc}_x)^2}{2}\bigg](1)=0,
\end{equation}
the estimates
\begin{equation}\label{ss9.3}
\ln c\leq\pcszmd{0}(x)\leq\ln C,\quad 0<c\leq\pcswnd{0}(x)\leq C,\quad|\pcsdl{0}|+\|\pcswnd{0}-\theta_{L}\|_3\leq C\delta,\quad |(\pcsmd{0},\pcsdws{0})|_2\leq C,
\end{equation}
\begin{gather}
2\ln b\leq\pcszmd{\pc}(x)\leq2\ln B,\quad 0<\frac{\theta_L}{2}\leq\pcswnd{\pc}(x)\leq \frac{3\theta_l}{2},\quad |\pcsdl{\pc}|+\|\pcswnd{\pc}-\theta_{L}\|_3\leq C\delta,\notag\\
\|\pcszmd{\pc}\|_2+\|(\pc\pd{x}{3}\pcszmd{\pc},\pc^2\pd{x}{4}\pcszmd{\pc})\|+|\pcsdws{\pc}|_2\leq C,\qquad\forall\pc\in(0,\pc_1], \label{ss9.2}
\end{gather}
and the equations
\begin{equation}\label{ss8.1a}
S[e^{\pcszmd{0}},\pcsdl{0},\pcswnd{0}]\pcszmd{0}_x+\pcswnd{0}_x=\pcsdws{0}_x-\pcsdl{0}e^{-\pcszmd{0}},
\end{equation}
\begin{equation}\label{ss8.1b}
S[e^{\pcszmd{\pc}},\pcsdl{\pc},\pcswnd{\pc}]\pcszmd{\pc}_x+\pcswnd{\pc}_x-\frac{\pc^2}{2}\Bigg[\pcszmd{\pc}_{xx}+\frac{(\pcszmd{\pc}_{x})^2}{2}\Bigg]_x=\pcsdws{\pc}_x-\pcsdl{\pc}e^{-\pcszmd{\pc}},
\end{equation}
\begin{equation}\label{ss8.1c}
S[e^{\pcszmd{\pc}},\pcsdl{\pc},\pcswnd{\pc}]\pcszmd{\pc}_x-S[e^{\pcszmd{0}},\pcsdl{0},\pcswnd{0}]\pcszmd{0}_x+\pcpswnd_{x}-\pcpsdws_x-\frac{\pc^2}{2}\Bigg[\pcszmd{\pc}_{xx}+\frac{(\pcszmd{\pc}_{x})^2}{2}\Bigg]_x=-\Big(\pcsdl{\pc}e^{-\pcszmd{\pc}}-\pcsdl{0}e^{-\pcszmd{0}}\Big),
\end{equation}
\begin{multline}\label{ss8.1d}
\pcsdl{\pc}\pcswnd{\pc}_x-\pcsdl{0}\pcswnd{0}_x-\frac{2}{3}\Big(\pcsdl{\pc}\pcswnd{\pc}\pcszmd{\pc}_x-\pcsdl{0}\pcswnd{0}\pcszmd{0}_x\Big)-\frac{2}{3}\pcpswnd_{xx}+\frac{\pc^2}{3}\pcsdl{\pc}\Big(\pcszmd{\pc}_{xxx}-2\pcszmd{\pc}_{xx}\pcszmd{\pc}_{x}\Big)\\
=\frac{1}{3}\Big[(\pcsdl{\pc})^2e^{-\pcszmd{\pc}}-(\pcsdl{0})^2e^{-\pcszmd{0}}\Big]-\Big[e^{\pcszmd{\pc}}(\pcswnd{\pc}-\theta_L)-e^{\pcszmd{0}}(\pcswnd{0}-\theta_L)\Big],
\end{multline}
\begin{equation}\label{ss8.1e}
\pcpsdws_{xx}=e^{\pcszmd{\pc}}-e^{\pcszmd{0}},\qquad\forall\pc\in (0,\pc_1]
\end{equation}
due to the boundary conditions \eqref{sbc} and \eqref{0sbc}, the estimates \eqref{0se} and \eqref{145.1},  and the equations \eqref{1dsfqhd} and \eqref{1dsfhd}. 

\begin{proof}[\textbf{Proof of Theorem \ref{thm3}.}]
Firstly, we prove the convergence rate \eqref{ss39.1a} in $\pc\in (0,\pc_1]$. Note that if $\delta$ is small enough, we have known that the quantum stationary current density $\pcsdl{\pc}=J[e^{\pcszmd{\pc}},\pcswnd{\pc}]$ is defined by the explicit formula \eqref{104.1}. Furthermore, the limit stationary current density $\pcsdl{0}$ can also be written by the same formula \eqref{104.1} as $\pcsdl{0}=J[e^{\pcszmd{0}},\pcswnd{0}]$. Therefore, the following estimate 
\begin{align}
|\pcpsdl|=|\pcsdl{\pc}-\pcsdl{0}|&\leq C\Big(\delta\|\pcpszmd\|+\|\pcpswnd_x\|\Big)\notag\\
&\leq C\Big(\delta\|\pcpszmd_x\|+\|\pcpswnd_x\|\Big).\label{ss10.2}
\end{align}
follows from the straightforward computations with the formula \eqref{104.1}, the estimates \eqref{ss9.3} and \eqref{ss9.2}.

Multiplying the equation \eqref{ss8.1c} by $\pcpszmd_x$ and integrating the resultant equality over the domain $\Omega$, we obtain
\begin{multline}\label{ss10.5}
\int_0^1\Big(S[e^{\pcszmd{\pc}},\pcsdl{\pc},\pcswnd{\pc}]\pcszmd{\pc}_x-S[e^{\pcszmd{0}},\pcsdl{0},\pcswnd{0}]\pcszmd{0}_x\Big)\pcpszmd_xdx-\int_0^1\pcpsdws_x\pcpszmd_xdx\\
=\frac{\pc^2}{2}\int_0^1\Bigg[\pcszmd{\pc}_{xx}+\frac{(\pcszmd{\pc}_{x})^2}{2}\Bigg]_x\pcpszmd_xdx-\int_0^1\pcpswnd_{x}\pcpszmd_xdx-\int_0^1\Big(\pcsdl{\pc}e^{-\pcszmd{\pc}}-\pcsdl{0}e^{-\pcszmd{0}}\Big)\pcpszmd_xdx.
\end{multline}
By virtue of integration by parts, the boundary conditions \eqref{psbc} and \eqref{psbcxx}, the equation \eqref{ss8.1e}, the Young and H\"older inequalities, the mean value theorem and the estimates \eqref{ss9.3}, \eqref{ss9.2} and \eqref{ss10.2}, the left-side of \eqref{ss10.5} can be estimated as follows
\allowdisplaybreaks
\begin{align}
\eqref{ss10.5}_l&=\int_0^1\Big(\pcsS{\pc}\pcszmd{\pc}_x-\pcsS{0}\pcszmd{0}_x\Big)\pcpszmd_xdx+\int_0^1\pcpsdws_{xx}\pcpszmd dx\notag\\
&=\int_0^1\Big[\Big(\pcsS{\pc}-\pcsS{0}\Big)\pcszmd{\pc}_x+\pcsS{0}\pcpszmd_x\Big]\pcpszmd_xdx+\int_0^1\underbrace{\Big(e^{\pcszmd{\pc}}-e^{\pcszmd{0}}\Big)(\pcszmd{\pc}-\pcszmd{0})}_{\geq0} dx\notag\\
&\geq\frac{\theta_L}{4}\|\pcpszmd_x\|^2+\int_0^1\Big(\pcsS{\pc}-\pcsS{0}\Big)\pcszmd{\pc}_x\pcpszmd_xdx\notag\\
&\geq\frac{\theta_L}{4}\|\pcpszmd_x\|^2-\mu\|\pcpszmd_x\|^2-C_\mu\|\pcpswnd\|^2-C\delta\|(\pcpszmd_x,\pcpswnd_x)\|^2\notag\\
&\geq\frac{\theta_L}{8}\|\pcpszmd_x\|^2-C\|\pcpswnd\|_1^2,\label{ss11.1}
\end{align}
where we have used the notation $\pcsS{\pc}:=S[e^{\pcszmd{\pc}},\pcsdl{\pc},\pcswnd{\pc}]$ for any $\pc\in[0,\pc_1]$ and the following estimate
\begin{equation}\label{ss12.1-2-3}
|(\pcsS{\pc}-\pcsS{0})(x)|\leq |\pcpswnd(x)|+C\delta|\pcpsdl|+C\delta^2|\pcpszmd(x)|,\quad\forall x\in\Omega.
\end{equation}
Similarly, we further estimate the right-side of \eqref{ss10.5} as follows
\begin{align}
\eqref{ss10.5}_r=&\frac{\pc^2}{2}\int_0^1\Bigg[\pcszmd{\pc}_{xx}+\frac{(\pcszmd{\pc}_{x})^2}{2}\Bigg]_x\pcpszmd_xdx-\int_0^1\pcpswnd_{x}\pcpszmd_xdx-\int_0^1\Big(\pcsdl{\pc}e^{-\pcszmd{\pc}}-\pcsdl{0}e^{-\pcszmd{0}}\Big)\pcpszmd_xdx\notag\\
\leq&-\frac{\pc^2}{2}\int_0^1\Bigg[\pcszmd{\pc}_{xx}+\frac{(\pcszmd{\pc}_{x})^2}{2}\Bigg]\pcpszmd_{xx}dx+\|\pcpszmd_x\|\|\pcpswnd_{x}\|+C|\pcpsdl|\|\pcpszmd_x\|+C\delta\|\pcpszmd\|\|\pcpszmd_x\|\notag\\
\leq&C\pc^2\int_0^1\Big(|\pcszmd{\pc}_{xx}|+|\pcszmd{\pc}_{x}|^2\Big)|\pcpszmd_{xx}|dx+\big(C\delta+\mu\big)\|\pcpszmd_x\|^2+C_\mu\|\pcpswnd_{x}\|^2\notag\\
\leq&C\pc^2\underbrace{\Big(\|\pcszmd{\pc}_{xx}\|+1\Big)\|\pcpszmd_{xx}\|}_{\leq C}+\big(C\delta+\mu\big)\|\pcpszmd_x\|^2+C_\mu\|\pcpswnd_{x}\|^2\notag\\
\leq&C\pc^2+\big(C\delta+\mu\big)\|\pcpszmd_x\|^2+C_\mu\|\pcpswnd_{x}\|^2.\label{ss12.4-5-ss13.1}
\end{align}
Substituting \eqref{ss11.1} and  \eqref{ss12.4-5-ss13.1} into \eqref{ss10.5}, and letting $\delta$ and $\mu$ small enough, we have
\begin{equation}\label{ss13.3}
\|\pcpszmd\|_1^2\leq C\|\pcpswnd\|_1^2+C\pc^2,
\end{equation}
where we have used the Poincar\'e inequality $\|\pcpszmd\|\leq C\|\pcpszmd_x\|$.

Next, multiplying the equation \eqref{ss8.1d} by $\pcpswnd$ and integrating the resultant equality over the domain $\Omega$, we obtain
\begin{align}\label{ss14.1}
-\frac{2}{3}\int_0^1\pcpswnd_{xx}&\pcpswnd dx+\int_0^1\Big[e^{\pcszmd{\pc}}(\pcswnd{\pc}-\theta_L)-e^{\pcszmd{0}}(\pcswnd{0}-\theta_L)\Big]\pcpswnd dx\notag\\
&=-\frac{\pc^2}{3}\pcsdl{\pc}\int_0^1\Big(\pcszmd{\pc}_{xxx}-2\pcszmd{\pc}_{xx}\pcszmd{\pc}_{x}\Big)\pcpswnd dx+\frac{2}{3}\int_0^1\Big(\pcsdl{\pc}\pcswnd{\pc}\pcszmd{\pc}_x-\pcsdl{0}\pcswnd{0}\pcszmd{0}_x\Big)\pcpswnd dx\notag\\
&\qquad\qquad-\int_0^1\Big(\pcsdl{\pc}\pcswnd{\pc}_x-\pcsdl{0}\pcswnd{0}_x\Big)\pcpswnd dx+\frac{1}{3}\int_0^1\Big[(\pcsdl{\pc})^2e^{-\pcszmd{\pc}}-(\pcsdl{0})^2e^{-\pcszmd{0}}\Big]\pcpswnd dx\\
&=I_1+I_2+I_3+I_4.\notag
\end{align}
By the same fashion used to derive the estimate \eqref{ss13.3}, the left-side of \eqref{ss14.1} can be estimated as
\begin{align}
\eqref{ss14.1}_l&=\frac{2}{3}\|\pcpswnd_{x}\|^2+\int_0^1\Big[\Big(e^{\pcszmd{\pc}}-e^{\pcszmd{0}}\Big)(\pcswnd{\pc}-\theta_L)+e^{\pcszmd{0}}\pcpswnd\Big]\pcpswnd dx\notag\\
&\geq\frac{2}{3}\|\pcpswnd_{x}\|^2+c\|\pcpswnd\|^2-C\delta\|(\pcpszmd,\pcpswnd)\|^2\notag\\
&\geq c\|\pcpswnd\|_1^2-C\delta\|\pcpszmd\|^2.\label{ss14.2-ss18.3}
\end{align}
We further estimate the integrals $I_i$ $(i=1,2,3,4)$ on the right-side of \eqref{ss14.1} one by one,
\begin{align}
I_1&=\frac{\pc^2}{3}\pcsdl{\pc}\int_0^1\pcszmd{\pc}_{xx}\pcpswnd_xdx+\frac{2\pc^2}{3}\pcsdl{\pc}\int_0^1\pcszmd{\pc}_{x}\pcszmd{\pc}_{xx}\pcpswnd dx\notag\\
&\leq \frac{1}{3}\pc^2|\pcsdl{\pc}|\|\pcszmd{\pc}_{xx}\|\|\pcpswnd_x\|+\frac{2}{3}\pc^2|\pcsdl{\pc}||\pcszmd{\pc}_{x}|_0\|\pcszmd{\pc}_{xx}\|\|\pcpswnd\|\notag\\
&\leq C\pc^2\|\pcpswnd\|_1\leq C\pc^2\|\pcswnd{\pc}-\theta_L+\theta_L-\pcswnd{0}\|_1\notag\\
&\leq C\Big(\|\pcswnd{\pc}-\theta_L\|_1+\|\pcswnd{0}-\theta_L\|_1\Big)\pc^2\notag\\
&\leq C\pc^2.\label{ss18.1}
\end{align}
It is easy to estimate $I_3+I_4$ by the standard computations, that is,
\begin{equation}\label{ss15.1-ss18.2}
I_3+I_4\leq C\delta\|(\pcpszmd,\pcpswnd)\|_1^2.
\end{equation}
However, we need to pay more attention to the integral $I_2$ due to the non-flatness of $\pcszmd{\pc}$,
\begin{align}
I_2&=\frac{2}{3}\int_0^1\Big(\pcsdl{\pc}\pcswnd{\pc}\pcszmd{\pc}_x-\pcsdl{0}\pcswnd{0}\pcszmd{0}_x\Big)\pcpswnd dx\notag\\
&=\frac{2}{3}\int_0^1\Big(\pcpsdl\pcswnd{\pc}\pcszmd{\pc}_x+\pcsdl{0}\pcpswnd\pcszmd{\pc}_x+\pcsdl{0}\pcswnd{0}\pcpszmd_x\Big)\pcpswnd dx\notag\\
&=\frac{2}{3}\int_0^1\Big[\pcpsdl\big(\pcswnd{\pc}-\theta_L+\theta_L\big)\pcszmd{\pc}_x+\pcsdl{0}\pcszmd{\pc}_x\pcpswnd+\pcsdl{0}\pcswnd{0}\pcpszmd_x\Big]\pcpswnd dx\notag\\
&=\frac{2\theta_L}{3}\pcpsdl\int_0^1\pcpswnd\pcszmd{\pc}_xdx+\frac{2}{3}\int_0^1\Big[\big(\pcswnd{\pc}-\theta_L\big)\pcszmd{\pc}_x\pcpsdl+\pcsdl{0}\pcszmd{\pc}_x\pcpswnd+\pcsdl{0}\pcswnd{0}\pcpszmd_x\Big]\pcpswnd dx\notag\\
&=-\frac{2\theta_L}{3}\pcpsdl\int_0^1\pcpswnd_x\pcszmd{\pc}dx+\frac{2}{3}\int_0^1\Big[\big(\pcswnd{\pc}-\theta_L\big)\pcszmd{\pc}_x\pcpsdl+\pcsdl{0}\pcszmd{\pc}_x\pcpswnd+\pcsdl{0}\pcswnd{0}\pcpszmd_x\Big]\pcpswnd dx\notag\\
&\leq-\frac{2\theta_L}{3}\pcpsdl\int_0^1\pcpswnd_x\pcszmd{\pc}dx\notag\\
&\qquad+\frac{2}{3}\Big(|\pcswnd{\pc}-\theta_L|_0|\pcszmd{\pc}_x|_0|\pcpsdl|\|\pcpswnd\|+|\pcsdl{0}||\pcszmd{\pc}_x|_0\|\pcpswnd\|^2+|\pcsdl{0}||\pcswnd{0}|_0\|\pcpszmd_x\|\|\pcpswnd\|\Big)\notag\\
&\leq-\frac{2\theta_L}{3}\pcpsdl\int_0^1\pcpswnd_x\pcszmd{\pc}dx+C\delta\|(\pcpszmd,\pcpswnd)\|_1^2\notag\\
&=-\frac{2\theta_L}{3}\big(\pcsdl{\pc}-\pcsdl{0}\big)\int_0^1\pcpswnd_x\pcszmd{\pc}dx+C\delta\|(\pcpszmd,\pcpswnd)\|_1^2\notag\\
&=-\frac{2\theta_L}{3}\Bigg[\frac{2\big(\bar{b}+\int_0^1\pcswnd{\pc}_{x}\pcszmd{\pc}dx\big)}{\pcsK{\pc}}-\frac{2\big(\bar{b}+\int_0^1\pcswnd{0}_{x}\pcszmd{0}dx\big)}{\pcsK{0}}\Bigg]\int_0^1\pcpswnd_x\pcszmd{\pc}dx+C\delta\|(\pcpszmd,\pcpswnd)\|_1^2\notag\\
&=-\frac{4\theta_L}{3}\Bigg[\frac{1}{\pcsK{\pc}}\int_0^1\Big(\pcswnd{\pc}_{x}\pcszmd{\pc}-\pcswnd{0}_{x}\pcszmd{0}\Big)dx+\bigg(\bar{b}+\int_0^1\pcswnd{0}_{x}\pcszmd{0}dx\bigg)\bigg(\frac{1}{\pcsK{\pc}}-\frac{1}{\pcsK{0}}\bigg)\Bigg]\int_0^1\pcpswnd_x\pcszmd{\pc}dx\notag\\
&\qquad+C\delta\|(\pcpszmd,\pcpswnd)\|_1^2\notag\\
&=-\frac{4\theta_L}{3}\Bigg[\frac{1}{\pcsK{\pc}}\int_0^1\Big(\pcpswnd_{x}\pcszmd{\pc}+\pcswnd{0}_{x}\pcpszmd\Big)dx-\bigg(\bar{b}+\int_0^1\pcswnd{0}_{x}\pcszmd{0}dx\bigg)\frac{\pcsK{\pc}-\pcsK{0}}{\pcsK{\pc}\pcsK{0}}\Bigg]\int_0^1\pcpswnd_x\pcszmd{\pc}dx\notag\\
&\qquad+C\delta\|(\pcpszmd,\pcpswnd)\|_1^2\notag\\
&=\underbrace{-\frac{4\theta_L}{3\pcsK{\pc}}\bigg(\int_0^1\pcpswnd_{x}\pcszmd{\pc}dx\bigg)^2}_{\leq0}\notag\\
&\qquad-\frac{4\theta_L}{3\pcsK{\pc}}\Bigg[\int_0^1\pcswnd{0}_{x}\pcpszmd dx-\frac{\big(\bar{b}+\int_0^1\pcswnd{0}_{x}\pcszmd{0}dx\big)}{\pcsK{0}}\Big(\pcsK{\pc}-\pcsK{0}\Big)\Bigg]\int_0^1\pcpswnd_x\pcszmd{\pc}dx\notag\\
&\qquad+C\delta\|(\pcpszmd,\pcpswnd)\|_1^2\notag\\
&\leq-\frac{4\theta_L}{3\pcsK{\pc}}\Bigg[\int_0^1\pcswnd{0}_{x}\pcpszmd dx-\frac{\pcsdl{0}}{2}\Big(\pcsK{\pc}-\pcsK{0}\Big)\Bigg]\int_0^1\pcpswnd_x\pcszmd{\pc}dx+C\delta\|(\pcpszmd,\pcpswnd)\|_1^2\notag\\
&\leq C\delta\Big(\|\pcpszmd\|+|\pcsK{\pc}-\pcsK{0}|\Big)\|\pcpswnd_x\|+C\delta\|(\pcpszmd,\pcpswnd)\|_1^2\notag\\
&\leq C\delta\|(\pcpszmd,\pcpswnd)\|_1^2,\label{ss17.2}
\end{align}
where we have adopted the notation $\pcsK{\pc}:=K[e^{\pcszmd{\pc}},\pcswnd{\pc}]$ (see formula \eqref{104.1}) for any $\pc\in[0,\pc_1]$ and the following estimates
\begin{equation}\label{Pss4B-ss10.2bc}
0<c\leq\frac{1}{\pcsK{\pc}}\leq C,\qquad |\pcsK{\pc}-\pcsK{0}|\leq C\Big(\|\pcpszmd\|+\|\pcpswnd_x\|\Big),\qquad\forall\pc\in[0,\pc_1]
\end{equation}
which follow from the straightforward but tedious computations. Inserting the estimates \eqref{ss14.2-ss18.3}$\sim$\eqref{ss17.2} into \eqref{ss14.1}, and letting $\delta\ll1$, we have
\begin{equation}\label{ss19.2}
\|\pcpswnd\|_1^2\leq C\delta\|\pcpszmd\|_1^2+C\pc^2.
\end{equation}
Moreover, substituting \eqref{ss19.2} into \eqref{ss13.3}, and letting $\delta$ small enough, we get
\begin{equation}\label{ss19.4}
\|\pcpszmd\|_1^2\leq C\pc^2,\quad\forall\pc\in(0,\pc_1].
\end{equation}
Combining \eqref{ss19.4} with \eqref{ss19.2}, \eqref{ss10.2} and the elliptic estimate $\|\pcpsdws\|_3\leq C\|\pcpszmd\|_1$, we obtain
\begin{equation}\label{ss20.1}
\|\pcpszmd\|_1+|\pcpsdl|+\|\pcpswnd\|_1+\|\pcpsdws\|_3\leq C\pc.
\end{equation}

Next, we solve $\pcpswnd_{xx}$ from the equation \eqref{ss8.1d} and directly take the $L^2$-norm of the resultant equality, the standard but tedious computations yield the following estimate
\begin{equation}\label{ss21.1}
\|\pcpswnd_{xx}\|\leq C\pc\|\pc\pcszmd{\pc}_{xxx}\|+C\pc^2\|\pcszmd{\pc}_{xx}\|+C\Big(\|\pcpszmd\|_1+|\pcpsdl|+\|\pcpswnd\|_1\Big)\leq C\pc.
\end{equation}
Adding \eqref{ss20.1} and \eqref{ss21.1} up, we have
\begin{equation}\label{ss22.2}
\|\pcpszmd\|_1+|\pcpsdl|+\|\pcpswnd\|_2+\|\pcpsdws\|_3\leq C\pc,\quad\forall\pc\in(0,\pc_1].
\end{equation}
By using the exponential transformations $\pcsmd{\pc}=e^{\pcszmd{\pc}}$, $\pcsmd{0}=e^{\pcszmd{0}}$ and the above estimate \eqref{ss22.2}, we have showed the algebraic convergence rate \eqref{ss39.1a}.

Now, we begin to show the convergence \eqref{ss39.1b}. Firstly, differentiating the equation \eqref{ss8.1d} once and solving $\pcpswnd_{xxx}$ from the resultant equation, and taking the $L^2$-norm of the expression of $\pcpswnd_{xxx}$, then these computations yield the following estimate
\begin{align}
\|\pcpswnd_{xxx}\|\leq&C\Big(\|\pc^2\pcszmd{\pc}_{xxxx}\|+\underbrace{\pc\|\pc\pcszmd{\pc}_{xxx}\|+\pc^2\|\pcszmd{\pc}_{xx}\|+\|\pcpszmd\|_1+|\pcpsdl|+\|\pcpswnd\|_2}_{\leq C\pc}+\|\pcpszmd_{xx}\|\Big)\notag\\
\leq&C\Big(\|\pcpszmd_{xx}\|+\|\pc^2\pcszmd{\pc}_{xxxx}\|\Big) +C\pc.\label{ss22.1c}
\end{align}
Adding the elliptic estimate $\|\pcpsdws_{xxxx}\|\leq C\|\pcpszmd_{xx}\|$ to the above estimate \eqref{ss22.1c}, we have
\begin{equation}\label{wnd3+dws4}
\|(\pcpswnd_{xxx},\pcpsdws_{xxxx})\|\leq C\Big(\|\pcpszmd_{xx}\|+\|\pc^2\pcszmd{\pc}_{xxxx}\|\Big) +C\pc,\quad\forall\pc\in(0,\pc_1].
\end{equation}

In order to complete the proof, we need to establish the convergence results $\|\pcpszmd_{xx}\|\rightarrow0$, $\|\pc\pcszmd{\pc}_{xxx}\|\rightarrow0$ and $\|\pc^2\pcszmd{\pc}_{xxxx}\|\rightarrow0$ as $\pc\rightarrow0$. To this end, we first show $\|\pcpszmd_{xx}\|$ converges to zero as $\pc$ tends to zero. From the boundedness \eqref{ss9.2} of $\|\pcszmd{\pc}\|_2$ and the strong convergence \eqref{ss22.2}, we have
\begin{equation}\label{ss23.1}
\pcszmd{\pc}_{xx}\rightharpoonup\pcszmd{0}_{xx}\quad\text{in}\ L^2(\Omega)\ \text{weakly as}\ \pc\rightarrow0.
\end{equation}
However, we need to improve the above weak convergence into strong convergence. For this purpose, differentiating the equation \eqref{ss8.1b} once, multiplying the resultant equality by $\pcszmd{\pc}_{xx}+(\pcszmd{\pc}_x)^2/2$ and integrating the result over the domain $\Omega$, the integration by parts yields
\begin{equation}\label{ss26.1}
\int_0^1\pcsS{\pc}(\pcszmd{\pc}_{xx})^2dx+\frac{\pc^2}{2}\int_0^1\bigg\{\bigg[\pcszmd{\pc}_{xx}+\frac{(\pcszmd{\pc}_{x})^2}{2}\bigg]_x\bigg\}^2dx=\pcR{\pc},\quad\forall\pc\in(0,\pc_1],
\end{equation}
where
\begin{multline}\label{ss26.2}
\pcR{\pc}:=-\int_0^1\pcsS{\pc}\pcszmd{\pc}_{xx}\frac{(\pcszmd{\pc}_x)^2}{2}dx-\int_0^1\pcsS{\pc}_x\pcszmd{\pc}_x\bigg[\pcszmd{\pc}_{xx}+\frac{(\pcszmd{\pc}_{x})^2}{2}\bigg]dx\\
+\int_0^1\bigg[\pcsdws{\pc}_{xx}-\Big(\pcsdl{\pc}e^{-\pcszmd{\pc}}\Big)_x-\pcswnd{\pc}_{xx}\bigg]\bigg[\pcszmd{\pc}_{xx}+\frac{(\pcszmd{\pc}_{x})^2}{2}\bigg]dx,\qquad\forall\pc\in[0,\pc_1].
\end{multline}
Similarly, differentiating the equation \eqref{ss8.1a} once, multiplying the resultant equality by $\pcszmd{0}_{xx}+(\pcszmd{0}_{x})^2/2$ and integrating the result over $\Omega$, we have
\begin{equation}\label{ss26.4}
\int_0^1\pcsS{0}(\pcszmd{0}_{xx})^2dx=\pcR{0},\quad\text{where}\ \pcR{0}\ \text{is given by}\ \eqref{ss26.2}\ \text{with}\ \pc=0.
\end{equation}
By using the estimates \eqref{ss9.3}, \eqref{ss9.2}, \eqref{ss22.2}, the weak convergence result \eqref{ss23.1} and the standard computations, we obtain
\begin{align}
0\leq|\pcR{\pc}-\pcR{0}|\leq&C\Big(\|\pcpszmd\|_1+|\pcpsdl|+\|\pcpswnd\|_2+\|\pcpsdws_{xx}\|\Big)\notag\\
&\quad+\bigg|\int_0^1\underbrace{\bigg[\pcsS{0}\frac{(\pcszmd{0}_x)^2}{2}+\pcsS{0}\pcszmd{0}_x+\pcswnd{0}_{xx}+\pcsdws{0}_{xx}+\Big(\pcsdl{0}e^{-\pcszmd{0}}\Big)_x\bigg]}_{=:f^0\in L^2(\Omega)}\big(\pcszmd{\pc}_{xx}-\pcszmd{0}_{xx}\big)dx\bigg|\notag\\
\leq&C\pc+\Big|\big<f^0, \pcszmd{\pc}_{xx}-\pcszmd{0}_{xx}\big>_{L^2(\Omega)}\Big|\rightarrow0\quad\text{as}\ \pc\rightarrow0.\label{ss31.2}
\end{align}
Combining \eqref{ss26.4} with \eqref{ss31.2}, we have
\begin{equation}\label{ss31.3}
\lim_{\pc\rightarrow0}\pcR{\pc}=\int_0^1\pcsS{0}(\pcszmd{0}_{xx})^2dx.
\end{equation}
On the other hand, owing to the estimates \eqref{ss12.1-2-3}, \eqref{ss9.2}, the Sobolev inequality, we obtain
\begin{align}
0\leq\bigg|\int_0^1\big(\pcsS{\pc}-\pcsS{0}\big)(\pcszmd{\pc}_{xx})^2dx\bigg|\leq&\big|\big(\pcsS{\pc}-\pcsS{0}\big)\big|_0\|\pcszmd{\pc}_{xx}\|^2\notag\\
\leq&C\Big(|\pcpszmd|_0+|\pcpsdl|+|\pcpswnd|_0\Big)\notag\\
\leq&C\Big(\|\pcpszmd\|_1+|\pcpsdl|+\|\pcpswnd\|_1\Big)\notag\\
\leq&C\pc\rightarrow0\quad\text{as}\ \pc\rightarrow0.\label{ss32.2}
\end{align}
Combining the limits \eqref{ss31.3}, \eqref{ss32.2} with the equality \eqref{ss26.1}, we have
\begin{align}
\Bigg\{\limsup_{\pc\rightarrow0}\bigg[\int_0^1\pcsS{0}(\pcszmd{\pc}_{xx})^2dx\bigg]^{1/2}\Bigg\}^2&\leq\limsup_{\pc\rightarrow0}\int_0^1\pcsS{0}(\pcszmd{\pc}_{xx})^2dx\notag\\
&=\lim_{\pc\rightarrow0}\int_0^1\big(\pcsS{\pc}-\pcsS{0}\big)(\pcszmd{\pc}_{xx})^2dx+\limsup_{\pc\rightarrow0}\int_0^1\pcsS{0}(\pcszmd{\pc}_{xx})^2dx\notag\\
&=\limsup_{\pc\rightarrow0}\int_0^1\big(\pcsS{\pc}-\pcsS{0}+\pcsS{0}\big)(\pcszmd{\pc}_{xx})^2dx\notag\\
&=\limsup_{\pc\rightarrow0}\int_0^1\pcsS{\pc}(\pcszmd{\pc}_{xx})^2dx\notag\\
&\leq\limsup_{\pc\rightarrow0}\pcR{\pc}=\lim_{\pc\rightarrow0}\pcR{\pc}=\int_0^1\pcsS{0}(\pcszmd{0}_{xx})^2dx,\label{ss33.3}
\end{align}
where we have used the non-negativity of the second term on the left-side of the equality \eqref{ss26.1}. Motivated by \eqref{ss33.3}, we choose $\pcsS{0}$ as the weight to define a weighted-$L^2$ space as follows
\begin{equation}\label{ss34.1a}
L^2_{\pcsS{0}}(\Omega):=\Bigg\{f:\Omega\rightarrow\mathbb{R} \text{ is measurable}\ \Bigg|\ \int_0^1\pcsS{0}|f|^2dx<+\infty\Bigg\}
\end{equation}
with the inner product
\begin{equation}\label{ss34.1b}
\big<f,g\big>_{L^2_{\pcsS{0}}(\Omega)}:=\int_0^1\pcsS{0}fgdx
\end{equation}
and the associated norm
\begin{equation}\label{ss34.1c}
\|f\|_{L^2_{\pcsS{0}}(\Omega)}:=\Bigg(\int_0^1\pcsS{0}|f|^2\;dx\Bigg)^{1/2}.
\end{equation}
Since the weight function $\pcsS{0}$ is strictly positive and continuous, then the weighted-$L^2$ space $L^2_{\pcsS{0}}(\Omega)$ is a Hilbert space and the weak convergence \eqref{ss23.1} implies
\begin{equation}\label{ss34.2}
\pcszmd{\pc}_{xx}\rightharpoonup\pcszmd{0}_{xx}\quad\text{in}\ L^2_{\pcsS{0}}(\Omega)\ \text{weakly as}\ \pc\rightarrow0.
\end{equation}
Furthermore, we can rewrite the inequality \eqref{ss33.3} in terms of the norm defined in \eqref{ss34.1c} as
\begin{equation}\label{ss33.3wn}
\limsup_{\pc\rightarrow0}\|\pcszmd{\pc}_{xx}\|_{L^2_{\pcsS{0}}(\Omega)}\leq\|\pcszmd{0}_{xx}\|_{L^2_{\pcsS{0}}(\Omega)}.
\end{equation}
The weak convergence \eqref{ss34.2} together with \eqref{ss33.3wn} implies the strong convergence
\begin{equation}\label{ss34.3}
\pcszmd{\pc}_{xx}\rightarrow\pcszmd{0}_{xx}\quad\text{in}\ L^2_{\pcsS{0}}(\Omega)\ \text{strongly as}\ \pc\rightarrow0,
\end{equation}
which immediately implies
\begin{equation}\label{ss34.5}
\|\pcpszmd_{xx}\|\rightarrow0,\quad\text{as}\ \pc\rightarrow0.
\end{equation}

In addition, we prove $\|\pc\pcszmd{\pc}_{xxx}\|$ converges to zero as $\pc$ tends to zero. From the strong convergence \eqref{ss34.3}, we can  directly deduce that
\begin{equation}\label{ss34.6}
\lim_{\pc\rightarrow0}\int_0^1\pcsS{0}(\pcszmd{\pc}_{xx})^2dx=\int_0^1\pcsS{0}(\pcszmd{0}_{xx})^2dx.
\end{equation}
Combining \eqref{ss32.2} with \eqref{ss34.6}, we have
\begin{equation}\label{ss35.1}
\lim_{\pc\rightarrow0}\int_0^1\pcsS{\pc}(\pcszmd{\pc}_{xx})^2dx=\int_0^1\pcsS{0}(\pcszmd{0}_{xx})^2dx.
\end{equation}
Letting $\pc\rightarrow0$ in the equality \eqref{ss26.1}, and using the limit results \eqref{ss31.3} and \eqref{ss35.1}, we can easily see that
\begin{equation}\label{ss35.4}
\pc\bigg\|\bigg[\pcszmd{\pc}_{xx}+\frac{(\pcszmd{\pc}_{x})^2}{2}\bigg]_x\bigg\|\rightarrow0,\quad\text{as}\ \pc\rightarrow0.
\end{equation}
Therefore, 
\begin{align}
0\leq\|\pc\pcszmd{\pc}_{xxx}\|=&\bigg\|\pc\bigg\{\bigg[\pcszmd{\pc}_{xx}+\frac{(\pcszmd{\pc}_{x})^2}{2}\bigg]_x-\pcszmd{\pc}_{x}\pcszmd{\pc}_{xx}\bigg\}\bigg\|\notag\\
\leq&\pc\bigg\|\bigg[\pcszmd{\pc}_{xx}+\frac{(\pcszmd{\pc}_{x})^2}{2}\bigg]_x\bigg\|+C\pc\rightarrow0,\quad\text{as}\ \pc\rightarrow0.\label{ss36.1}
\end{align}

Finally, we show $\|\pc^2\pcszmd{\pc}_{xxxx}\|$ converges to zero as $\pc$ tends to zero. Differentiating the equation \eqref{ss8.1b} once, solving the quantum term from the resultant equality and taking the $L^2$-norm of this quantum term, we obtain
\begin{equation}\label{ss36.2}
\frac{\pc^2}{2}\bigg\|\bigg[\pcszmd{\pc}_{xx}+\frac{(\pcszmd{\pc}_{x})^2}{2}\bigg]_{xx}\bigg\|=\pcr{\pc},\quad\forall\pc\in(0,\pc_1],
\end{equation}
where
\begin{equation}\label{ss36.3}
\pcr{\pc}:=\Big\|\Big(\pcsS{\pc}\pcszmd{\pc}_x\Big)_x+\pcswnd{\pc}_{xx}-\pcsdws{\pc}_{xx}+\Big(\pcsdl{\pc}e^{-\pcszmd{\pc}}\Big)_x\Big\|,\quad\forall\pc\in[0,\pc_1].
\end{equation}
The standard computations yield
\begin{align}
0\leq\pcr{\pc}=\pcr{\pc}-\pcr{0}\leq&\Big\|\Big(\pcsS{\pc}\pcszmd{\pc}_x\Big)_x-\Big(\pcsS{0}\pcszmd{0}_x\Big)_x+\pcpswnd_{xx}-\pcpsdws_{xx}+\Big(\pcsdl{\pc}e^{-\pcszmd{\pc}}\Big)_x-\Big(\pcsdl{0}e^{-\pcszmd{0}}\Big)_x\Big\|\notag\\
\leq&C\Big(\|\pcpszmd\|_1+|\pcpsdl|+\|\pcpswnd_{xx}\|+\|\pcpsdws_{xx}\|\Big)+C\|\pcpszmd_{xx}\|\notag\\
\leq&C\pc+C\|\pcpszmd_{xx}\|\rightarrow0,\quad\text{as}\ \pc\rightarrow0,
\end{align}
where we have used $\pcr{0}=0$ which follows from the differentiation of the equation \eqref{ss8.1a}. Consequently,
\begin{align}
0\leq\|\pc^2\pcszmd{\pc}_{xxxx}\|=&\bigg\|\pc^2\bigg\{\bigg[\pcszmd{\pc}_{xx}+\frac{(\pcszmd{\pc}_{x})^2}{2}\bigg]_{xx}-\bigg[\frac{(\pcszmd{\pc}_{x})^2}{2}\bigg]_{xx}\bigg\}\bigg\|\notag\\
\leq&\pc^2\bigg\|\bigg[\pcszmd{\pc}_{xx}+\frac{(\pcszmd{\pc}_{x})^2}{2}\bigg]_{xx}\bigg\|+\pc^2\bigg\|\bigg[\frac{(\pcszmd{\pc}_{x})^2}{2}\bigg]_{xx}\bigg\|\notag\\
=&2\pcr{\pc}+\pc^2\big\|\big(\pcszmd{\pc}_{xx}\big)^2+\pcszmd{\pc}_x\pcszmd{\pc}_{xxx}\big\|\notag\\
\leq&2\pcr{\pc}+\pc^2\Big(|\pcszmd{\pc}_{xx}|_0\|\pcszmd{\pc}_{xx}\|+|\pcszmd{\pc}_{x}|_0\|\pcszmd{\pc}_{xxx}\|\Big)\notag\\
\leq&2\pcr{\pc}+C\pc^2\Big(\|\pcszmd{\pc}_{xx}\|_1\|\pcszmd{\pc}_{xx}\|+\|\pcszmd{\pc}_{x}\|_1\|\pcszmd{\pc}_{xxx}\|\Big)\notag\\
\leq&2\pcr{\pc}+C\pc\rightarrow0,\quad\text{as}\ \pc\rightarrow0.\label{ss37.3}
\end{align}
From \eqref{wnd3+dws4}, \eqref{ss34.5}, \eqref{ss36.1} and \eqref{ss37.3}, we know that
\begin{equation}\label{ss38.1b}
\big\|\big(\pcpszmd_{xx},\pc\pcszmd{\pc}_{xxx},\pc^2\pcszmd{\pc}_{xxxx},\pcpswnd_{xxx},\pcpsdws_{xxxx}\big)\big\|\rightarrow0,\quad\text{as}\ \pc\rightarrow0.
\end{equation}
By using the exponential transformations $\pcsmd{\pc}=e^{\pcszmd{\pc}}$ and $\pcsmd{0}=e^{\pcszmd{0}}$ again, the strong convergence \eqref{ss38.1b} implies the convergence \eqref{ss39.1b}.
\end{proof}

\subsection{Non-stationary case}\label{Subsect.4.2}
In this subsection, we continue to discuss the semi-classical limit of the global solutions based on the existence and uniqueness results in Lemma \ref{lem2} and Theorem \ref{thm2}. We introduce the error variables as follows
\begin{equation}\label{131.1}
\pcpmd:=\pcmd{\pc}-\pcmd{0},\quad \pcpdl:=\pcdl{\pc}-\pcdl{0},\quad\pcpwnd:=\pcwnd{\pc}-\pcwnd{0},\quad\pcpdws:=\pcdws{\pc}-\pcdws{0}.
\end{equation}

Since the global solutions $(\pcmd{\pc},\pcdl{\pc},\pcwnd{\pc},\pcdws{\pc})$ and $(\pcmd{0},\pcdl{0},\pcwnd{0},\pcdws{0})$ satisfy the same initial and boundary conditions, the error variables $(\pcpmd,\pcpdl,\pcpwnd,\pcpdws)$ satisfy the following initial and boundary conditions
\begin{gather}
(\pcpmd,\pcpdl,\pcpwnd,\pcpdws)(0,x)=(0,0,0,0),\label{134.1}\\
(\pd{t}{k}\pcpmd,\pd{t}{k}\pcpwnd,\pd{t}{k}\pcpdws)(t,0)=(\pd{t}{k}\pcpmd,\pd{t}{k}\pcpwnd,\pd{t}{k}\pcpdws)(t,1)=(0,0,0),\quad k=0,1.\label{134.2}
\end{gather}
Moreover, subtracting \eqref{1dfhd} from \eqref{1dfqhd}, the error variables $(\pcpmd,\pcpdl,\pcpwnd,\pcpdws)$ also satisfy the equations
\begin{gather}
\pcpmd_t+\pcpdl_x=0,\label{136.1a}\\
\pcpdl_t+\pcpdl=\mathcal{H}_1(t,x)+\pc^2\pcmd{\pc}\Bigg[\frac{\big(\sqrt{\pcmd{\pc}}\big)_{xx}}{\sqrt{\pcmd{\pc}}}\Bigg]_x,\label{136.1b}\\
\pcmd{\pc}\pcpwnd_t-\frac{2}{3}\pcpwnd_{xx}+\pcmd{0}\pcpwnd=\mathcal{H}_2(t,x;\pc),\label{157.2}\\
\pcpdws_{xx}=\pcpmd,\label{136.1d}
\end{gather}
where
\begin{align}
\mathcal{H}_1(t,x):=&-\big(\pcpmd_x\pcwnd{\pc}+\pcmd{0}_x\pcpwnd\big)-\big(\pcpmd\pcwnd{\pc}_x+\pcmd{0}\pcpwnd_x\big)+\big(\pcpmd\pcdws{\pc}_x+\pcmd{0}\pcpdws_x\big)\notag\\
&+\Bigg[\bigg(\frac{\pcdl{\pc}}{\pcmd{\pc}}\bigg)^2\pcmd{\pc}_x-\bigg(\frac{\pcdl{0}}{\pcmd{0}}\bigg)^2\pcmd{0}_x\Bigg]-2\bigg(\frac{\pcdl{\pc}}{\pcmd{\pc}}\pcdl{\pc}_x-\frac{\pcdl{0}}{\pcmd{0}}\pcdl{0}_x\bigg),\label{136.1br}
\end{align}
and
\begin{align}
\mathcal{H}_2(t,x;\pc):=&-\pcwnd{0}_t\pcpmd-\big(\pcpdl\pcwnd{\pc}_x+\pcdl{0}\pcpwnd_x\big)-\frac{2}{3}\bigg[\pcmd{\pc}\pcwnd{\pc}\bigg(\frac{\pcdl{\pc}}{\pcmd{\pc}}\bigg)_x-\pcmd{0}\pcwnd{0}\bigg(\frac{\pcdl{0}}{\pcmd{0}}\bigg)_x\bigg]\notag\\
&+\frac{1}{3}\Bigg[\frac{\big(\pcdl{\pc}\big)^2}{\pcmd{\pc}}-\frac{\big(\pcdl{0}\big)^2}{\pcmd{0}}\Bigg]-\pcpmd\big(\pcwnd{\pc}-\theta_L\big)+\frac{\pc^2}{3}\Bigg[\pcmd{\pc}\Bigg(\frac{\pcdl{\pc}}{\pcmd{\pc}}\Bigg)_{xx}\Bigg]_x.\label{157.3}
\end{align}
Differentiating the equation \eqref{136.1b} with respect to $x$ and using the equation \eqref{136.1a}, we obtain the equation
\begin{align}\label{135.1}
\pcpmd_{tt}-\pcwnd{\pc}\pcpmd_{xx}+\pcpmd_t=&\pcpmd_x\pcwnd{\pc}_x+\big(\pcmd{0}_x\pcpwnd\big)_x+\big(\pcpmd\pcwnd{\pc}_x+\pcmd{0}\pcpwnd_x\big)_x\notag\\
&-\big(\pcpmd\pcdws{\pc}_x+\pcmd{0}\pcpdws_x\big)_x-\Bigg[\bigg(\frac{\pcdl{\pc}}{\pcmd{\pc}}\bigg)^2\pcmd{\pc}_x-\bigg(\frac{\pcdl{0}}{\pcmd{0}}\bigg)^2\pcmd{0}_x\Bigg]_x\notag\\
&+2\bigg(\frac{\pcdl{\pc}}{\pcmd{\pc}}\pcdl{\pc}_x-\frac{\pcdl{0}}{\pcmd{0}}\pcdl{0}_x\bigg)_x-\pc^2\Bigg\{\pcmd{\pc}\Bigg[\frac{\big(\sqrt{\pcmd{\pc}}\big)_{xx}}{\sqrt{\pcmd{\pc}}}\Bigg]_x\Bigg\}_x.
\end{align}

From the estimates \eqref{0de} and \eqref{de}, we can deduce the following estimates
\begin{gather}
\pcmd{0}(t,x),\ \pcwnd{0}(t,x),\ \pcS{0}(t,x):=S[\pcmd{0},\pcdl{0},\pcwnd{0}]\geq c>0,\notag\\
\|(\pcmd{0},\pcdl{0},\pcwnd{0},\pcdws{0})(t)\|_2+\|(\pcmd{0}_t,\pcdl{0}_t,\pcwnd{0}_t)(t)\|_1\leq C,\label{138.1}
\end{gather}
and
\begin{gather}
\pcmd{\pc}(t,x),\ \pcwnd{\pc}(t,x),\ \pcS{\pc}(t,x):=S[\pcmd{\pc},\pcdl{\pc},\pcwnd{\pc}]\geq c>0,\notag\\
\|(\pcmd{\pc},\pcdl{\pc},\pcwnd{\pc},\pcdws{\pc})(t)\|_2+\|(\pc\pd{x}{3}\pcmd{\pc},\pc\pd{x}{3}\pcdl{\pc},\pc^2\pd{x}{4}\pcmd{\pc})(t)\|+\|(\pcmd{\pc}_t,\pcdl{\pc}_t)(t)\|_1+\|\pcwnd{\pc}_t(t)\|\leq C,\label{138.2}
\end{gather}
where $c$ and $C$ are positive constants independent of $\pc$, $\delta$, $x$ and $t$.

\begin{proof}[\textbf{Proof of Theorem \ref{thm4}.}] 
Based on Theorem \ref{thm3}, Lemma \ref{lem2} and Theorem \ref{thm2}, it is easy to see that the assumption \eqref{sic} guarantees the coexistence of the quantum and limit global solutions with the same initial and boundary data. Thus, the above facts \eqref{134.1}$\sim$\eqref{138.2} are available now.

Multiplying the equation \eqref{136.1b} by $\pcpdl$ and integrating the resultant equality over the domain $\Omega$, we have
\begin{align}
\frac{d}{dt}\int_0^1\frac{1}{2}\big(\pcpdl\big)^2dx+\big\|\pcpdl(t)\big\|^2&=\int_0^1\mathcal{H}_1\pcpdl dx+\pc^2\int_0^1\pcmd{\pc}\Bigg[\frac{\big(\sqrt{\pcmd{\pc}}\big)_{xx}}{\sqrt{\pcmd{\pc}}}\Bigg]_x\pcpdl dx \notag\\
&\leq C\big\|\big(\pcpmd,\pcpdl,\pcpwnd\big)(t)\big\|_1^2+C\pc^2.\label{142.1}
\end{align}
In the derivation of the estimate \eqref{142.1}, we have used the Cauchy-Schwarz inequality, the elliptic estimate $\|\pcpdws(t)\|_2\leq C\|\pcpmd(t)\|$ and the following estimate
\begin{equation}
\|\mathcal{H}_1(t)\|\leq C\big\|\big(\pcpmd,\pcpmd_x,\pcpdl,\pcpdl_x,\pcpwnd,\pcpwnd_x,\pcpdws_x\big)(t)\big\|\label{140.1-141.1}
\end{equation}
to control the first term on the right-side of this equality, and we have also used the integration by parts, the boundary condition \eqref{bc-b} and the estimates \eqref{138.1}$\sim$\eqref{138.2} to bound the last term on the right-side as follows
\begin{align}
&\pc^2\int_0^1\pcmd{\pc}\Bigg[\frac{\big(\sqrt{\pcmd{\pc}}\big)_{xx}}{\sqrt{\pcmd{\pc}}}\Bigg]_x\pcpdl dx\notag\\
=&-\pc^2\int_0^1\frac{\big(\sqrt{\pcmd{\pc}}\big)_{xx}}{\sqrt{\pcmd{\pc}}}\big(\pcmd{\pc}\pcpdl\big)_xdx\notag\\
\leq&\pc^2\Bigg\|\Bigg[\frac{\big(\sqrt{\pcmd{\pc}}\big)_{xx}}{\sqrt{\pcmd{\pc}}}\Bigg](t)\Bigg\|\Big\|\big(\pcmd{\pc}\pcpdl\big)_x(t)\Big\|\notag\\
\leq&\pc^2\Bigg|\frac{1}{\sqrt{\pcmd{\pc}}}(t)\Bigg|_0\Big\|\big(\sqrt{\pcmd{\pc}}\big)_{xx}(t)\Big\|\Big(\big|\pcmd{\pc}_x(t)\big|_0\big\|\pcpdl(t)\big\|+\big|\pcmd{\pc}(t)\big|_0\big\|\pcpdl_x(t)\big\|\Big)\notag\\
\leq&C\pc^2.
\end{align}

Multiplying the equation \eqref{135.1} by $\pcpmd_t$ and integrating the resultant equality over the domain $\Omega$, we obtain
\allowdisplaybreaks
\begin{align}
&\frac{d}{dt}\int_0^1\frac{1}{2}\big(\pcpdl_x\big)^2dx-\int_0^1\pcwnd{\pc}\pcpmd_{xx}\pcpmd_tdx+\big\|\pcpdl_x(t)\big\|^2\notag\\
=&\int_0^1\Big[\pcpmd_x\pcwnd{\pc}_x+\big(\pcmd{0}_x\pcpwnd\big)_x-\big(\pcpmd\pcdws{\pc}_x+\pcmd{0}\pcpdws_x\big)_x\Big]\pcpmd_tdx\notag\\
&+\int_0^1\big(\pcpmd\pcwnd{\pc}_x+\pcmd{0}\pcpwnd_x\big)_x\pcpmd_tdx-\int_0^1\Bigg[\bigg(\frac{\pcdl{\pc}}{\pcmd{\pc}}\bigg)^2\pcmd{\pc}_x-\bigg(\frac{\pcdl{0}}{\pcmd{0}}\bigg)^2\pcmd{0}_x\Bigg]_x\pcpmd_tdx\notag\\
&+2\int_0^1\bigg(\frac{\pcdl{\pc}}{\pcmd{\pc}}\pcdl{\pc}_x-\frac{\pcdl{0}}{\pcmd{0}}\pcdl{0}_x\bigg)_x\pcpmd_tdx-\int_0^1\pc^2\Bigg\{\pcmd{\pc}\Bigg[\frac{\big(\sqrt{\pcmd{\pc}}\big)_{xx}}{\sqrt{\pcmd{\pc}}}\Bigg]_x\Bigg\}_x\pcpmd_tdx\notag\\
=&\Lambda_1+\Lambda_2+\Lambda_3+\Lambda_4+\Lambda_5,\label{142.2}
\end{align}
where we have used the equation \eqref{136.1a}. Next, we respectively estimate the second term on the left-side of the equation \eqref{142.2} and the integrals $\Lambda_l$, $l=1,2,\cdots,5$ on the right-side of \eqref{142.2} by using the integration by parts, the Sobolev inequality, the H\"older inequality, the Cauchy-Schwarz inequality, the equation \eqref{136.1a}, the boundary condition \eqref{134.2}, the equation \eqref{157.2} and the estimates \eqref{138.1}$\sim$\eqref{138.2} as follows
\begin{align}
-\int_0^1\pcwnd{\pc}\pcpmd_{xx}\pcpmd_tdx=&\int_0^1\big(\pcwnd{\pc}\pcpmd_t\big)_x\pcpmd_xdx\notag\\
=&\frac{d}{dt}\int_0^1\frac{1}{2}\pcwnd{0}\big(\pcpmd_x\big)^2dx-\int_0^1\frac{1}{2}\pcwnd{0}_t\big(\pcpmd_x\big)^2dx+\int_0^1\pcpwnd\pcpmd_{tx}\pcpmd_xdx\notag\\
&-\int_0^1\pcwnd{\pc}_x\pcpdl_x\pcpmd_xdx\notag\\
\geq&\frac{d}{dt}\int_0^1\frac{1}{2}\pcwnd{0}\big(\pcpmd_x\big)^2dx\notag\\
&-C\Big(\big|\pcwnd{0}_t\big|_0\big\|\pcpmd_x\big\|^2+\big|\pcpwnd\big|_0\big\|\pcpmd_{tx}\big\|\big\|\pcpmd_x\big\|+\big|\pcwnd{\pc}_x\big|_0\big\|\pcpdl_x\big\|\big\|\pcpmd_{x}\big\|\Big)\notag\\
\geq&\frac{d}{dt}\int_0^1\frac{1}{2}\pcwnd{0}\big(\pcpmd_x\big)^2dx-C\Big(\big\|\pcpmd_x\big\|^2+\big\|\pcpwnd\big\|_1\big\|\pcpmd_x\big\|+\big\|\pcpdl_x\big\|\big\|\pcpmd_{x}\big\|\Big)\notag\\
\geq&\frac{d}{dt}\int_0^1\frac{1}{2}\pcwnd{0}\big(\pcpmd_x\big)^2dx-C\big\|\big(\pcpmd_x,\pcpdl_x,\pcpwnd,\pcpwnd_x\big)(t)\big\|^2,\label{143.2}
\end{align}
and
\begin{equation}\label{145.1-2+152.2}
\Lambda_1\leq C\big\|\big(\pcpmd,\pcpmd_x,\pcpdl_x,\pcpwnd,\pcpwnd_x\big)(t)\big\|^2,
\end{equation}
and
\begin{align}
\Lambda_2=&-\int_0^1\big(\pcpmd_x\pcwnd{\pc}_x+\pcpmd\pcwnd{\pc}_{xx}+\pcmd{0}_x\pcpwnd_x+\pcmd{0}\pcpwnd_{xx}\big)\pcpdl_xdx\notag\\
\leq&-\int_0^1\pcmd{0}\pcpwnd_{xx}\pcpdl_xdx+C\big\|\big(\pcpmd,\pcpmd_x,\pcpdl_x,\pcpwnd_x\big)(t)\big\|^2\notag\\
=&\frac{3}{2}\int_0^1\pcmd{0}\big[\mathcal{H}_2(t,x;\pc)-\pcmd{\pc}\pcpwnd_t-\pcmd{0}\pcpwnd\big]\pcpdl_xdx+C\big\|\big(\pcpmd,\pcpmd_x,\pcpdl_x,\pcpwnd_x\big)(t)\big\|^2\notag\\
\leq&\mu\big\|\pcpwnd_t(t)\big\|^2+C_\mu\big\|\pcpdl_x(t)\big\|^2+C\|\mathcal{H}_2(t;\pc)\|^2+C\big\|\big(\pcpmd,\pcpmd_x,\pcpdl_x,\pcpwnd,\pcpwnd_x\big)(t)\big\|^2\notag\\
\leq&\mu\big\|\pcpwnd_t(t)\big\|^2+C_\mu\big\|\big(\pcpmd,\pcpdl,\pcpwnd\big)(t)\big\|_1^2+C\pc^2,\label{152.1}
\end{align}
where we have used the estimate
\begin{align}
&\|\mathcal{H}_2(t;\pc)\|\notag\\
\leq&C\big\|\big(\pcpmd,\pcpdl,\pcpwnd\big)(t)\big\|_1+C\Bigg\|\pc^2\Bigg[\pcmd{\pc}\Bigg(\frac{\pcdl{\pc}}{\pcmd{\pc}}\Bigg)_{xx}\Bigg]_x\Bigg\|\notag\\
=&C\big\|\big(\pcpmd,\pcpdl,\pcpwnd\big)(t)\big\|_1\notag\\
&+C\Bigg\|\pc^2\Bigg[\pcdl{\pc}_{xxx}-2\frac{\pcmd{\pc}_x}{\pcmd{\pc}}\pcdl{\pc}_{xx}+4\bigg(\frac{\pcmd{\pc}_x}{\pcmd{\pc}}\bigg)^2\pcdl{\pc}_x-3\frac{\pcdl{\pc}_x}{\pcmd{\pc}}\pcmd{\pc}_{xx}-4\bigg(\frac{\pcmd{\pc}_x}{\pcmd{\pc}}\bigg)^3\pcdl{\pc}+5\frac{\pcdl{\pc}\pcmd{\pc}_x}{(\pcmd{\pc})^2}\pcmd{\pc}_{xx}-\frac{\pcdl{\pc}}{\pcmd{\pc}}\pcmd{\pc}_{xxx}\Bigg]\Bigg\|\notag\\
\leq&C\big\|\big(\pcpmd,\pcpdl,\pcpwnd\big)(t)\big\|_1+C\pc\big\|\big(\pc\pcdl{\pc}_{xxx},\pcdl{\pc}_{xx},\pcdl{\pc}_x,\pcmd{\pc}_{xx},\pcdl{\pc},\pcmd{\pc}_{xx},\pc\pcmd{\pc}_{xxx}\big)(t)\big\|\notag\\
\leq&C\big\|\big(\pcpmd,\pcpdl,\pcpwnd\big)(t)\big\|_1+C\pc\label{148.3-151.1} 
\end{align}
in the derivation of \eqref{152.1}. Next, we continue to estimate
\begin{align}
\Lambda_3\leq&-\int_0^1\bigg(\frac{\pcdl{0}}{\pcmd{0}}\bigg)^2\pcpmd_{xx}\pcpmd_tdx+C\big\|\big(\pcpmd,\pcpdl\big)(t)\big\|_1^2\notag\\
=&\int_0^1\bigg[\bigg(\frac{\pcdl{0}}{\pcmd{0}}\bigg)^2\pcpmd_t\bigg]_x\pcpmd_{x}dx+C\big\|\big(\pcpmd,\pcpdl\big)(t)\big\|_1^2\notag\\
=&\frac{d}{dt}\int_0^1\frac{1}{2}\bigg(\frac{\pcdl{0}}{\pcmd{0}}\bigg)^2\big(\pcpmd_x\big)^2dx-\int_0^1\frac{1}{2}\bigg[\bigg(\frac{\pcdl{0}}{\pcmd{0}}\bigg)^2\bigg]_t\big(\pcpmd_x\big)^2dx-\int_0^1\bigg[\bigg(\frac{\pcdl{0}}{\pcmd{0}}\bigg)^2\bigg]_x\pcpdl_x\pcpmd_xdx\notag\\
&+C\big\|\big(\pcpmd,\pcpdl\big)(t)\big\|_1^2\notag\\
\leq&\frac{d}{dt}\int_0^1\frac{1}{2}\bigg(\frac{\pcdl{0}}{\pcmd{0}}\bigg)^2\big(\pcpmd_x\big)^2dx+C\big\|\big(\pcpmd,\pcpdl\big)(t)\big\|_1^2,\label{153.1}
\end{align}
and
\begin{align}
\Lambda_4\leq&2\int_0^1\frac{\pcdl{0}}{\pcmd{0}}\pcpdl_{xx}\pcpmd_tdx+C\big\|\big(\pcpmd,\pcpdl\big)(t)\big\|_1^2\notag\\
=&-2\int_0^1\frac{\pcdl{0}}{\pcmd{0}}\pcpmd_{tx}\pcpmd_tdx+C\big\|\big(\pcpmd,\pcpdl\big)(t)\big\|_1^2\notag\\
=&-\int_0^1\frac{\pcdl{0}}{\pcmd{0}}\big[\big(\pcpmd_t\big)^2\big]_xdx+C\big\|\big(\pcpmd,\pcpdl\big)(t)\big\|_1^2\notag\\
=&\int_0^1\bigg(\frac{\pcdl{0}}{\pcmd{0}}\bigg)_x\big(\pcpdl_x\big)^2dx+C\big\|\big(\pcpmd,\pcpdl\big)(t)\big\|_1^2\notag\\
\leq&C\big\|\big(\pcpmd,\pcpdl\big)(t)\big\|_1^2,\label{gs154.1}
\end{align}
and
\begin{align}
\Lambda_5=&\int_0^1\pc^2\pcmd{\pc}\Bigg[\frac{\big(\sqrt{\pcmd{\pc}}\big)_{xx}}{\sqrt{\pcmd{\pc}}}\Bigg]_x\pcpmd_{tx}dx\notag\\
=&\int_0^1\pc^2\Big[\sqrt{\pcmd{\pc}}\big(\sqrt{\pcmd{\pc}}\big)_{xxx}-\big(\sqrt{\pcmd{\pc}}\big)_{x}\big(\sqrt{\pcmd{\pc}}\big)_{xx}\Big]\pcpmd_{tx}dx\notag\\
\leq&C\pc^2\Big(\big\|\big(\sqrt{\pcmd{\pc}}\big)_{xxx}(t)\big\|+\big\|\big(\sqrt{\pcmd{\pc}}\big)_{xx}(t)\big\|\Big)\big\|\pcpmd_{tx}(t)\big\|\notag\\
\leq&C\pc.\label{155.1}
\end{align}
Substituting \eqref{143.2}$\sim$\eqref{152.1} and \eqref{153.1}$\sim$\eqref{155.1} into \eqref{142.2}, we have
\begin{multline}\label{156.2}
\frac{d}{dt}\int_0^1\bigg[\frac{1}{2}\pcS{0}\big(\pcpmd_x\big)^2+\frac{1}{2}\big(\pcpdl_x\big)^2\bigg]dx+\big\|\pcpdl_x(t)\big\|^2\\
\leq\mu\big\|\pcpwnd_t(t)\big\|^2+C_\mu\big\|\big(\pcpmd,\pcpdl,\pcpwnd\big)(t)\big\|_1^2+C\pc.
\end{multline}

Multiplying the equation \eqref{157.2} by $\pcpwnd_t$ and integrating the resultant equality over the domain $\Omega$, we get
\begin{equation}\label{161.1}
\int_0^1\pcmd{\pc}\big(\pcpwnd_t\big)^2dx-\int_0^1\frac{2}{3}\pcpwnd_{xx}\pcpwnd_tdx+\int_0^1\pcmd{0}\pcpwnd\pcpwnd_tdx=\int_0^1\mathcal{H}_2(t,x;\pc)\pcpwnd_tdx.
\end{equation}
Similarly, we can use the standard computations to deal with the each term in \eqref{161.1} as follows
\begin{equation}\label{161.2}
\int_0^1\pcmd{\pc}\big(\pcpwnd_t\big)^2dx\geq 2c\big\|\pcpwnd_t(t)\big\|^2,
\end{equation}
and
\begin{equation}\label{161.3}
-\int_0^1\frac{2}{3}\pcpwnd_{xx}\pcpwnd_tdx=\int_0^1\frac{2}{3}\pcpwnd_{x}\pcpwnd_{xt}dx=\frac{d}{dt}\int_0^1\frac{1}{3}\big(\pcpwnd_{x}\big)^2dx,
\end{equation}
and
\begin{align}
\int_0^1\pcmd{0}\pcpwnd\pcpwnd_tdx=&\frac{d}{dt}\int_0^1\frac{1}{2}\pcmd{0}\big(\pcpwnd\big)^2dx-\int_0^1\frac{1}{2}\pcmd{0}_t\big(\pcpwnd\big)^2dx\notag\\
\geq&\frac{d}{dt}\int_0^1\frac{1}{2}\pcmd{0}\big(\pcpwnd\big)^2dx-C\big\|\pcpwnd(t)\big\|^2,\label{161.4}
\end{align}
and
\begin{align}
\int_0^1\mathcal{H}_2(t,x;\pc)\pcpwnd_tdx\leq&c\big\|\pcpwnd_t(t)\big\|^2+C\|\mathcal{H}_2(t;\pc)\|^2\notag\\
\leq&c\big\|\pcpwnd_t(t)\big\|^2+C\big\|\big(\pcpmd,\pcpdl,\pcpwnd\big)(t)\big\|_1^2+C\pc^2,\label{161.5}
\end{align}
where we have used the estimate \eqref{148.3-151.1} again in the last inequality of \eqref{161.5}. Substituting \eqref{161.2}$\sim$\eqref{161.5} into \eqref{161.1}, we have
\begin{equation}\label{161.6}
\frac{d}{dt}\int_0^1\bigg[\frac{1}{2}\pcmd{0}\big(\pcpwnd\big)^2+\frac{1}{3}\big(\pcpwnd_{x}\big)^2\bigg]dx+c\big\|\pcpwnd_t(t)\big\|^2\leq C\big\|\big(\pcpmd,\pcpdl,\pcpwnd\big)(t)\big\|_1^2+C\pc^2.
\end{equation}

Adding \eqref{142.1}, \eqref{156.2} and \eqref{161.6} up, and letting $\mu$ small enough, we obtain
\begin{equation}\label{163.1}
\frac{d}{dt}\mathcal{E}^\pc(t)+\underbrace{\big\|\pcpdl(t)\big\|^2+\big\|\pcpdl_x(t)\big\|^2+\frac{c}{2}\big\|\pcpwnd_t(t)\big\|^2}_{\geq0}\leq C\big\|\big(\pcpmd,\pcpdl,\pcpwnd\big)(t)\big\|_1^2+C\pc,
\end{equation}
where
\begin{equation}\label{163.2}
\mathcal{E}^\pc(t):=\int_0^1\bigg[\frac{1}{2}\pcS{0}\big(\pcpmd_x\big)^2+\frac{1}{2}\big(\pcpdl\big)^2+\frac{1}{2}\big(\pcpdl_x\big)^2+\frac{1}{2}\pcmd{0}\big(\pcpwnd\big)^2+\frac{1}{3}\big(\pcpwnd_{x}\big)^2\bigg]dx.
\end{equation}
From the estimate \eqref{138.1}, it is easy to check the following equivalent relation
\begin{equation}\label{163.3}
c\big\|\big(\pcpmd,\pcpdl,\pcpwnd\big)(t)\big\|_1^2\leq\mathcal{E}^\pc(t)\leq C\big\|\big(\pcpmd,\pcpdl,\pcpwnd\big)(t)\big\|_1^2,\quad\forall t\in[0,\infty)
\end{equation}
by using the Poincar\'e inequality. Therefore, the inequality \eqref{163.1} implies
\begin{equation}\label{164.1}
\frac{d}{dt}\mathcal{E}^\pc(t)\leq 2\gamma_3\mathcal{E}^\pc(t)+C\pc,\quad\forall t\in[0,\infty),
\end{equation}
where the positive constant $\gamma_3$ is independent of $\pc$ and $t$. Applying the Gronwall inequality to \eqref{164.1}, we have
\begin{equation}\label{164.3}
\big\|\big(\pcpmd,\pcpdl,\pcpwnd\big)(t)\big\|_1\leq Ce^{\gamma_3 t}\pc^{1/2},\quad\forall t\in[0,\infty)
\end{equation}
Combining \eqref{164.3} with the elliptic estimate $\|\pcpdws(t)\|_3\leq C\|\pcpmd(t)\|_1$, we get the desired estimate \eqref{gs164.5}.

Finally, for fixed $\pc\in(0,\delta_6)$, we define a time
\begin{equation}\label{165.1}
T_\pc:=-\frac{\ln\pc}{4\gamma_3}>0.
\end{equation}
For $t\leq T_\pc$, the estimate \eqref{gs164.5} yields that
\begin{equation}\label{165.2}
\big\|\big(\pcpmd,\pcpdl,\pcpwnd\big)(t)\big\|_1+\big\|\pcpdws(t)\big\|_3\leq Ce^{\gamma_3 T_\pc}\pc^{1/2}=C\pc^{1/4}.
\end{equation}
For $t\geq T_\pc$, using the estimates \eqref{de}, \eqref{ss39.1a} and \eqref{0de}, we obtain
\begin{align}
&\big\|\big(\pcpmd,\pcpdl,\pcpwnd\big)(t)\big\|_1+\big\|\pcpdws(t)\big\|_3\notag\\
\leq&\|(\pcmd{\pc}-\pcsmd{\pc},\pcdl{\pc}-\pcsdl{\pc},\pcwnd{\pc}-\pcswnd{\pc})(t)\|_1+\|(\pcdws{\pc}-\pcsdws{\pc})(t)\|_3\notag\\
&+\|(\pcsmd{\pc}-\pcsmd{0},\pcsdl{\pc}-\pcsdl{0},\pcswnd{\pc}-\pcswnd{0})\|_1+\|(\pcsdws{\pc}-\pcsdws{0})\|_3\notag\\
&+\|(\pcmd{0}-\pcsmd{0},\pcdl{0}-\pcsdl{0},\pcwnd{0}-\pcswnd{0})(t)\|_1+\|(\pcdws{0}-\pcsdws{0})(t)\|_3\notag\\
\leq&C\big(e^{-\gamma_2T_\pc}+\pc+e^{-\gamma_1T_\pc}\big)\notag\\
=&C\Big(\pc^{\frac{\gamma_2}{4\gamma_3}}+\pc+\pc^{\frac{\gamma_1}{4\gamma_3}}\Big)\notag\\
\leq&C\pc^{\gamma_4},\label{165.3}
\end{align}
where
\begin{equation}
\gamma_4:=\min\bigg\{\frac{\gamma_1}{4\gamma_3},\frac{\gamma_2}{4\gamma_3},\frac{1}{4}\bigg\}>0.
\end{equation}
Owing to \eqref{165.2} and \eqref{165.3}, we have
\begin{equation}\label{166.1}
\big\|\big(\pcpmd,\pcpdl,\pcpwnd\big)(t)\big\|_1+\big\|\pcpdws(t)\big\|_3\leq C\pc^{\gamma_4},\quad\forall t\in[0,\infty).
\end{equation}
Note that the right-side of \eqref{166.1} is independent of $t$, this immediately implies the estimate \eqref{gs1}.
\end{proof}

\section{Appendix}\label{A}
In this appendix, we study the unique solvability of the linear IBVP \eqref{a3.1}$\sim$\eqref{a3.3}. 

Firstly, the parabolic equation \eqref{a3.1-3} with the initial condition $\hat{\wnd}(0,x)=\wnd_0(x)$ and the boundary condition \eqref{a3.3-3} has a unique solution $\hat{\wnd}\in\mathfrak{Y}_2([0,T])\cap H^1(0,T;H^1(\Omega))$ for given function $(\rmd,\dl)\in\big[\mathfrak{Y}_4([0,T])\cap H^2(0,T;H^1(\Omega))\big]\times\big[\mathfrak{Y}_3([0,T])\cap H^2(0,T;L^2(\Omega))\big]$. This fact is proved by the Galerkin method (see \cite{T88,Z90} for example). 

Next, we only need to show the unique solvability of the following linear IBVP for given function $(\rmd,\dl,\wnd,\hat{\wnd})$, namely,
\begin{subequations}\label{a7.1}
\begin{numcases}{}
2\rmd\hat{\rmd}_t+\hat{\dl}_x=0, \label{a7.1-1}\\
\hat{\dl}_t+2S[\rmd^2,\dl,\wnd]\rmd\hat{\rmd}_x+\frac{2\dl}{\rmd^2}\hat{\dl}_x+\rmd^2\hat{\wnd}_x-\pc^2\rmd^2\Bigg(\frac{\hat{\rmd}_{xx}}{\rmd}\Bigg)_x=\rmd^2\dws_x-\dl, \label{a7.1-2}\\
\dws:=\Phi[\rmd^2], \qquad\forall t>0,\ \forall x\in\Omega:=(0,1),\label{a7.1-3}
\end{numcases}
\end{subequations}
with the initial condition
\begin{equation}\label{a7.2}
(\hat{\rmd},\hat{\dl})(0,x)=(\rmd_0,\dl_0)(x),
\end{equation}
and the boundary conditions
\begin{subequations}\label{a7.3}
\begin{gather}
\hat{\rmd}(t,0)=\rmd_{l},\qquad \hat{\rmd}(t,1)=\rmd_{r},\label{a7.3-1}\\
\hat{\rmd}_{xx}(t,0)=\hat{\rmd}_{xx}(t,1)=0.\label{a7.3-2}
\end{gather}
\end{subequations}

To this end, performing the procedure $\pd{x}{}\eqref{a7.1-2}/(-2\rmd)$ and inserting the transformation $U(t,x):=\hat{\rmd}(t,x)-\bar{w}(x)$ into the resultant system, where $\bar{w}(x):=w_l(1-x)+w_rx$, we can equivalently reduce the IBVP \eqref{a7.1}$\sim$\eqref{a7.3} to the IBVP of a fourth order wave equation satisfied by $U$,
\begin{gather}
U_{tt}+b_0\pd{x}{}U_t+b_1U_t+b_2U_x+b_3U_{xx}+a\pd{x}{4}U=f,\label{a8.1}\\
U(0,x)=\rmd_0(x)-\bar{w}(x),\quad U_t(0,x)=-\frac{\dl_{0x}}{2\rmd_0}(x),\label{a8.2}\\
U(t,0)=U(t,1)=U_{xx}(t,0)=U_{xx}(t,1)=0,\label{a8.3}
\end{gather}
where
\begin{align}
&b_0:=\frac{2j}{w^2},\qquad b_1:=\frac{1}{w}\bigg[\bigg(\frac{2j}{w}\bigg)_x+w_t\bigg],\qquad b_2:=-\frac{1}{w}\bigg[\bigg(\wnd-\frac{j^2}{w^4}\bigg)w\bigg]_x,\notag\\
&b_3:=-\bigg[\bigg(\wnd-\frac{j^2}{w^4}\bigg)+\frac{\pc^2}{2}\frac{w_{xx}}{w}\bigg],\qquad a:=\frac{\pc^2}{2},\notag\\
&f:=-\frac{1}{2w}\big(w^2\dws_x-j-w^2\hat{\wnd}_x\big)_x+\frac{1}{w}\bigg[\bigg(\wnd-\frac{j^2}{w^4}\bigg)w\bigg]_x\bar{w}_x.\label{a8.4}
\end{align}
Applying the Lemma A.1 (\cite{NS08}, P870) to the linear IBVP \eqref{a8.1}$\sim$\eqref{a8.3}, we see that this problem has a unique solution $U\in\mathfrak{Y}_4([0,T])$.

We proceed to construct the solution $(\hat{\rmd},\hat{\dl})$ to the IBVP \eqref{a7.1}$\sim$\eqref{a7.3} from $U$ as follows,
\begin{subequations}
\begin{gather}
\hat{\rmd}(t,x):=U(t,x)+\bar{w}(x),\label{a9.0}\\
\hat{j}(t,x):=-\int_0^x2w\hat{w}_t(t,y)dy+\hat{j}(t,0),\label{a9.1}\\
\hat{j}(t,0):=\int_0^t\bigg[-2\bigg(\wnd-\frac{j^2}{w^4}\bigg)w\hat{w}_x+\frac{4j}{w}\hat{w}_t-w^2\hat{\wnd}_x\notag\\
\qquad\qquad\qquad\qquad\qquad\qquad\qquad\qquad\qquad+\pc^2w^2\bigg(\frac{\hat{w}_{xx}}{w}\bigg)_x+w^2\dws_x-j\bigg](\tau,0)d\tau+j_0(0).\label{a9.2}\notag
\end{gather}
\end{subequations}
By the standard argument (see \cite{NS08,NS09} for example), we can easily see that the function $(\hat{\rmd},\hat{\dl})\in\big[\mathfrak{Y}_4([0,T])\cap H^2(0,T;H^1(\Omega))\big]\times\big[\mathfrak{Y}_3([0,T])\cap H^2(0,T;L^2(\Omega))\big]$ is a desired solution to the linear IBVP \eqref{a7.1}$\sim$\eqref{a7.3}.

\section*{Acknowledgements}
The research of KJZ was supported in part by NSFC (No.11371082) and the Fundamental Research Funds for the Central Universities (No.111065201).


\end{document}